\title{Period differential equations  for the   families of $K3$ surfaces with $2$ parameters derived from  the reflexive polytopes}
\author{Atsuhira Nagano}
\documentclass{article}
\usepackage{mathrsfs,bm,graphics,amsfonts,amsthm,amssymb,amsmath,amscd,booktabs}
\usepackage[pdflatex]{graphicx}
\def\bigzerou{\smash{\lower1.7ex\hbox{\b 0}}}

\newtheorem{thm}{Theorem}[section]
\newtheorem{df}{Definition}[section]
\newtheorem{lem}{Lemma}[section]
\newtheorem{prop}{Proposition}[section]
\newtheorem{rem}{Remark}[section]

\newtheorem{cor}{Corollary}[section]

\def\comment#1{{ }}
\makeatletter
  
      \@addtoreset{equation}{section}
        
      \makeatletter

\begin{document}
\maketitle

\begin{abstract}
In this paper, we study the period mappings for the  families of $K3$ surfaces derived from the $3$-dimensional $5$-verticed reflexive polytopes. 
We determine the lattice structures, the  period differential equations and the projective monodromy groups.
 Moreover, we show that one of our period differential equations coincides with
  the unifomizing differential equation of the Hilbert modular orbifold for the field $\mathbb{Q}(\sqrt{5})$.
\end{abstract}

\footnote[0]{Running head: $K3$ surfaces with $2$ parameters}
\setlength{\baselineskip}{14 pt}

\section*{Introduction}

A $K3$ surface $S$ is characterized by the condition $K_S=0$ and simply connectedness. 
It means that  a $K3$ surface is a $2$-dimensional Calabi-Yau manifold.
Batyrev \cite{Batyrev} introduced the notion of the reflexive polytope for the study of  the Calabi-Yau manifold.

In this article, we use the $3$-dimensional reflexive polytopes with at most terminal singularities.
These polytopes with $5$-vertices are given as 
\begin{align}
&P_0= 
\begin{pmatrix}
1&0&0&0&-1\\
0&1&0&0&-1\\
0&0&1&-1&-2\\
\end{pmatrix}, \label{polytope P}\\
P_1=
\begin{pmatrix}
1&0&0&-1&0\\
0&1&0&0&-1\\
0&0&1&-1&-1\\
\end{pmatrix},
\quad
&P_2=
\begin{pmatrix}
1&0&0&0&-1\\
0&1&0&-1&-1\\
0&0&1&-1&-1\\
\end{pmatrix},
\quad
P_3=
\begin{pmatrix}
1&0&0&-1&0\\
0&1&0&-1&0\\
0&0&1&0&-1\\
\end{pmatrix}, \label{polytope 1,2,3}\\
&P_4= 
\begin{pmatrix}
1&0&0&0&-1\\
0&1&0&0&-1\\
0&0&1&-1&-1\\
\end{pmatrix}, \label{polytope Ishige}
\end{align}
where the column vectors correspond to the coordinates of the vertices (see   \cite{Otsuka} or \cite{Kreuzer}).

Each polytope $P_j$  defines a $2$-parameter family $\mathcal{F}_j=\{S_j(\lambda,\mu)\}$ of $K3$ surfaces 
 given by the affine equation  
\begin{align}
\begin{cases}\label{equationsF_j}
S_0(\lambda,\mu)&:F_0(x,y,z) = xyz^2(x+y+z+1)+\lambda xyz +\mu=0,\\
S_1(\lambda,\mu)&:F_1(x,y,z)= xyz(x+y+z+1)+\lambda x +\mu y=0,\\
S_2(\lambda,\mu)&: F_2(x,y,z)=xyz(x+y+z+1)+\lambda x +\mu=0,\\
S_3(\lambda,\mu)&: F_3(x,y,z)=xyz(x+y+z+1)+\lambda z +\mu xy=0,\\
S_4(\lambda,\mu)&: F_4(x,y,z)=xyz(x+y+z+1)+\lambda xy +\mu =0.
\end{cases}
\end{align}

Recently, Ishige \cite{Ishige} has made a research on the family $\mathcal{F}_4$ derived from the polytope $P_4$ in (\ref{polytope Ishige}).
He made a computer aided approximation of a generator of the monodromy group of his differential equation. 
There,  he  noticed that his monodromy group is isomorphic to the extended Hilbert modular group for $\mathbb{Q}(\sqrt{2})$.

Inspired by Ishige's research,
we study the other families  $\mathcal{F}_j$ $(j=0,1,2,3)$.
Among them, the family $\mathcal{F}_0$ is especially interesting, for this is related to the Hilbert modular orbifold for the field $\mathbb{Q}(\sqrt{5})$.
In the classical  theory of the elliptic functions and the  Gauss hypergeometric differential equations,  the differential equation
$$
\lambda (1-\lambda) \frac{d^2 \eta}{d \lambda ^2} + (1-2\lambda) \frac{d \eta}{d \lambda} -\frac{1}{4} \eta=0
$$
has two remarkable properties.
This is a period differential equation for the family of elliptic curves 
$\{S(\lambda)\}$, with
$$
S(\lambda): y^2 =x(x-1)(x-\lambda ).
$$
Also, this  is the uniformizing differential equation of the orbifold $\mathbb{H}/\Gamma(2)$, where $\mathbb{H}$ is the upper half plane.
We can find an analogous story on our family $\mathcal{F}_0$: The period differential equation for $\mathcal{F}_0$ is the uniformizing differential equation of the Hilbert modular orbifold  for $\mathbb{Q}(\sqrt{5})$.

This paper is organized  as the following.

In Section 1, we give explicit defining equations for the families $\mathcal{F}_j$ $(j=0,1,2,3)$ of $K3$ surfaces and  elliptic fibrations for them.

 In Section 2,   we construct the period mappings  for our families. They are  multivalued analytic mappings from  Zariski open domains  in $\mathbb{P}^2(\mathbb{C})$ to $2$-dimensional domains of type $IV$.
 Then, we determine the Picard numbers for generic members of $\mathcal{F}_j$ $(j=0,1,2,3)$ (Theorem \ref{thm:P4Picard} and Theorem \ref{18theorem}).
 
In Section 3, we determine   the N\'{e}ron-Severi lattice 
 and  the transcendental lattice $A_j$ $(j=0,1,2,3)$ for  a generic member of $\mathcal{F}_j$ $(j=0,1,2,3)$ (Theorem \ref{latticeThm}).

In Section 4, we obtain the period differential equations for  $\mathcal{F}_j$ $(j=0,1,2,3)$ (Theorem \ref{thm:periodDE!}). 
They have the $4$-dimensional space of solutions. 

In Section 5, applying the Torelli theorem and the lattice theory, we prove that the projective monodromy group of the period differential equation   for $\mathcal{F}_j$ $(j=0,1,2,3)$ is equal to the  group $PO^+(A_j,\mathbb{Z})$ $(j=0,1,2,3)$ for the transcendental lattice $A_j$ $(j=0,1,2,3)$ (Theorem \ref{thm:monodromy} and Theorem \ref{monodromytheorem}).

In particular, the projective monodromy group of the period differential equation for $\mathcal{F}_0$ is isomorphic to the group $\langle PSL(2,\mathcal{O}),\tau \rangle$, where $\mathcal{O}$ is the ring of integers in $\mathbb{Q}(\sqrt{5})$, $PSL(2,\mathcal{O})$ is the Hilbert modular group and $\tau$ is the involution on $\mathbb{H}\times \mathbb{H}$ given by the coordinate exchange.

In Section 6, we show that our period differential equation for $\mathcal{F}_0$ coincides with the uniformizing differential equation of the Hilbert modular orbifold $(\mathbb{H}\times\mathbb{H})/\langle PSL(2,\mathcal{O}),\tau\rangle$ given by Sato \cite{Sato} under an explicit birational transformation (Theorem \ref{periodUDEThm}).

\section{A family of $K3$ surfaces and elliptic fibration}

A $3$-dimensional reflexive polytope with at most terminal singularities  is defined by the intersection of several half spaces
$$
a_j x + b_j y + c_j z \leq 1, \hspace{2mm} (a_j,b_j,c_j)\in\mathbb{Z}^3 \hspace{2mm} (j=1,\cdots,s) 
$$
in $\mathbb{R}^3$ with the conditions 
\par
{\rm (a)} every vertex is a lattice point,
\par
{\rm (b)} the origin is the unique inner lattice point,
\par
{\rm (c)} only the vertices are the lattice points on the boundary. \\
If a reflexive polytope satisfies the condition
\par
{\rm (d)} every face is a triangle and its 3 vertices generate the lattice,\\
it is called a Fano polytope.
Among the polytopes in (\ref{polytope P}), (\ref{polytope 1,2,3}) and (\ref{polytope Ishige}), $P_0, P_2,P_3$ and $P_4$ are the Fano polytopes.

\vspace{1.5mm}
Let us start from the polytope $P_0$  in (\ref{polytope P}).
We obtain a family of algebraic $K3$ surfaces from $P_0$ by the following canonical procedure (for detail, see \cite{Oda} Chapter 2):
\par
{\rm (i)} Make a toric 3-fold $X$ from the reflexive polytope $P_0$. This is a rational variety.
\par
{\rm (ii)} Take a divisor $D$ on $X$ that is linearly equivalent to $-K_X$.
\par
{\rm (iii)} Generically,  $D$ is represented by a $K3$ surface.

In this case, $D$ is given by
\begin{eqnarray}\label{lindiv}
a_1 + a_2 t_1 + a_3 t_2 + a_4 t_3 +a_5 \frac{1}{t_3} + a_6 \frac{1}{t_1 t_2 t_3 ^2} =0,
\end{eqnarray}
with complex parameters $a_1, \cdots, a_6$.
Every monomial in the left hand side corresponds to a lattice point in $P_0$.
Setting
\begin{eqnarray}\label{eq:para}
x=\frac{a_2 t_1}{a_1},\quad
y=\frac{a_3 t_2}{a_1},\quad
z=\frac{a_4 t_3}{a_1},\quad
\lambda =\frac{a_4 a_5}{a_1^2},\quad
\mu =\frac{a_2 a_3 a_4^2 a_6}{a_1^5},
\end{eqnarray}
 we obtain a family of $K3$ surfaces 
$\mathcal{F}_0=\{S_0 (\lambda , \mu )\}$ with two parameters $\lambda , \mu$ with
\begin{eqnarray}\label{eq:P4}
S_0(\lambda,\mu) : xyz^2(x+y+z+1)+\lambda xyz +\mu=0.
 \end{eqnarray}

In the same way, we obtain the corresponding families of  $K3$ surfaces 
$\mathcal{F}_j =\{ S_j (\lambda,\mu)\}$ for $P_j\hspace{1mm} (j=1,2,3)$ in (\ref{polytope 1,2,3})
 given by the affine equations
\begin{align}
& S_1(\lambda , \mu ) :x y z(x+y+z+1) + \lambda x + \mu y =0, \label{eq:P1}\\
& S_2(\lambda , \mu)  : x y z ( x + y+z+1) + \lambda x + \mu =0,\label{eq:P2}\\
& S_3(\lambda ,\mu ) : x y z (x +y+z+1) + \lambda z+ \mu x y =0.\label{eq:P3}
\end{align}

We give an elliptic fibration for every $K3$ surface. 
It is represented in the form
$$
 y^2 = 4 x^3 -g_2(z) x_3 -g_3(z),
$$
where $g_2$ ($g_3$, resp.) is a polynomial of $z$ with $5\leq {\rm deg} (g_2) \leq 8$ ($7 \leq {\rm deg} (g_3)  \leq  12$, resp.). 
In this paper, we call it the Kodaira normal form.
From the Kodaira normal form, we  obtain singular fibres of  elliptic fibration.

In this section, we give elliptic fibrations for our families $\mathcal{F}_j$ $(j=0,1,2,3)$ of $K3$ surfaces.
The  singular fibres of these fibration are given as  in  Table 1.

\begin{center}
\begin{tabular}{lcccc}
\toprule
Family  &$\mathcal{F}_0$& $\mathcal{F}_1$&$\mathcal{F}_2$&$\mathcal{F}_3$\\  
  \midrule
Singular Fibres &$I_3 + I_{15} + 6 I_1$ & $I_9+I_3^*+6I_1$&$ I_1^*+I_{11}+6I_1$ &$I_9+I_9+6I_1$\\
\bottomrule
\end{tabular}
\end{center} 
\begin{center}
{\bf Table 1.}
\end{center}

\subsection{Elliptic fibration for $\mathcal{F}_0$}
\begin{prop}
 {\rm (1)} The surface $S_0(\lambda,\mu)$ is birationally equivalent to the surface defined by the equation
 \begin{eqnarray}\label{eq:P4prekodaira}
y_1^2=4x_0^3+(\lambda ^2 +2\lambda z+ z^2 +2\lambda z^2 +2z^3+z^4)x_0^2+(-2\lambda \mu z-2\mu z^2 -2\mu z^3)x_0+\mu ^2 z^2.
\end{eqnarray}
This equation gives an elliptic fibration of  $S(\lambda,\mu)$ over $z$-sphere.

{\rm (2)} The elliptic surface given by  {\rm (\ref{eq:P4prekodaira})} has the holomorphic sections
\begin{align}
\begin{cases}\label{P4section}
Q: z\mapsto (x_0 ,y_1,z)=(0,\mu z ,z),\\
R: z\mapsto (x_0 ,y_1,z)=(0,-\mu z ,z).
\end{cases}
\end{align}
\end{prop}

\begin{proof}
(1) By the birational transformation
\begin{align*}
x = \frac{-\mu}{x_0} , \quad y=\frac{-\lambda x_0 -y_1  + \mu z -x_0 z - x_0 z^2 }{2 x_0 z} ,  
\end{align*}
(\ref{eq:P4}) is transformed to (\ref{eq:P4prekodaira}).

(2) This is clear.
\end{proof}

Set
\begin{eqnarray}\label{DivLambda}
\Lambda _0=\{(\lambda,\mu) \in \mathbb{C} ^2 | \lambda \mu (\lambda ^2(4\lambda -1)^3-2(2+25\lambda(20\lambda -1))\mu -3125\mu ^2)\not=0  \}.
\end{eqnarray}

\begin{prop}\label{prop1.2}
Suppose $(\lambda,\mu)\in \Lambda_0$. The elliptic surface given by {\rm (\ref{eq:P4prekodaira})} has the singular fibres of type $I_3$ over $z=0$, of type $I_{15}$ over $z=\infty$ and other six fibres of type $I_1$.
\end{prop}

\begin{proof}
(\ref{eq:P4prekodaira}) is described in the Kodaira normal form
\begin{eqnarray}\label{eq:P4kodaira}
y_1 ^ 2 =4 x_1 ^ 3 - g_2(z) x_1 - g_3(z),       z \not=\infty,
\end{eqnarray}
with
\begin{align*}
\begin{cases}
&g_2(z)=\displaystyle\frac{1}{216}(18\lambda^4+432\lambda \mu z 
+72\lambda^3z(1+z) +108\lambda^2 z^2(1+z)^2\\
\vspace*{0.3cm}&\quad\quad\quad\quad\quad\quad\quad\quad\quad\quad\quad\quad\quad
+72\lambda z^3(1+z)^3+18z^2(1+z)(24\mu +z^2(1+z)^3)),\\
&g_3(z)=\displaystyle\frac{-1}{216}(\lambda^6+36\lambda^3  \mu z +6\lambda^5z(1+z )+108\lambda^2 \mu z^2(1+z)+15\lambda^4 z^2(1+z)^2\\
&\quad\quad\quad\quad\quad\quad
+108\lambda\mu z^3(1+z)^2
+20\lambda ^3z^3(1+z)^3+15\lambda^2 z^4 (1+z)^4 +6\lambda z^5 (1+z)^5\\
&\quad\quad\quad\quad\quad\quad\quad\quad\quad\quad\quad\quad\quad\quad\quad\quad\quad\quad
+z^2(216\mu^2 +36\mu z^2(1+z)^3 + z^4 (1+z)^6)),
\end{cases}
\end{align*}
and 
\begin{eqnarray}\label{eq:P4kodaira2}
y_2 ^ 2 =4 x_2 ^ 3 - h_2(z_1) x_2 - h_3(z_1),  z_1\not= \infty,     
\end{eqnarray}
with 
\begin{eqnarray}\label{eq:P4h_3}
\begin{cases}
\vspace*{0.3cm}
& h_2(z_1)=\displaystyle2\mu z_1^5(1+z_1+\lambda z_1^2)+\frac{1}{12}(1+z_1+\lambda z_1^2)^4, \\
& h_3(z_1)=\displaystyle -(\frac{1}{6} \mu z_1^5 (1+z_1+\lambda z_1 ^2)^3 +\frac{1}{216}(1+z_1+\lambda   z_1^2)^6+\mu ^2 z_1^{10}),\notag
\end{cases}
\end{eqnarray}
where $z_1=1/z$.
We have the discriminant of the right hand side of (\ref{eq:P4kodaira}) for $x_1$ ((\ref{eq:P4kodaira2}) for $x_2$, resp.):
\begin{align}\label{eq:P4discriminant}
\begin{cases}\vspace*{0.1cm}
&D_0=64\mu^3 z^3 (\lambda^3 +3\lambda^2 z +27\mu z +3\lambda z^2 +3\lambda^2 z^2 +z^3 +6 \lambda z^3 +3z^4 +3\lambda z^4 +3z^5 +z^6),\\
&D_\infty=64\mu^3 z_1^{15} (1+3z_1 +3z_1^2+3\lambda z_1^2+z_1^3 +6\lambda z_1^3 +3 \lambda z_1^4 +3\lambda ^2 z_1^4 +3\lambda ^2 z_1^5+27\mu z_1^5 +\lambda ^3z_1^6),
\end{cases}
\end{align}
respectively.

From these data, we obtain the required statement (see \cite{Kodaira2}).
\end{proof}

\begin{rem}\label{Lambda remark}
We have a parametrization
$$
\lambda (a)=\frac{(a-1)(a+1)}{5}, \quad \mu(a)=\frac{(2a-3)^3 (a+1)^2}{3125}
$$
of the locus $\lambda ^2(4\lambda -1)^3-2(2+25\lambda(20\lambda -1))\mu -3125\mu ^2=0$.
It is a rational curve.
In Section 4, we shall obtain the above $\Lambda_0$ as the complement of the singular locus of the period differential equation for $\mathcal{F}_0$
 in the $(\lambda,\mu)$-space.
\end{rem}

\begin{rem}
Let $\chi$ denote the Euler characteristic.
According to {\rm \cite{Kodaira2} Theorem 12.1}
(see also {\rm \cite{Shiga3}}),
 an elliptic fibred algebraic surface $S$ over $\mathbb{P}^1(\mathbb{C})$ is a $K3$ surface if and only if $\chi (S)=24$  provided $S$ is given in the Kodaira normal form. Due to this criterion and {\rm Proposition \ref{prop1.2}}, we can check directly  that  $S_0(\lambda,\mu)$ is a $K3$ surface for $(\lambda,\mu)\in\Lambda_0$.
\end{rem}

For $(\lambda,\mu)\in \Lambda_0$, let $O$ be the zero of the Mordell-Weil group  of sections of the elliptic fibration given by (\ref{eq:P4prekodaira}) over $\mathbb{C}(z)$ (see Section 3.2). $O$ is given by the set of the points at infinity on every fibre.
Let $Q$ and $R$ be the sections in (\ref{P4section}).
 $R$ is the inverse element of $Q$ in the Mordell-Weil group.
 Let $F$ be a general fibre of this fibration. 
Let $I_3=a_0+a_1+a_1'$ be the irreducible decomposition of the fibre at $z=0$ given as in Figure 1. 
 We may suppose $O\cap a_0 \not= \phi, Q\cap a_1\not= \phi$ and $R\cap a_1'\not=\phi$.
By the same way, let  $I_{15}= b_0 + b_1+\cdots+b_7+b_1'+\cdots+b_7'$ be the irreducible decomposition of the fibre at $z = \infty$ given as in  Figure 2. 
We may suppose $O\cap b_0 \not= \phi, Q\cap b_5\not= \phi$ and $R\cap b_5'\not=\phi$.

\begin{figure}[h]
\center
\includegraphics[scale=1]{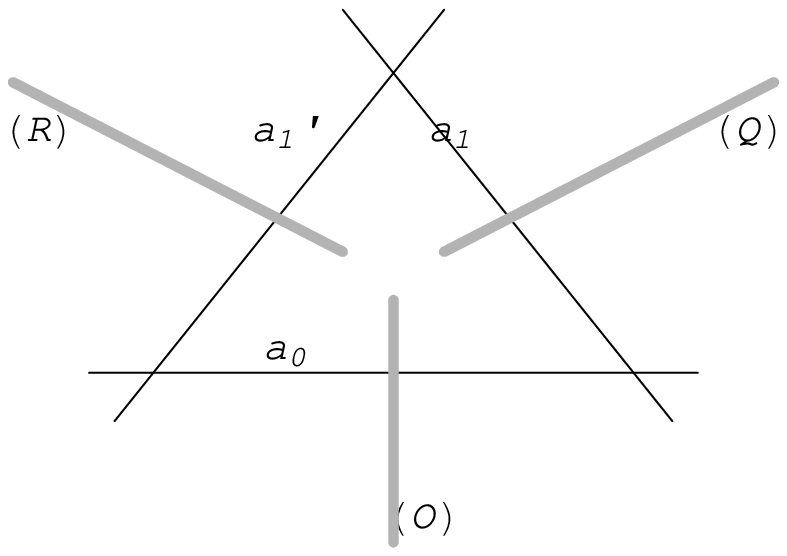}
\caption{}
\end{figure}

\begin{figure}[h]
\center
\includegraphics[scale=1]{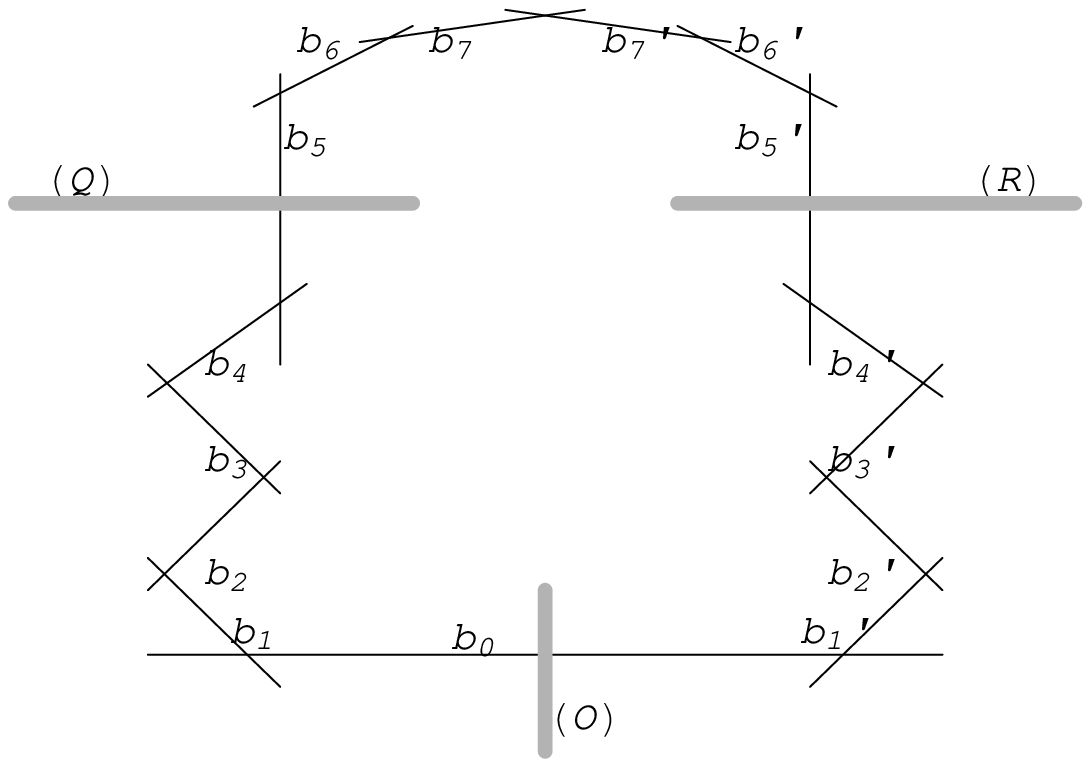}
\caption{}
\end{figure}

We set a sublattice $L_0=L_0 (\lambda, \mu) \subset H_2(S_0 (\lambda,\mu) ,\mathbb{Z})$ for $(\lambda,\mu)\in \Lambda_0$ by
\begin{eqnarray}\label{N-Sbasis}
L_0 (\lambda,\mu) =\langle b_1,b_2,b_3,b_4,b_5,Q,b_6,b_7,b'_1,b'_2,b'_3,b'_4,b'_5,R,b'_6,b'_7,F,O\rangle_{\mathbb{Z}}.
\end{eqnarray}

Set
\begin{eqnarray*}\label{A18}
A_{18}(-1)=
\underbrace{
\left(
\begin{array}{ccccccc}
-2 &1 & & &  &&\\
1&-2 &1& &&&\\
 & 1& -2&\ddots & &&\\
 & & & \ddots&&&\\
 & & &\ddots & -2&1&\\
 &&&&1&-2&1\\
 &&&&&1&-2
\end{array}
\right)}_{18} 
 \quad.
\end{eqnarray*}
Let 
$E_{i,j}$   $(1\leq i,j\leq 18) $ be the  matrix unit.
We obtain the corresponding intersection matrix $M_0$ for $L_0$:
\begin{align}
\label{matrix:P4preintersection}
M_0=&A_{18}(-1) +2E_{17,17}  -(E_{6,7}+E_{7,6})+(E_{5,7}+E_{7,5})-(E_{14,15}+E_{15,14})+(E_{13,15}+E_{15,13}) \notag \\
&-(E_{8,9}+E_{9,8}) -(E_{16,17}+E_{17,16})+(E_{6,17}+E_{17,6})+(E_{8,16}+E_{16,8})+(E_{14,17}+E_{17,14}).
\end{align}

We have 
\begin{eqnarray}\label{-5 matrix}
{\rm det} (M_0)=-5.
\end{eqnarray}
Therefore, the generators of $L_0$ are independent.

\subsection{Elliptic fibration for $\mathcal{F}_1$}
\begin{prop} 
The surface $ S_1(\lambda,\mu)$ is birationally equivalent  to the surface defined by the equation
\begin{eqnarray}\label{P1preKodaira}
z_1^2= y_1^3+ (\mu^2 + 2\mu x_1 + 
          x_1^2 - 4 x_1^3) y_1^2 +( -8\lambda \mu  x_1^3 - 
        8\lambda x_1^4)y_1+16 \lambda^2 x_1^6  .
  \end{eqnarray}
This equation gives an elliptic fibration of $ S_1(\lambda,\mu)$
with the holomorphic section
\begin{align}\label{P1section}
Q: x_1 \mapsto (x_1,y_1,z_1)=(x_1, 0, 4 \lambda x_1^3).
\end{align}
\end{prop}

\begin{proof}
By the birational transformation 
\begin{align*}
x=-\frac{2 x_1^2 y_1}{-4 \lambda x_1^3 +\mu  y_1 + x_1 y_1 +z_1},
 y=\frac{y_1^2}{2 x_1 (-4 \lambda x_1^3 +\mu  y_1 + x_1 y_1 +z_1)},
 z=-\frac{-4 \lambda x_1^3 +\mu  y_1 + x_1 y_1 +z_1}{2 x_1 y_1},
\end{align*}
 {\rm(\ref{eq:P1})} is transformed to (\ref{P1preKodaira}).
  \end{proof}

(\ref{P1preKodaira}) gives an elliptic fibration for the surface $ S_1(\lambda,\mu)$.
Set
\begin{eqnarray}\label{Lambda_1}
\Lambda_1=\{(\lambda,\mu)\in\mathbb{C}^2| \lambda \mu (729\lambda^2 - 54\lambda(27\mu-1) +(1+27\mu)^2\not =0)\}.
\end{eqnarray}

\begin{prop}\label{propP1kodaira}
Suppose $(\lambda,\mu)\in\Lambda_1$. The elliptic surface given by {\rm(\ref{P1preKodaira})} has the singular fibres  of type $I_9$ over $x_1=0$, of type $I_3^{*}$ over $x_1=\infty$ and other six fibres of type $I_1$.
\end{prop}

\begin{proof}
(\ref{P1preKodaira}) is described in the Kodaira normal form
\begin{eqnarray}\label{P1Kodaira}
z_2 ^ 2 = 4 y_2 ^ 3 - g_2(x_1) y_2 - g_3(x_1),    \quad   x_1 \not=\infty,
\end{eqnarray}
with
\begin{align*}
\begin{cases}
\vspace*{0.2cm}
&g_2(x_1)=\displaystyle-4\Big( -\frac{\mu^4}{3} -\frac{4\mu^3 x_1}{3} - 2\mu^2 x_1^2 -\frac{4\mu x_1^3}{3} - 8\lambda\mu x_1^3 +\frac{8\mu^2 x_1^3}{3}- \frac{x_1^4}{3}-8 \lambda x_1^4 + \frac{16 \mu x_1^4}{3} + \frac{8 x_1^5}{3}- \frac{16 x_ 1^6}{3}\Big),\\
\vspace*{0.2cm}
&g_3(x_1)=\displaystyle-4\Big( \frac{2 \mu ^6}{27} +\frac{4 \mu^5 x_1}{9} + \frac{10 \mu ^4 x_1^2}{9}+\frac{40 \mu^3 x_1^3}{27} +\frac{8 \lambda \mu^3 x_1^3}{3}-\frac{8 \mu^4 x_1^3}{9}+\frac{10\mu^2 x_1^4 }{9} \\
\vspace*{0.2cm}
&\displaystyle
\quad\quad\quad\quad\quad\quad\quad\quad+8\lambda \mu^2 x_1^4 - \frac{32 \mu^3 x_1^4}{9} +\frac{4\mu x_1^5}{9} +8\lambda \mu x_1^5
- \frac{16\mu^2 x_1^5}{3}+\frac{2 x_1^6}{27}+\frac{8\lambda x_1^6 }{3}+16 \lambda^2 x_1^6 - \frac{32 \mu x_1^6}{9}\\
&\displaystyle
\quad\quad\quad\quad\quad\quad\quad\quad
 -\frac{32 \lambda \mu x_1^6}{3} +\frac{32\mu^2 x_1^6 }{9}- \frac{8 x_1^7}{9} -\frac{32 \lambda x_1^7}{3} +\frac{64 \mu x_1^7}{9} +\frac{32x_1^8}{9}- \frac{128 x_1^9}{27}\Big),
\end{cases}
\end{align*}
and
\begin{eqnarray}\label{P1Kodaira2}
z_3 ^ 2 = 4y_3 ^ 3 - h_2(x_2) y_3 - h_3(x_2),    \quad   x_2 \not=\infty,
\end{eqnarray}
with
\begin{align*}
\begin{cases}
\vspace*{0.2cm}
&h_2(x_2)=\displaystyle -4\Big( -\frac{16 x_2^2}{3}+\frac{8 x_2^3}{3} -\frac{x_2^4}{3} -8\lambda x_2^4 +\frac{16 \mu x_2^4}{3}-\frac{4\mu x_2^5}{3}- 8\lambda \mu x_2^5 + \frac{8 \mu^2 x_2^5}{3}-2 \mu^2 x_2^6 -\frac{4\mu^3  x_2^7}{3}-\frac{\mu^4 x_2^8}{3}\Big),\\
\vspace*{0.2cm}
&\displaystyle h_3(x_2)=-4\Big(-\frac{128 x_2^3}{27}+\frac{32 x_2^4}{9} -\frac{8x_2^5}{9}-\frac{32 \lambda x_2^5}{3}+\frac{64 \mu x_2^5}{9}+ \frac{2 x_2^6}{27} +\frac{8 \lambda x_2^6}{3} +16\lambda^2 x_2^6 -\frac{32 \mu x_2^6 }{9} \\
\vspace*{0.2cm}
&\displaystyle \quad\quad\quad\quad\quad\quad\quad\quad-\frac{32 \lambda \mu x_2^6}{3}+ \frac{32 \mu^2 x_2^6}{9} +\frac{4\mu x_2^7}{9}+8\lambda \mu x_2^7 -\frac{16\mu^2 x_2^7}{3}+\frac{10 \mu^2 x_2^8}{9}+8\lambda \mu^2 x_2^8 -\frac{32 \mu^4 x_2^8}{9}\\
\vspace*{0.2cm}
&\displaystyle \quad\quad\quad\quad\quad\quad\quad\quad +\frac{40 \mu^3 x_2^{11}}{27}+\frac{8 \lambda \mu^3 x_2^9}{3} -\frac{8\mu^4 x_2^9}{9} +\frac{10\mu^4 x_2^{10}}{9} + \frac{4\mu^5 x_2^{11}}{9}+\frac{2 \mu^6 x_2^{12}}{27}\Big),
\end{cases}
\end{align*}
where $x_1=1/x_2$. We have the discriminant of the right hand side of (\ref{P1Kodaira}) for $y_1$ ((\ref{P1Kodaira2}) for $y_2$, resp.):
\begin{align*}
\begin{cases}
&D_0=256\lambda^2 x_1^9(\lambda \mu^3 - \mu^4 + 3 \lambda 
      \mu^2 x_1 - 4 \mu^3 x_1 + 3 \lambda \mu x_1^2 - 6 \mu^2 x_1^2 + \lambda x_1^3 + 
        27 \lambda^2 x_1^3 \\
 &\quad\quad\quad\quad\quad\quad\quad\quad\quad
         - 4 \mu x_1^3 - 36 \lambda \mu x_1^3 + 8 \mu^2 
        x_1^3 - x_1^4 - 36 \lambda x_1^4 + 16 \mu 
          x_1^4 + 8 x_1^5 - 16 x_1^6),\\
&D_\infty=256 \lambda^2 x_2^9(-16 + 8 x_2 - x_2^2 - 36 \lambda x_2^2 + 16 
        \mu x_2^2 + \lambda x_2^3 + 27 \lambda^2 x_2^3 - 4 \mu x_2^3 - 36 
        \lambda \mu x_2^3 \\
&        \quad\quad\quad\quad\quad\quad\quad\quad\quad\quad
 8 \mu^2 x_2^3 + 3 \lambda \mu x_2^4 - 6 \mu^2 
          x_2^4 + 3 \lambda \mu^2 x_2^5 - 4 \mu^3 x_2^5 + \lambda \mu^3 x_2^6 - 
        \mu^4 x_2^6).
        \end{cases}
\end{align*}
From these deta, we obtain the required statement.
\end{proof}

The elliptic fibration given by (\ref{P1preKodaira}) is illustrated in  Figure 3.

\begin{figure}[h]
\center
\includegraphics[scale=0.9]{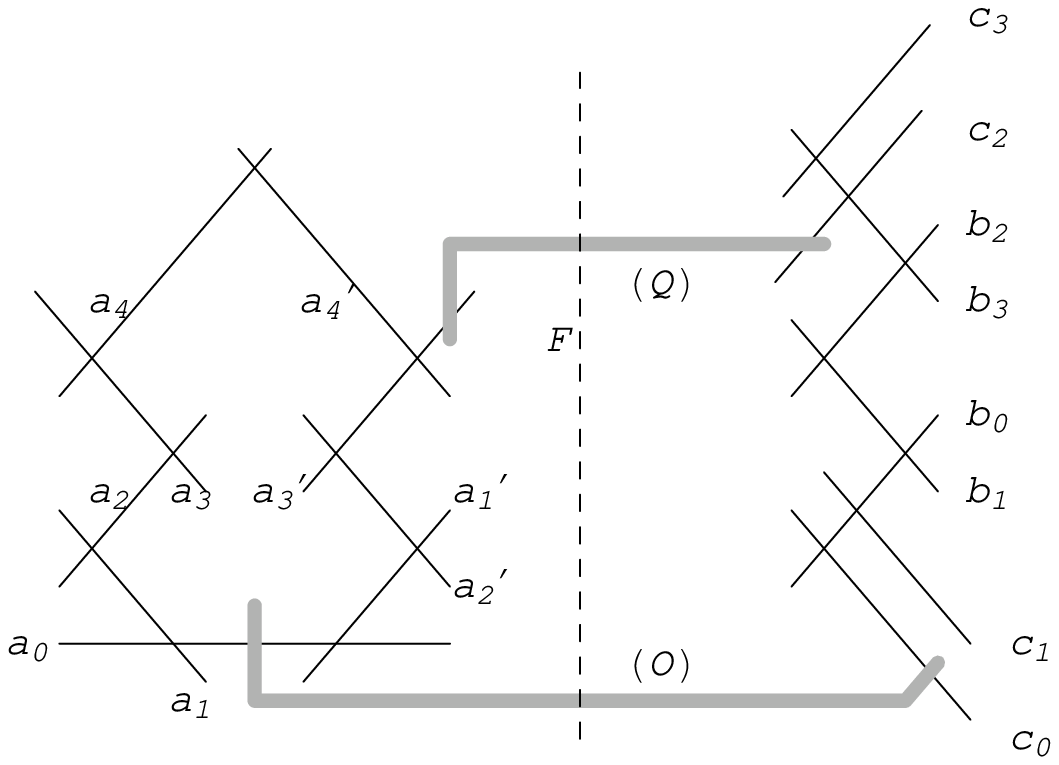}
\caption{}
\end{figure}

For this fibration, let $O$ be the zero of the Mordell-Weil group,  $Q$ be  the section in (\ref{P1section}) and $F$ be a general fibre.
Note that $Q\cap a_3\not = \phi$ at $x_1=0$ and $Q\cap c_2 \not =\phi$ at $x_1 = \infty$.
Set
\begin{eqnarray}\label{L1definition}
L_1=\langle a_1,a_2,a_3,a_4,a_4',a_3',a_2',a_1',c_1,b_0,b_1,b_2,b_3,c_2,c_3,O,Q,F  \rangle_\mathbb{Z}.
\end{eqnarray}
We have the following intersection matrix $M_1$ for $L_1$:
\begin{align}\label{intersection L1} \notag
M_1=&A_{18}(-1)
 -(E_{8,9}+E_{9,8})-(E_{14,15}+E_{15,14})+(E_{13,15}+E_{15,13})
+(E_{3,17}+E_{17,3})\\ 
&+(E_{14,17}+E_{17,14})-(E_{16,17}+E_{17,16})+(E_{16,18}+E_{18,16})-(E_{15,16}+E_{16,15})+2E_{18,18}.
\end{align}
We have ${\rm det}(M_1)=-9$.
Therefore, the generators of $L_1$ are independent.

\subsection{Elliptic fibration for $\mathcal{F}_2$}
\begin{prop}
The surface $ S_2(\lambda,\mu)$ is birationally equivalent  to the surface defined by the equation
\begin{eqnarray}\label{P2preKodaira}
z_1^2=x_1^3  + (-4\lambda y + 
    y^2 + 2 y^3 + y^4) x_1^2 + (-8\mu y^3 - 8\mu  y^4) x_1 + 16\mu ^2 y^4.
\end{eqnarray}
This equation gives an elliptic fibration of $ S_2(\lambda,\mu)$ with the holomorphic section
\begin{align}\label{P2section}
Q: y\mapsto (x_1,y,z_1)=(0,y,4\mu y^2)
\end{align}
\end{prop}

\begin{proof}
 By the birational transformation 
\begin{align*}
x=\frac{x_1^2}{2 y (x_1 y - 4 \mu y^2 + x_1 y + z_1)}, z=-\frac{x_1 y - 4 \mu y^2 + x_1 y + z_1}{2 x_1 y},
\end{align*}
(\ref{eq:P2}) is transformed to (\ref{P2preKodaira}).
\end{proof}

(\ref{P2preKodaira}) gives an elliptic fibration for $ S_2(\lambda,\mu)$. Set
\begin{eqnarray}\label{Lambda_2}
\Lambda_2=\{(\lambda,\mu)\in\mathbb{C}^2| \lambda\mu (\lambda ^2(1+27\lambda )^2-2\lambda \mu(1+189\lambda )+(1+576\lambda )\mu ^2-256\mu ^3)\not = 0  \}.
\end{eqnarray}

\begin{prop}\label{propP2kodaira}
Suppose $(\lambda,\mu)\in\Lambda_2$. The elliptic surface given by {\rm(\ref{P2preKodaira})} has  the singular fibres of type $I_1^{*}$ over $y=0$,   of type $I_{11}$ over $y=\infty$ and other six fibres of type $I_1$. 
\end{prop}

\begin{proof}
(\ref{P2preKodaira}) is described in the Kodaira normal form
\begin{eqnarray}\label{P2Kodaira}
z_2 ^ 2 = 4x_2 ^ 3 - g_2(y) x_2 - g_3(y),    \quad   y \not=\infty,
\end{eqnarray}
with
\begin{align*}
\begin{cases}
\vspace*{0.2cm}
&\displaystyle g_2(y)=-4\Big(-\frac{16\lambda^2 y^2}{3}+\frac{8\lambda y^3}{3}-8\mu y^3 -\frac{y^4}{3}+\frac{16\lambda y^4}{3}-8\mu y^4 -\frac{4y^5}{3}+\frac{8\lambda y^5}{3}-2y^6-\frac{4y^7}{3}-\frac{y^8}{3}\Big),\\
\vspace*{0.2cm}
&\displaystyle g_3(y)=-4\Big(-\frac{128 \lambda ^3 y^3}{27}+\frac{32\lambda^2 y^4}{9}-\frac{32 \lambda \mu y^4}{3} +16 \mu^2 y^4 - \frac{8\lambda y^5 }{9} +\frac{64 \lambda^2 y^5 }{9}+\frac{8\mu y^5}{3} -\frac{32 \lambda \mu y^5}{3}+\frac{2y^6}{27}-\frac{32\lambda y^6}{9} \\
&\displaystyle\quad\quad
+\frac{32\lambda^2 y^6}{9}+8\mu y^6 +\frac{4y^7}{9} -\frac{16\lambda y^7}{3}+8\mu y^7 +\frac{10 y^8}{9}-\frac{32 \lambda y^8}{9} +\frac{8\mu y^8}{3} +\frac{40 y^9}{27}-\frac{8\lambda y^9}{9} +\frac{10 y^{10}}{9}+\frac{4 y^{11}}{9} +\frac{2 y^{12}}{27}\Big),
\end{cases}
\end{align*}
and
\begin{eqnarray}\label{P2Kodaira2}
z_3 ^ 2 = 4x_3 ^ 3 - h_2(y_1) x_3 - h_3(y_1),    \quad   y_1 \not=\infty,
\end{eqnarray}
with
\begin{align*}
\begin{cases}
\vspace*{0.2cm}
&\displaystyle h_2(y_1)=-4\Big(-\frac{1}{3} -\frac{4 y_1 }{3} -2y_1^2 -\frac{4y_1^2}{3}+\frac{8\lambda y_1^3}{3}-\frac{y_1^4}{3} +\frac{16\lambda y_1^4}{3}-8\mu y_1^4 +\frac{8\lambda y_1^5}{3} -8\mu y_1^5 -\frac{16\lambda^2 y_1^6}{3}\Big),\\
\vspace*{0.2cm}
&\displaystyle h_3(y_1)=-4\Big(\frac{2}{27} +\frac{4 y_1 }{9} +\frac{10 y_1^2}{9} +\frac{40 y_1^3}{27} -\frac{8 \lambda y_1^3}{9} +\frac{10 y_1^4}{9} - \frac{32 \lambda y_1^4}{9} +\frac{8\mu y_1^4}{3} +\frac{4 y_1^5}{9} -\frac{16\lambda y_1^5}{3} +8\mu y_1^5\\
\vspace*{0.2cm}
&\displaystyle    \quad\quad\quad\quad\quad\quad\quad\quad
+\frac{2 y_1^6}{27} -\frac{32 \lambda y_1^6}{9} +\frac{32 \lambda^2 y_1^6}{9} +8\mu y_1^6 -\frac{8\lambda y_1^7}{9}  +\frac{64 \lambda^2 y_1^7}{9} +\frac{8\mu y_1^9}{3}\\
\vspace*{0.2cm}
&\displaystyle    \quad\quad\quad\quad\quad\quad\quad\quad\quad\quad 
- \frac{32 \lambda \mu y_1^7}{3}+\frac{32 \lambda^2 y_1^8}{9} -\frac{32\lambda \mu y_1^8}{3} +16 \mu^2 y_1^8-\frac{128 \lambda^3 y_1^9}{27}\Big),
\end{cases}
\end{align*}
where $y=1/y_1$. 
We have the discriminant of the right hand side of (\ref{P2Kodaira}) for $x_2$((\ref{P2Kodaira2}) for  $x_3$, resp.):
\begin{align*}
\begin{cases}
&D_0=-256\mu^2 y^7(16\lambda ^3 - 8\lambda^2 y + 36 
      \lambda \mu y - 27\mu^2 y +
       \lambda  y^2 - 16\lambda^2 y^2 - \mu y^2 + 36\lambda\mu y^2 + 4\lambda y^3\\
 &  \quad\quad\quad\quad\quad\quad\quad\quad\quad\quad\quad\quad
    - 8\lambda^2 y^3 - 3\mu y^3 + 6\lambda y^4 - 3 \mu  y^4 + 4 \lambda y^5 - 
       \mu y^5 + \lambda y^6), \\
&D_\infty=-256\mu ^2 y_1^{11}(\lambda + 4\lambda  y_1 - \mu y_1 + 6\lambda  
      y_1^2 - 3\mu y_1^2 + 4\lambda y_1^3 - 8\lambda^2 
        y_1^3 - 3\mu  y_1^3 + \lambda y_1^4 - 16\lambda^2 y_1^4\\
        &\quad\quad\quad\quad\quad\quad\quad\quad\quad\quad\quad\quad
                 - \mu y_1^4 + 36\lambda \mu y_1^4 - 8\lambda^2 y_1^5 + 36\lambda\mu y_1^5 - 27\mu^2 y_1^5 + 16\lambda^3 y_1^6).
\end{cases}
\end{align*}
From these data, we obtain the required statement.
\end{proof}

The elliptic fibration given by (\ref{P2preKodaira}) is illustrated in  Figure 4.

\begin{figure}[h]
\center
\includegraphics[scale=0.9]{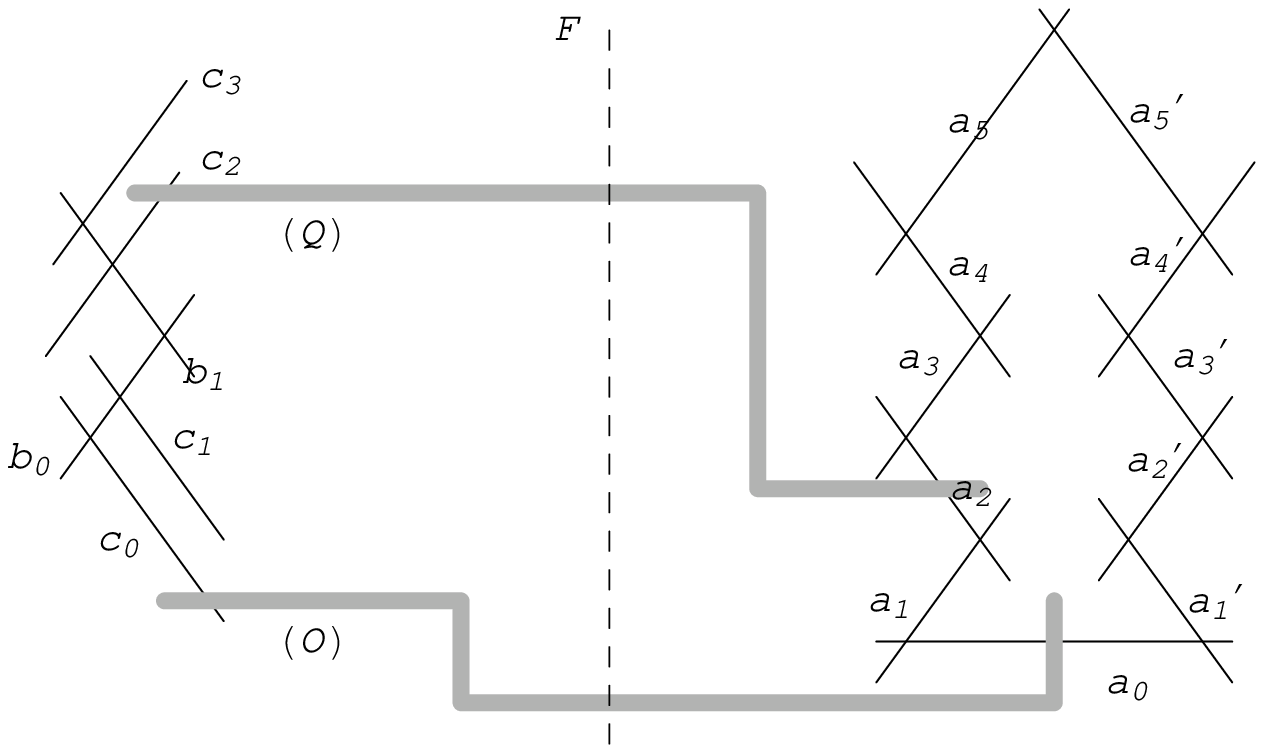}
\caption{}
\end{figure}

For this fibration, let $O$ be the Mordell-Weil group,  $Q$ be the section in (\ref{P2section}) and $F$ be a general fibre.  Note $Q\cap a_2\not =\phi$ and $Q\cap c_2\not=\phi$.
Set
\begin{eqnarray} \label{L2definition}
L_2=\langle  a_1,a_2,a_3,a_4,a_5,a_5',a_4',a_3',a_2',a_1',c_1,b_0,b_1,c_2,c_3,O,Q,F \rangle_\mathbb{Z}.
\end{eqnarray}
We have the following intersection matrix $M_2$ for $L_2$:
\begin{align}\label{intersection L2}\notag
M_2=& A_{18}(-1) -(E_{10,11}+E_{11,10})-(E_{15,16}+E_{16,15})+(E_{4,17}+E_{17,4})+(E_{14,17}+E_{17,4})\\ 
& -(E_{14,15}+E_{15,14})
+(E_{13,15}+E_{15,13})-(E_{16,17}+E_{17,16})+(E_{16,18}+E_{18,16})+2E_{18,18}.
\end{align}
We have ${\rm det}(M_2)=-9$.

\subsection{Elliptic fibrations for $\mathcal{F}_3$}
\begin{prop}
 The surface $ S_3(\lambda,\mu)$ is birationally equivalent  to the surface defined by the equation
\begin{eqnarray}\label{P3preKodaira}
y_1^2=
z_1^3  + (\lambda^2 + 2 \lambda x_1 + x_1^2 -4\mu x_1^2 -4 x_1^3) z_1^2+ 16 \mu x_1^5 z_1 .
\end{eqnarray}
This equation gives an elliptic fibration of $ S_3(\lambda,\mu)$ with the holomorphic sections
\begin{align}\label{P3section}
\begin{cases}
Q: z_1 \mapsto (x_1,y_1,z_1)=(x_1,4 \mu x_1^2 (x_1+\lambda),4 x_1^2 \mu),\\
O':z_1 \mapsto (x_1,y_1,z_1)=(x_1,0,0).
\end{cases}
\end{align}
The section $O'$ satisfies $2O'=O$.
\end{prop}

\begin{proof}
 By the birational transformation
\begin{align*}
x=\frac{2 x_1^2 (4\mu x_1^2 - z_1)}{y_1 + \lambda z_1+ x_1 z_2},
y=\frac{y_1 + \lambda z_1 + x_1 z_1 }{2 x_1 (4 \mu x_1^2 - z_1)},
z=-\frac{z_1( 4\mu x_1^2 - z_1)}{2 x_1 (y_1 + \lambda z_1 + x_1 z_1)},
\end{align*}
(\ref{eq:P3}) is transformed to (\ref{P3preKodaira}).
\end{proof}

(\ref{P3preKodaira}) gives an elliptic fibration for $ S_3(\lambda,\mu)$. Set
\begin{eqnarray} \label{Lambda_3}
\Lambda_3=\{(\lambda,\mu)\in\mathbb{C}^2|\lambda\mu (729\lambda ^2-(4\mu -1)^3+54\lambda (1+12\mu ))\not =0\}.
\end{eqnarray}

\begin{prop}
Suppose $(\lambda,\mu)\in\Lambda_3$. The elliptic surface given by {\rm(\ref{P3preKodaira})} has the singular fibres of type $I_{10}$ over $z=0$,   of type $I^*_{2}$ over $z=\infty$ and other six fibres of type $I_1$. 
\end{prop}

\begin{proof}
(\ref{P3preKodaira}) is described in the Kodaira normal form
\begin{eqnarray}\label{P3Kodaira1}
y_2^2 = 4z_2^3 - g_2(x_1) z_2 -g_3(x_1), \quad x_1\not = \infty,
\end{eqnarray}
with
\begin{align*}
\begin{cases}
\vspace*{0.2cm}
&\displaystyle g_2(x_1)=-4\Big(-\frac{\lambda^4}{3}-\frac{4\lambda^3 x_1}{3} - 2\lambda^2 x_1^2 +\frac{8\lambda^2 \mu x_1^2}{3}-\frac{4\lambda x_1^3 }{3}+\frac{8 \lambda ^2 x_1^3}{3}+\frac{16 \lambda\mu x_1^3}{3}\\
\vspace*{0.2cm}
& \displaystyle
\quad\quad\quad\quad\quad\quad\quad
-\frac{x_1^4}{3}+\frac{16 \lambda x_1^4}{3}+\frac{8\mu x_1^4}{3}- \frac{16\mu^2 x_1^4}{3}+\frac{8 x_1^5}{3}+\frac{16\mu x_1^5}{3}-\frac{16 x_1^6}{3}
 \Big),\\
 \vspace*{0.2cm}
 &\displaystyle
 g_3(x_1)=-4\Big(\frac{2 \lambda^6 }{27}+ \frac{4\lambda^5 x_1}{9}+\frac{10\lambda^4 x_1^2}{9}-\frac{8\lambda^4 \mu x_1^2}{9}+\frac{40 \lambda ^3 x_1^3}{27}-\frac{8\lambda^4 x_1^3}{9}-\frac{32 \lambda^3 \mu x_1^3 }{9}  +\frac{10\lambda^2 x_1^4}{9}-\frac{32\lambda^3 x_1^4}{9} \\
  \vspace*{0.2cm}
  &\displaystyle
   \quad\quad
  -\frac{16 \lambda^2 \mu  x_1^4}{3}+\frac{32 \lambda^2 \mu^2 x_1^4}{9}+\frac{4 \lambda x_1^5}{9}-\frac{16 \lambda^2 x_1^5}{3}-\frac{32 \lambda \mu x_1^5}{9}+ \frac{16 \lambda ^2 \mu x_1^5}{9}+\frac{64 \lambda\mu^2 x_1^5}{9}+\frac{2 x_1^6}{27}-\frac{32 \lambda x_1^6}{9} +\frac{32 \lambda ^2 x_1^6}{9}\\
  \vspace*{0.2cm}
  &\displaystyle
   \quad\quad
-\frac{8\mu x_1^6}{9}+\frac{32 \lambda \mu x_1^6}{9}+\frac{32 \mu^2 x_1^6}{9}-\frac{128 \mu^3 x_1^6}{27}-\frac{8 x_1^7}{9}+\frac{64 \lambda x_1^7}{9}+\frac{16\mu x_1^7}{9}+\frac{64 \mu^2 x_1^7}{9}+\frac{32 x_1^8}{9}+\frac{64 \mu x_1^8}{9}-\frac{128 x_1^9}{27}
  \Big), 
\end{cases}
\end{align*}
and 
\begin{eqnarray}\label{P3Kodaira2}
y_3^2 =4 z_3^3 - h_2(x_2) z_3 -h_3(x_2), \quad x_2\not = \infty,
\end{eqnarray}
with
\begin{align*}
\begin{cases}
\vspace*{0.2cm}
&\displaystyle h_2(x_2)=-4\Big(-\frac{16 x_2^2}{3}+\frac{8 x_2^3}{3}+\frac{16\mu x_2^3}{3}-\frac{x_2^4}{3}+\frac{16\lambda x_2^4}{3}+\frac{8\mu x_2^4}{3}-\frac{16\mu^2 x_2^4}{3}-\frac{4\lambda x_2^5}{3}
\\
\vspace*{0.2cm}
& \displaystyle
\quad\quad\quad\quad\quad\quad\quad\quad\quad\quad\quad\quad\quad\quad\quad
+\frac{8\lambda^2 x_2^5 }{3}+\frac{16 \lambda \mu x_2^5}{3}-2\lambda^2 x_2^6+\frac{8\lambda^2 \mu x_2^6}{3}-\frac{4 \lambda^3 x_2^7}{3}-\frac{\lambda^4 x_2^8}{3},
\Big),\\
 \vspace*{0.2cm}
 &\displaystyle
 h_3(x_2)=-4\Big(-\frac{128 x_2^3}{27}+\frac{32 x_2^4}{9}+\frac{64 \mu x_2^4}{9}-\frac{8 x_2^5}{9}+\frac{64 \lambda x_2^5}{9}
 +\frac{16 \mu x_2^5}{9}+\frac{64\mu^2 x_2^5}{9}
 +\frac{2 x_2^6}{27}-\frac{32 \lambda x_2^6}{9}+\frac{32 \lambda^2 x_2^6}{9}-\frac{8\mu x_2^6}{9}
 \\
  \vspace*{0.2cm}
  &\displaystyle
   \quad\quad\quad
  +\frac{32 \lambda \mu x_2^6}{9}+\frac{32 \mu^2 x_2^6}{9}-\frac{128 \mu^3 x_2^6}{27}+\frac{4 \lambda x_2^7}{9} - \frac{16\lambda^2 x_2^7 }{3}- \frac{32 \lambda^2 \mu^2 x_2^7}{9}+\frac{16\lambda^2 \mu x_2^7}{9}-\frac{32 \lambda^3 x_2^8}{9}+\frac{64\lambda \mu^2 x_2^7}{9}  \\
  \vspace*{0.2cm}
  &\displaystyle
   \quad\quad\quad
 +\frac{10\lambda^2x_2^8}{9}-\frac{16\lambda^2\mu x_2^8}{3}
 +\frac{40 \lambda^3 x_2^9}{27}
  -\frac{8\lambda^4 x_2^9}{9} - \frac{32 \lambda^3 \mu x_2^9}{9}+\frac{10 \lambda^4 x_2^{10}}{9}-\frac{8\lambda^4 \mu x_2^{10}}{9}+\frac{4\lambda^5 x_2^{11}}{9}
+\frac{2 \lambda^6 x_2^{12}}{9}
  \Big), 
\end{cases}
\end{align*}
where $x_1=1/x_2$. We have the discriminant of the right hand side of (\ref{P3Kodaira1}) for $z_2$ ((\ref{P3Kodaira2}) for $z_3$ resp):

\begin{align*}
\begin{cases}
&D_0=-256 \mu^3 x_1^{10} (\lambda^4 +4\lambda^3 x_1 +6\lambda^2 x_1^2 -8\lambda^2 \mu x_1^2+4\lambda x_1^3-8\lambda^2 x_1^3 -16\lambda\mu x_1^3\\
&\quad\quad\quad\quad\quad\quad\quad\quad\quad\quad\quad\quad\quad\quad
+x_1^4-16\lambda x_1^4 -8\mu x_1^4 +16\mu^2 x_1^4 - 8x_1^5 -32 \mu x_1^5 +16 x_1^6 
),\\
&D_\infty=-256 \mu^2  x_2^8 (16-8 x_2 -32 \mu x_2+x_2^2 -16 \lambda x_2^2 -8\mu x_2^2+16\mu^2 x_2^2\\
&\quad\quad\quad\quad\quad\quad\quad\quad\quad\quad\quad\quad\quad\quad
+4\lambda x_2^3-8\lambda^2 x_2^3 -16\lambda\mu x_2^3+6\lambda^2 x_2^4 -8\lambda^2 \mu x_2^4 +4\lambda^3 x_2^5 +\lambda^4 x_2^6).
\end{cases}
\end{align*}
From these data, we obtain the required statement.
\end{proof}

The elliptic fibration given by (\ref{P3preKodaira}) is illustrated in  Figure 5.

\begin{figure}[h]
\center
\includegraphics[scale=0.9]{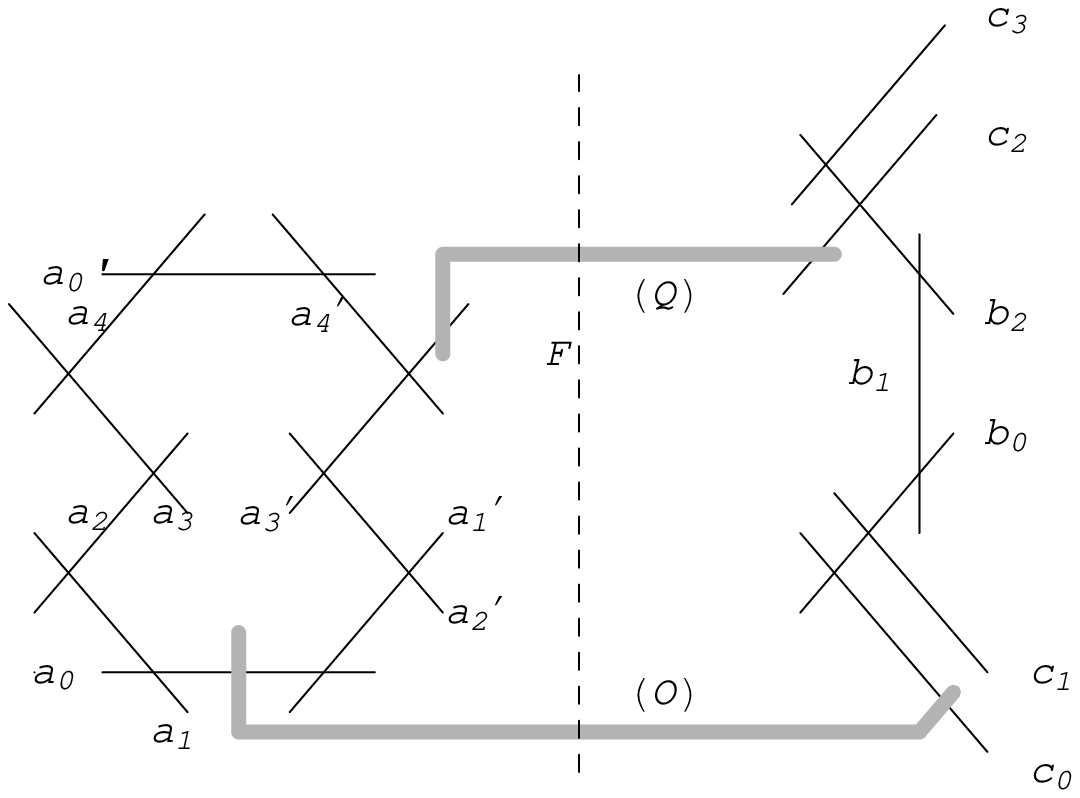}
\caption{}
\end{figure}

For this fibration,  let $O$ be the zero of the Mordell-Weil group, $Q
$ be the section in (\ref{P3section}) and $F$ be a general fibre.
Set
\begin{eqnarray} \label{L3'def}
L'_3=\langle a_1,a_2,a_3,a_4,a_0', a_4',a_3',a_2',a_1',c_1,b_0,b_1,b_2,c_2,c_3,O,F,Q
 \rangle_\mathbb{Z}.\end{eqnarray}
We have ${\rm det}(L_3')=-36$.

\vspace{5mm}

We need another elliptic fibration.

\begin{prop}
 The surface $ S_3(\lambda,\mu)$ is birationally equivalent  to the surface defined by the equation
\begin{eqnarray}\label{P3preKodairaanother}
y_1^2=x_1^3  + (\mu^2 + 2\mu z + z^2 + 2\mu z^2 + 
        2 z^3 + z^4)x_1^2 + (-8\lambda \mu  z^3 - 8\lambda  z^4 - 8\lambda  z^5)x_1+ 16\lambda^2 z^6.
\end{eqnarray}
This equation gives an elliptic fibration of $ S_3(\lambda,\mu)$ 
with the holomorphic sections
\begin{align}\label{P3sectionanother}
\begin{cases}
&Q_0: z \mapsto (x_1,y_1,z)=(0, 4 \lambda z^3, z),\\
&R_0:z\mapsto  (x_1,y_1,z)=(0, -4 \lambda z^3, z).
\end{cases}
\end{align}
\end{prop}

\begin{proof}
By the birational transformation
\begin{align*}
x=-\frac{4\lambda z^2}{x_1'},y=\frac{-\mu x_1' -y_1' -x_1' z - x_1' z^2 +4\lambda z^3}{2 x_1' z},
\end{align*}
(\ref{eq:P3}) is transformed to (\ref{P3preKodairaanother}).
\end{proof}

\begin{prop}\label{propP3kodaira}
Suppose $(\lambda,\mu)\in\Lambda_3$. The elliptic surface given by {\rm(\ref{P3preKodairaanother})} has the singular fibres of type $I_9$ over $z=0$,   of type $I_{9}$ over $z=\infty$ and other six fibres of type $I_1$. 
\end{prop}

\begin{proof}
(\ref{P3preKodairaanother}) is described in the Kodaira normal form
\begin{eqnarray}\label{P3Kodaira3}
{y}_{2} ^ 2 = 4{x}_{2} ^ 3 - g_2(z){ x} _{2}- g_3(z),    \quad   z \not=\infty,
\end{eqnarray}
with
\begin{align*}
\begin{cases}
\vspace*{0.2cm}
&\displaystyle g_2(z)=-4\Big(-\frac{\mu^4}{3}-\frac{4\mu^3 z}{3} -2 \mu^2 z^2 - \frac{4\mu^3 z^2}{3}-\frac{4\mu z^3}{3} -8\lambda \mu z^3 -4 \mu^2 z^3 -\frac{z^4}{3}-8\lambda z^4\\
\vspace*{0.2cm}
&\displaystyle  \quad\quad\quad \quad\quad\quad \quad\quad\quad
 -4 \mu z^4- 2\mu^2 z^4- \frac{4z^5}{3}-8\lambda z^5 -4 \mu z^5 -2 z^6 -\frac{4 \mu z^6}{3} -\frac{4 z^7}{3}-\frac{z^8}{3}\Big) ,\\
\vspace*{0.2cm}
&\displaystyle g_3(z)=-4\Big(\frac{2\mu ^6}{27} +\frac{4 \mu ^5 z}{9} +\frac{10 \mu ^4 z^2}{9}+\frac{4 \mu ^5 z^2}{9}+ \frac{40 \mu^3 z^3}{27}+\frac{8 \lambda \mu^3 z^3}{3}+ \frac{20 \mu^4 z^3}{9}+ \frac{10 \mu^2 z^4}{9}\\\vspace*{0.2cm}
&\displaystyle \quad +8\lambda \mu^2 z^4 +\frac{40 \mu^3 z^4}{9}+ \frac{10 \mu^4 z^4}{9} +\frac{4 \mu z^5}{9} +8 \lambda\mu z^5 +\frac{40 \mu^2 z^5}{9} +8\lambda \mu^2 z^5 +\frac{40 \mu^3 z^5}{9} +\frac{2 z^6}{27} \\
&\displaystyle \quad\vspace*{0.2cm}
+ \frac{8 \lambda z^6}{3} +16 \lambda^2 z^6+\frac{20 \mu z^6}{9}+16 \lambda \mu z^6
 +\frac{20 \mu ^2 z^6}{3}+\frac{40 \mu^3 z^6}{27}+ \frac{4 z^7}{9}+8\lambda z^7 
 +\frac{40\mu z^7}{9}+8\lambda\mu z^7\\
 &\displaystyle \quad\quad\quad
  + \frac{10 z^8}{9} +8 \lambda z^8
  +\frac{40 \mu z^8}{9}+\frac{10 \mu^2 z^8}{9}+\frac{40 z^9}{27}+\frac{8\lambda z^9}{3}+\frac{20 \mu z^9}{9}+\frac{10 z^{10}}{9}+\frac{4\mu z^{10}}{9}+\frac{4 z^{11}}{9}+\frac{2 z^{12}}{27}\Big),
\end{cases}
\end{align*}
and
\begin{eqnarray}\label{P3Kodaira4}
{y}_3 ^ 2 = 4{x}_3 ^ 3 - h_2(z_1) {x}_3 - h_3(z_1),    \quad   z_1 \not=\infty,
\end{eqnarray}
with
\begin{align*}
\begin{cases}
\vspace*{0.2cm}
&\displaystyle h_2(z_1)=-4\Big(-\frac{1}{3}-\frac{4z_1}{3} -2 z_1 - \frac{4\mu z_1^2}{3}-\frac{4z_1^3}{3} -8\lambda z_1^3 -4 z_1^3 -\frac{z_1^4}{3}-8\lambda z_1^4\\
\vspace*{0.2cm}
&\displaystyle  \quad\quad\quad \quad\quad\quad \quad\quad\quad\quad\quad\quad
 -4\mu z_1^4 - 2\mu^2 z_1^2 -\frac{4\mu z_1^5}{3}-8\lambda\mu z_1^5-4\mu ^2 z_1^5 -2\mu^2 z_1^6 -\frac{4\mu ^3 z_1^6}{3}-\frac{4\mu^3 z_1^7}{3}-\frac{\mu^4 z_1^8}{3}\Big) ,\\
\vspace*{0.2cm}
&\displaystyle h_3(z_1)=-4\Big(\frac{2}{27}+\frac{4z_1}{9}+\frac{10 z_1^2}{9}+\frac{4\mu z_1^2}{9}+\frac{40 z_1^3}{27}+\frac{8\lambda z_1^3}{3}+\frac{20\mu z_1^3}{9}+\frac{10 z_1^4}{9}+8\lambda z_1^4\\
\vspace*{0.2cm}
&\displaystyle \quad\quad
\quad
 +\frac{40 \mu z_1^4}{9}+\frac{10 \mu^2 z_1^4}{9}+\frac{4 z_1^5}{9}+8\lambda z_1^5+\frac{40 \mu z_1^5}{9}+8\lambda\mu z_1^5+\frac{40 \mu^2 z_1^5}{9}+\frac{2 z_1^6}{27}
+\frac{8\lambda z_1^6}{3}+16\lambda^2 z_1^6 
\\
&\displaystyle \quad\quad\quad\vspace*{0.2cm}+\frac{20 \mu z_1^6}{9}+16 \lambda \mu z_1^6 +\frac{20 \mu^2 z_1^6}{3} + \frac{40 \mu^3 z_1^6}{27}
+\frac{4\mu z_1^7}{9}+8\lambda \mu^2 z_1^7 +\frac{40\mu^3 z_1^7}{9}+\frac{10 \mu^2 z_1^8}{9}
+8\lambda \mu^2 z_1^8\\
&\displaystyle \quad\quad\quad
+\frac{10 \mu^4 z_1^8}{9} +\frac{40 \mu^3 z_1^8}{9}+\frac{40 \mu^3 z_1^9}{27}+\frac{8\lambda \mu^3 z_1^9}{3}+\frac{20 \mu^4 z_1^9 }{9}+\frac{10\mu^4 z_1^{10}}{9}+\frac{4\mu^5 z_1^{10}}{9}+\frac{4\mu^5 z_1^{11}}{9}+\frac{2 \mu^6 z_1^{12}}{27} \Big),
\end{cases}
\end{align*}
where $z=1/z_1$
We have the discriminant of the right hand side of (\ref{P3Kodaira3}) for ${x'}_2$((\ref{P3Kodaira4})  for ${x'}_3$, resp.):
\begin{align*}
\begin{cases}
&D_0=256 \lambda^3 z^9 (\mu^3 + 3\mu ^2 z +3 \mu z^2 + 3 \mu^2 z^2 +z^3 +27 \lambda z^3 +6\mu z^3 +3z^4 + 3\mu z^4 +3z^5 +z^6),\\
&D_\infty=256 \lambda^3 z_1^9(1 +3 z_1 + 3 z_1^2 +3 \mu z_1^2 + z_1^3+ 27 \lambda z_1^3 +6\mu z_1^3 + 3\mu z_1^4 +3\mu^2 z_1^4 + 3\mu^2 z_1^5+\mu^3 z_1^6).
\end{cases}
\end{align*}
From these data, we obtain the required statement.
\end{proof}

This fibration is illustrated in  Figure 6.

\begin{figure}[h]
\center
\includegraphics[scale=0.9]{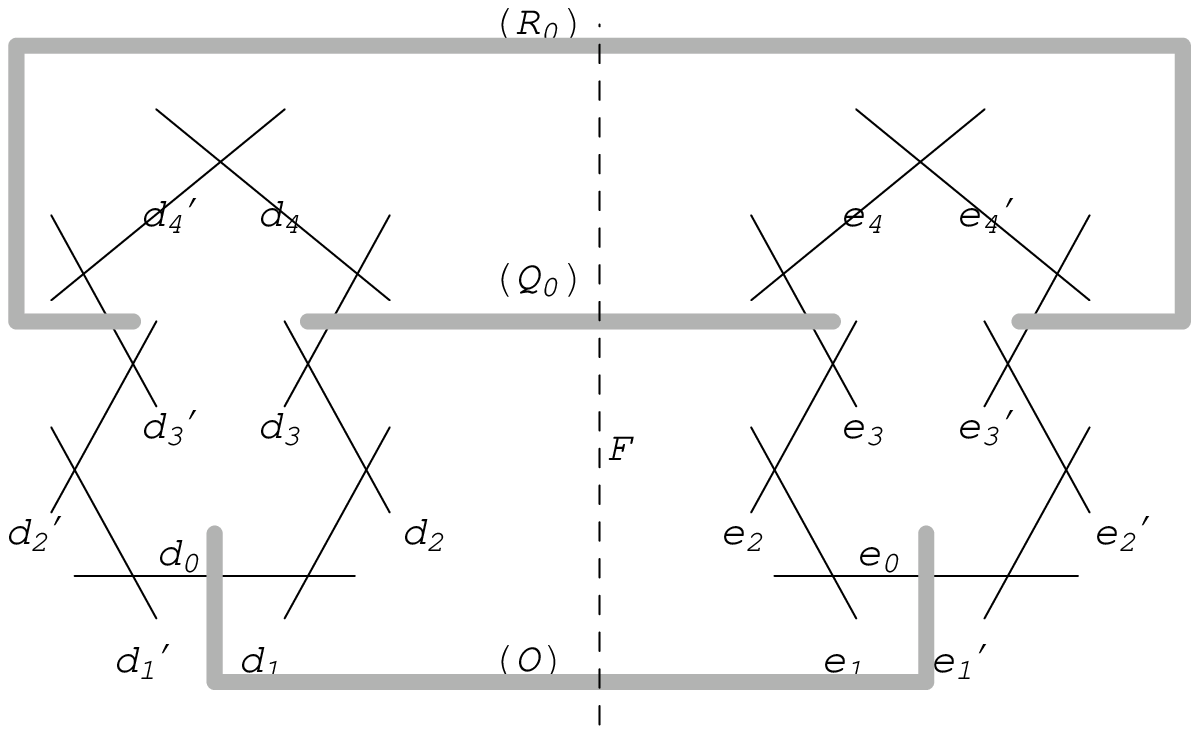}
\caption{}
\end{figure}

For this fibration, let $O$ be the zero of the Mordell-Weil group,  $Q_0$ and $R_0$  be the sections in (\ref{P3sectionanother}) and $F$ be a general fibre.
Set 
\begin{eqnarray}\label{L3definition}
L_3 =\langle d_1,d_2,d_3,d_4,d_4',d_3',d_2',d_1',e_1,e_2,e_3,e_4,e_3',e_2',O,Q_0,R_0,F \rangle_\mathbb{Z}.
\end{eqnarray}
We have the following intersection matrix $M_3$ for $L_3$:
\begin{align}\label{intersection L3}\notag
M_3=&A_{18}(-1) +2E_{18,18}
 - (E_{8,9}+E_{9,8}) - (E_{14,15}+E_{15,14})- (E_{12,13}+E_{13,12})\\
\notag&+(E_{3,16}+E_{16,3})+(E_{6,17}+E_{17,6})+(E_{11,16}+E_{16,11})+(E_{13,17}+E_{17,13})\\
&-(E_{15,16}+E_{16,15})+(E_{15,18}+E_{18,15})+(E_{16,18}+E_{18,16})-(E_{16,17}+E_{17,16})\end{align}
We have ${\rm det}(M_3)=-9$.

\section{The Picard numbers}

For a general $K3$ surface $S$,  the homology group  $H_2(S,\mathbb{Z})$ is a free $\mathbb{Z}$-module of rank $22$.
The intersection form of $H_2(S,\mathbb{Z})$ is given by 
\begin{eqnarray}\label{K3intersection}
E_8(-1)\oplus E_8(-1)\oplus U\oplus U\oplus U,
\end{eqnarray}
where 
\begin{align*}
&E_8(-1)=\left(
\begin{array}{cccccccc}
-2 & 1 & & & & & & \\
1 & -2 & 1 & & &{O} & & \\
 & 1 & -2 & 1 & & & & \\
 & &1&-2&1& & & \\
 & & &1 &-2&1&1&  \\
 & & & & 1&-2 &0 &  \\
 & & { O}& & 1&0 &-2 &1 \\
 & & & & & & 1&-2 
\end{array} 
\right), \quad\quad\quad
U=\begin{pmatrix}
0&1\\
1&0
\end{pmatrix}.
\end{align*}
Let ${\rm NS} (S)$,  the N\'{e}ron-Severi lattice of $S$, denote the sublattice in $H_2(S,\mathbb{Z})$ generated by the divisors on $S$.
 We call the rank of ${\rm NS}(S)$ the Picard number.
The orthogonal complement  ${\rm Tr}(S)={\rm NS}(S)^\perp$ in $H_2(S,\mathbb{Z})$ is called the transcendental lattice of $S$.
Let $(C \cdot D)$ be  the intersection number of $2$-cycles $C$ and $D$  on $S$.

In this section, 
we define the period mappings and determine the Picard numbers for our families. 
We state the precise argument only for the case of the family $\mathcal{F}:=\mathcal{F}_0$ of the $K3$ surfaces $S(\lambda,\mu):=S_0(\lambda,\mu)$.

\subsection{S-marked $K3$ surfaces}

The lattice $L:=L_0$ in (\ref{N-Sbasis}) is contained in ${\rm NS}(S(\lambda,\mu))$ and of rank $18$. So we have

\begin{prop}\label{prop:Picardineq}
$$
{\rm  rank \enspace NS}(S(\lambda ,\mu )) \geq  18.
$$
\end{prop}

We have also

\begin{prop}\label{prop:primitive}
$L$ is a primitive sublattice of $H_2(S(\lambda,\mu),\mathbb{Z})$.
\end{prop}

\begin{proof}
 By  (\ref{-5 matrix}),  we have ${\rm det}(L)=-5$.
 It does not contain any square factor. So $L$ is primitive.
 \end{proof}

\begin{df}\label{Definition2.2}
 For a $K3$ surface $S(\lambda,\mu)$ $((\lambda,\mu)\in \Lambda)$, set 
\begin{align*}
\begin{cases}
&\gamma_5=b_1,\hspace{1mm}\gamma_6=b_2,\hspace{1mm}\gamma_7=b_3,\hspace{1mm}\gamma_8=b_4,\hspace{1mm}\gamma_9=b_5,\hspace{1mm}\gamma_{10}=Q,\\
&\gamma_{11}=b_6,\hspace{1mm}\gamma_{12}=b_7,\hspace{1mm}\gamma_{13}=b_1',\hspace{1mm}\gamma_{14}=b_2',\hspace{1mm}\gamma_{15}=b_3',\hspace{1mm}\gamma_{16}=b_4',\\
&\gamma_{17}=b_5',\hspace{1mm}\gamma_{18}=R,\hspace{1mm}\gamma_{19}=b_6',\hspace{1mm}\gamma_{20}=b_7',\hspace{1mm}\gamma_{21}=O,\hspace{1mm}\gamma_{22}=F,
\end{cases}
\end{align*}
given by {\rm (\ref{N-Sbasis})}.
Let $\check{S}=S(\lambda_0,\mu_0)$ be a reference surface for a fixed point $(\lambda_0,\mu_0)\in \Lambda=\Lambda_0$. Set  $\check{L}=L(\lambda_0,\mu_0) \subset H_2(\check{S},\mathbb{Z})$.
We define a S-marking $\psi$ of $S(\lambda,\mu)$ to be an isomorphism
$
\psi :H_2(S(\lambda,\mu ),\mathbb{Z})\rightarrow \check{L}
$
with the property that $\psi^{-1}(\gamma _j) = \gamma _j$ for $5\leq j\leq 22$. We call the pair $(S(\lambda,\mu),\psi)$ an S-marked $K3$ surface.
\end{df}

By the above definition, a S-marking $\psi$ has the  property:
\begin{align*}
&\psi ^{-1}(F)=F, \hspace{1mm}\psi^{-1}(O)=O, \hspace{1mm}\psi^{-1}(Q)=Q,\hspace{1mm} \psi^{-1}(R)=R,\\
&\psi^{-1}(b_j)=b_j, \hspace{1mm}\psi^{-1}(b'_j)=b'_j  {\rm \quad} (1\leq j\leq 7).
\end{align*}

\begin{df}\label{df:iso,equi}
Two $S$-marked $K3$ surfaces $(S,\psi )$ and $(S',\psi')$ are said to be 
isomorphic if there is a biholomorphic mapping $f:S\rightarrow S'$ with 
$$
\psi' \circ f_{\ast} \circ \psi^{-1} ={\rm id}_{H_2(\check{S},\mathbb{Z})}.
$$
Two $S$-marked $K3$ surfaces $(S,\psi )$ and $(S',\psi')$ are said to be 
equivalent if there is a biholomorphic mapping $f:S\rightarrow S'$ with 
$$
\psi' \circ f_{\ast} \circ \psi^{-1}|_{ \check{L}}  ={\rm id} _{\check{L}}.
$$
\end{df}

By Proposition \ref{prop:primitive}, the basis $\{\gamma_5,\cdots,\gamma_{22}\}$ of $L(\subset H_2(S(\lambda,\mu),\mathbb{Z}))$  is extended to a basis 
\begin{eqnarray}\label{basisL}
\{\gamma_1,\gamma_2,\gamma_3,\gamma_4,\gamma_5,\cdots, \gamma_{22}\}
\end{eqnarray}
 of $H_2(S(\lambda,\mu),\mathbb{Z})$.
Let $\{\gamma_1^*, \cdots , \gamma_{22}^*\}$ be the dual basis of $\{\gamma_1,\cdots,\gamma_{22}\}$ with respect to the intersection form (\ref{K3intersection}). 
Set
\begin{equation} \label{df:L''}
L_t=\langle\gamma^*_1,\gamma^*_2,\gamma^*_3,\gamma^*_4\rangle_\mathbb{Z} \subset H_2(S(\lambda,\mu),\mathbb{Z}).
\end{equation}
We have  $L_t = L^\perp $.

\subsection{Period mapping}
First, we state the definition of the period mapping for general $K3$ surfaces.

For a $K3$ surface $S$,
there exists unique holomorphic 2-form $\omega$ up to a constant factor.
 Let $\{\gamma_1,\cdots \gamma_{22}\}$ be a basis of $H_2(S,\mathbb{Z})$.
$$
\eta' = \Big(\int_{\gamma_1} \omega:\cdots : \int_{\gamma_{22}} \omega \Big)\in \mathbb{P}^{21}(\mathbb{C})
$$ 
is called a period of $S$. 
Let $\{\gamma_1,\cdots,\gamma_r\}$ be a basis of ${\rm Tr}(S)$. Note that 
\begin{eqnarray}\label{neronseverivanish}
\int _\gamma \omega =0, \quad\quad({}^\forall \gamma \in {\rm NS}(S)). 
\end{eqnarray}
The period $\eta'$ is  reduced to 
$$
\eta = \Big(\int_{\gamma_1} \omega:\cdots:\int_{\gamma_r} \omega\Big)\in \mathbb{P}^{r-1}(\mathbb{C}).
$$
We note that ${\rm NS}(S)$ is a lattice of signature $(1,\cdot)$ and ${\rm Tr}(S)$ is a lattice of the signature $(2,\cdot)$.

\begin{df}\label{df:periodmap}
Let $\check{S}=S(\lambda_0,\mu_0)$ be the reference surface. Take a small neighborhood $\delta $ 
of $(\lambda_0,\mu_0)$ in $\Lambda$ so that we have a local topological  trivialization 
$$
\tau : \{ S(\lambda ,\mu)|(\lambda ,\mu)\in \delta \} \rightarrow \check{S} \times \delta .
$$
Let $p: \check{S} \times \delta \rightarrow \check{S}$ be the canonical projection, and set $r=p\circ \tau$.
Then, 
$$r'(\lambda,\mu)=r|_{S(\lambda,\mu)}$$
 gives a deformation of surfaces.
We note that $r'$ preserves the lattice $L$.
Take an S-marking $\check{\psi} $ of $\check{S}$.  We obtain the $S$-markings of $S(\lambda,\mu)$ 
by $\psi =\check{\psi}\circ r'_*$ for $(\lambda ,\mu )\in \delta$.
We define the local period mapping $\Phi=\Phi_0 :\delta \rightarrow \mathbb{P}^3(\mathbb{C})$ by
\begin{eqnarray}
&&
\Phi ((\lambda ,\mu))=\Big(\int_{\psi^{-1}(\gamma_1)} \omega: \ldots: \int_{\psi^{-1}(\gamma_4)} \omega \Big),
\end{eqnarray}
where $\gamma_1,\cdots,\gamma_4 \in L$ are given by {\rm (\ref{basisL})}.
We define the multivalued period mapping $\Lambda \rightarrow \mathbb{P}^3(\mathbb{C})$ by making the 
analytic continuation of $\Phi$ along any arc starting from $(\lambda_0,\mu_0)$ in $\Lambda$.
\end{df}

In general, we have the Riemann-Hodge relation for the period:
\begin{align*}
\eta' M {}^t\eta'=0, \quad\quad\quad
\eta' M {}^t\bar{\eta'}>0,
\end{align*} 
where $M$ is the intersection matrix $(\gamma_j^*\cdot\gamma_k^*)_{1\leq j,k\leq 22}$.

For our case,
according to the relation (\ref{neronseverivanish}),  the Riemann-Hodge relation is reduced to

\begin{align}
\eta A {} ^t \eta
=0,\quad\quad\quad
\eta A {}^t\bar{\eta } >0,
\end{align}
where
\begin{align*}
A=
(\gamma_j^* \cdot \gamma_k^*)_{1\leq j,k\leq 4}
\end{align*}
and
$$
\eta=\Big(
\int_{\psi^{-1}(\gamma _1)} \omega: \int_{\psi^{-1}(\gamma_2)} \omega: \int_{\psi^{-1}(\gamma_3 )} \omega: \int_{\psi^{-1}(\gamma _4)}\omega
\Big).
$$

\begin{rem}
In {\rm Theorem \ref{thm:intersection}},
we shall show  that the above matrix $A$ is given by
\begin{eqnarray*}
A=A_0=\begin{pmatrix}
0&1&0&0\\
1&0&0&0\\
0&0&2&1\\
0&0&1&-2
\end{pmatrix}.
\end{eqnarray*}
\end{rem}

Set 
$$
\mathcal{D}=\mathcal{D}_0=\{\xi=(\xi_1:\xi_2:\xi_3:\xi_4) \in\mathbb{P}^3(\mathbb{C}) \text{ }|\text{ } \xi A {}^t\xi =0, \xi A {}^t \bar{\xi} >0\}.
$$
We have $\Phi(\Lambda)\subset\mathcal{D}$. Note that $\mathcal{D}$ is composed of two connected components.
Let $\mathcal{D}^+$ be the  component where  $(1:1:-\sqrt{-1}:0)$ is a point of $\mathcal{D}^+$.  And let $\mathcal{D}^-$ be the other component.

\begin{df}\label{Def:monodromy}
The fundamental group $\pi _1(\Lambda,*)$ acts on the $\mathbb{Z}$-module $\langle\psi^{-1}(\gamma_1),\cdots,\psi^{-1}(\gamma_4)\rangle_\mathbb{Z}$. So, it induces the action on $\mathcal{D}$. This action induces a group of  projective linear transformations which is a subgroup of $PGL(4,\mathbb{Z})$. We call it the projective monodromy group of the period mapping $\Phi:\Lambda\rightarrow\mathcal{D}$.
\end{df}

\subsection{The Picard number}

By the Riemann-Roch Theorem for surfaces and the Serre duality, we have the following lemma:

\begin{lem}\label{lem:K3fibre}
Let $S$ be a $K3$ surface with the elliptic fibration $\pi : S \rightarrow \mathbb{P}^1(\mathbb{C})$ and  $F$ be a fixed general fibre. Then, $\pi$ is the unique elliptic fibration up to ${\rm Aut}(\mathbb{P}^1(\mathbb{C}))$ which has $F$ as a general fibre. 
\end{lem}

\begin{df}
Let $(S_1,\pi_1,\mathbb{P}^1(\mathbb{C}))$ and  $(S_2,\pi_2,\mathbb{P}^1(\mathbb{C}))$ be  two elliptic surfaces. If  there exist a biholomorphic mapping $f : S_1\rightarrow S_2$ and $\varphi \in {\rm Aut}(\mathbb{P}^1(\mathbb{C}))$ such that $\varphi \circ \pi_1 = \pi_2 \circ f$, we say  $(S_1,\pi_1,\mathbb{P}^1(\mathbb{C}))$ and $(S_2,\pi_2,\mathbb{P}^1(\mathbb{C}))$ are isomorphic as elliptic surfaces.
\end{df}

For an elliptic surface given by the Kodaira normal form $y^2 = 4 x^3 -g_2(z)x -g_3(z)$, we  define the $j$-invariant  {\rm ( see \cite{Kodaira2} Section 7)}:
\begin{eqnarray}\label{eq:j-function}
j(z) = \frac{g^3_2(z)}{g^3_2(z) -27g^2_3(z)} \in \mathbb{C}(z).
\end{eqnarray}
From the definition, we have

\begin{prop}\label{prop:j-inv}
Let $(S_1,\pi_1 ,\mathbb{P}^1(\mathbb{C}))$ and $(S_2, \pi_2,\mathbb{P}^1(\mathbb{C}))$ be two elliptic surfaces given by the Kodaira normal forms.
Let $j_1(z)$ and $j_2(z)$ be  the corresponding  $j$-invariants of the Kodaira normal forms, respectively. If $(S_1,\pi_1 ,\mathbb{P}^1(\mathbb{C}))$ and  $(S_2, \pi_2,\mathbb{P}^1(\mathbb{C}))$ are isomorphic,  then there exists $\varphi \in {\rm Aut}(\mathbb{P}^1(\mathbb{C}))$ such that  $\pi_1 ^{-1}(p)$ and $\pi_2^{-1}(\varphi(p))$ are the fibres of the same type for any $p\in\mathbb{P}^1(\mathbb{C})$ and $j_2\circ \varphi = j_1$.
\end{prop}

For $(\lambda,\mu)\in\Lambda$, let
$$
\pi: S(\lambda,\mu)\rightarrow \mathbb{P}^1(\mathbb{C})=(\text{$z$-sphere})
$$
be the canonical elliptic fibration given by (\ref{eq:P4prekodaira}). 

\begin{lem}\label{lem:elliptic1-1}
Suppose $(\lambda_1,\mu_1 ) ,(\lambda_2 ,\mu_2 )\in \Lambda$. If $(S(\lambda_1,\mu_1),\pi_1 ,\mathbb{P}^1(\mathbb{C}))$ is isomorphic to $(S(\lambda_2,\mu_2),\pi _2,\mathbb{P}^1(\mathbb{C}))$ as elliptic surfaces, then it holds $(\lambda _1,\mu_1)=(\lambda_2,\mu_2)$.
\end{lem}

\begin{proof}
Let $f :S_1 \rightarrow  S_2$ be the biholomorphic mapping which gives the equivalence of elliptic surfaces. According to Proposition \ref{prop:j-inv}, there exists $\varphi \in \text{Aut}(\mathbb{P}^1(\mathbb{C}))$ which satisfies $\varphi \circ \pi _1=\pi _2 \circ f$. By Proposition \ref{prop1.2},  
we have $\pi_j^{-1}(0)=I_3$ and $\pi_j^{-1}(\infty)=I_{15}$ ($j=1,2$). 
So, $\varphi$ has the form $\varphi : z\mapsto az$ with some $a\in \mathbb{C}-{0}$.
Let $D_0(z;\lambda_j,\mu_j)$ $(j=1,2)$ be the discriminant. 
From (\ref{eq:P4discriminant}), we have
\begin{eqnarray*}
\frac{D_0(z; \lambda_j, \mu_j)}{64 \mu_j ^3 z^3} =\lambda_j^3 +3\lambda_j^2 z +27\mu_j z +3\lambda_j z^2 +3\lambda_j^2 z^2 +z^3 +6 \lambda_j z^3 +3z^4 +3\lambda_j z^4 +3z^5 +z^6 \hspace{3.5mm}  (j=1,2).
\end{eqnarray*} 
The six roots of $D_0(z; \lambda_1, \mu_1)/64 \mu_1 ^3 z^3$ ($D_0(z; \lambda_2, \mu_2)/64 \mu_2 ^3 z^3$, resp.) give the six images of  singular fibres of type $I_1$ of $S(\lambda_1,\mu_1)$ ($S(\lambda_2,\mu_2)$, resp.). The roots of $D_0(z;\lambda_1,\mu_1)/64 \mu_1^3 z^3$ are sent by $\varphi$ to those of  $D_0(z;\lambda_2,\mu_2)/64 \mu_2^3 z^3$.
Observing the coefficients of $D_0(z;\lambda_1,\mu_1)$ and $D_0(z;\lambda_2,\mu_2)$, we obtain that $a=1$.
Therefore, we have $(\lambda_1,\mu_1)=(\lambda_2,\mu_2)$. 
\end{proof}

\begin{prop}\label{lem:S-marked}
Two S-marked $K3$ surfaces $(S(\lambda_1,\mu_1) ,\psi_1)$ and $(S(\lambda_2,\mu_2), \psi_2) $ are equivalent if and only if there exists an isomorphism of elliptic surfaces between
$(S(\lambda_1,\mu_1),\pi_1,\mathbb{P}^1(\mathbb{C}))$ and $ (S(\lambda_2,\mu_2),\pi_2,\mathbb{P}^1(\mathbb{C}))$.
\end{prop}

\begin{proof}
The sufficiency is clear. We prove the necessity. 
Let $(\lambda_1,\mu_1),(\lambda_2,\mu_2)\in \Lambda$. Suppose the equivalence of S-marked $K3$ surfaces
$$
(S(\lambda _1,\mu_1) ,\psi_1)\simeq (S(\lambda_2 ,\mu _2),\psi_2).
$$
Then, there exists a biholomorphic mapping $f : S(\lambda_1,\mu_1)\rightarrow  S(\lambda_2,\mu_2)$ such that $\psi_2\circ f_*\circ \psi_1^{-1} | _{L}=id_{L}$. Especially, for general fibres $F_1 \in \text{Div}(S_1)$ and $F_2 \in \text{Div}(S_2)$, we have  $f_*(F_1)=F_2$.

So, $S(\lambda_2,\mu_2)$ has two elliptic fibrations $\pi_2$ and $\pi_1 \circ f^{-1}$ which have a general fibre $F_2$. According to Lemma \ref{lem:K3fibre},  it holds
$$
\pi_2 = \pi_1 \circ f^{-1}
$$
up to $\text{Aut}(\mathbb{P}^1(\mathbb{C}))$.
\end{proof}

\begin{cor}\label{prop:S-marked}
Let $(\lambda_1,\mu_1)$ and $(\lambda_2,\mu_2)$ be in $\Lambda$. 
Two S-marked $K3$ surfaces $(S(\lambda_1,\mu_1), \psi_1)$ and $(S(\lambda_2,\mu_2), \psi_2)$ are equivalent  if and only if 
$(\lambda_1,\mu_1)=(\lambda_2,\mu_2)$.
\end{cor}

\begin{proof}
From the proposition and Lemma \ref{lem:elliptic1-1}, we obtain the required statement.
\end{proof}

\begin{thm}\label{thm:Torelli}{\rm (The local Torelli theorem for S-marked $K3$ surfaces)}
Let $\delta\subset\Lambda$ be a sufficiently small neighborhood of $(\lambda_0,\mu_0)$, and $(\lambda_1,\mu_1),$ $(\lambda_2,\mu_2)\in\delta$.
Suppose $\Phi(\lambda_1,\mu_1)=\Phi(\lambda_2,\mu_2)$, then there exists an isomorphism of S-marked $K3$ surfaces $(S(\lambda_1,\mu_1),\psi_1)\simeq (S(\lambda_2,\mu_2),\psi_2)$.
\end{thm}

We have

\begin{thm}\label{thm:P4Picard}
For a generic point $(\lambda , \mu ) \in \Lambda$, we have
$$
{\rm rank \enspace NS}(S(\lambda , \mu))=18.
$$
\end{thm}

\begin{proof}
By Proposition \ref{prop:Picardineq}, we already have 
${\rm rank  \enspace NS}(S(\lambda,\mu))\geq 18$.   
Let $\delta$ be a small neighborhood of $(\lambda,\mu)$. 
Suppose we have ${\rm rank \enspace  NS}(S(\lambda',\mu'))>18$ for all $(\lambda',\mu')\in \delta$. Then, $\Phi(\delta)$ cannot contain any open set of $\mathcal{D}$.
By  Corollary \ref{prop:S-marked} and  Theorem \ref{thm:Torelli}, the period mapping is injective.
This is a contradiction.
\end{proof}

\begin{cor}\label{cor:dim}
The $\mathbb{C}$-vector space generated by the germs of holomorphic functions 
$$
\int_{\psi^{-1} (\gamma_1)}\omega, \cdots, \int_{\psi^{-1}(\gamma_4)} \omega
$$
is $4$-dimensional.
\end{cor}

\begin{proof}
It is clear, for  the rank of the transcendental lattice ${\rm Tr}(S(\lambda,\mu))$ is $22-18=4$.
\end{proof}

\vspace{5mm}

We can determine the Picard number of the family $\mathcal{F}_j$ $(j=1,2,3)$ by the same method.
Recall the lattice $L_1$ ($L_2,L_3,$ resp.) 
in (\ref{L1definition}) ((\ref{L2definition}), (\ref{L3definition}), resp.)
for $\mathcal{F}_1$ ($\mathcal{F}_2,\mathcal{F}_3, resp.$). 
Set $j\in \{1,2,3\}$.
Let $\{\gamma_1^*,\cdots,\gamma_{22}^*\}$ be a basis of $H_2(S_j(\lambda,\mu),\mathbb{Z})$ such that we have $\langle \gamma_1^*,\cdots,\gamma_4^*\rangle_\mathbb{Z} =L_j^\perp$.
Take a dual basis $\{\gamma_1,\cdots,\gamma_{22}\}$ of $H_2(S_j(\lambda,\mu),\mathbb{Z})$, namely it holds $(\gamma_j\cdot\gamma_k^*)=\delta_{jk}$ $(1\leq j,k \leq22)$. 
By the same procedure as for $\mathcal{F}_0$,
we define the multivalued analytic period mapping $\Phi_j:\Lambda_j \rightarrow \mathcal{D}_j$ given by
$$
(\lambda,\mu) \mapsto \Big(\int_{\gamma_1} \omega_j :\cdots:\int_{\gamma_4} \omega_j \Big),
$$
where $\omega_j$ is the unique holomorphic $2$-form on $S_j(\lambda,\mu)$ up to a constant factor and 
$\mathcal{D}_j$ is the domain of type $IV$ defined by the intersection matrix $(\gamma_j^*\cdot\gamma_k^*)_{1\leq j,k\leq4}$.
Moreover,  we have  the Kodaira normal forms of the elliptic fibrations (\ref{P1preKodaira}), (\ref{P2preKodaira}) and (\ref{P3preKodairaanother}) (these forms appear in the proofs of Proposition \ref{propP1kodaira}, \ref{propP2kodaira} and \ref{propP3kodaira}). 
Observing the coefficients of these forms, we can prove the  lemmas corresponding to Lemma \ref{lem:elliptic1-1}.
Therefore, we obtain the following theorem.

\begin{thm}\label{18theorem}
The Picard number of a generic member of the families $\mathcal{F}_j$ $(j=1,2,3)$ are equal to 18.
\end{thm}

\section{The N\'{e}ron-Severi lattices}
For our further study, 
we need the explicit lattice structures of the N\'{e}ron-Severi lattices and those of the transcendental lattices.
  In this section, we show the following theorem.

\begin{thm}\label{latticeThm}
The intersection matrices of N\'{e}ron-Severi lattices ${\rm NS}$
and the transcendental lattices ${\rm Tr} $
of a generic member of  $\mathcal{F}_j$ $(j=0,1,2,3)$
are given
 as in {\rm Table 2}.
\begin{center}
\begin{tabular}{llcc}
\toprule
Polytope& Family & {\rm NS} &{\rm Tr}  \\
\midrule
\vspace{2mm}
$P_0$& $\mathcal{F}_0$ & $E_8(-1)\oplus E_8(-1)\oplus \begin{pmatrix}2&1\\ 1&-2\end{pmatrix}$  & \quad\quad$U\oplus \begin{pmatrix} 2&1\\1&-2 \end{pmatrix}=:A_0$ \\
\vspace{2mm}
$P_1$&$\mathcal{F}_1$  & $E_8(-1)\oplus E_8(-1)\oplus \begin{pmatrix}0&3\\3&0\end{pmatrix}$ &\quad\quad $U \oplus \begin{pmatrix}0&3\\ 3&0 \end{pmatrix}=:A_1$ \\
\vspace{2mm}
$P_2$&$\mathcal{F}_2$  & $E_8(-1)\oplus E_8(-1)\oplus \begin{pmatrix}0&3\\3&2\end{pmatrix}$ & \quad\quad$U \oplus \begin{pmatrix}0&3\\ 3&-2 \end{pmatrix}=:A_2$ \\
$P_3$&$\mathcal{F}_3$  & $E_8(-1)\oplus E_8(-1)\oplus \begin{pmatrix}0&3\\3&-2\end{pmatrix}$ & \quad\quad $U \oplus \begin{pmatrix}0&3\\ 3&2 \end{pmatrix}=:A_3$ \\
\bottomrule
\end{tabular}
\end{center}
\begin{center}
{\bf Table 2.}
\end{center}
\end{thm}

\begin{rem}\label{mirrorRem}
Koike {\rm \cite{Koike}}  made a research on the   families of $K3 $ surfaces  derived from the dual polytopes of 3-dimensional Fano polytopes.
The polytopes $P_0$, $P_2$ and $P_3$ in our notation  are the  Fano polytopes.
Due to Koike, we have  N\'{e}ron-Severi lattices for  the dual polytopes $P_0^\circ$, $P_2^\circ$ and $P_3^\circ$ (given by  {\rm Table 3}).

 \begin{center}
\begin{tabular}{lcc}
\toprule
Dual Polytope  &{\rm NS} &{\rm Tr}\\
\midrule 
\vspace{2mm}
$P_0^\circ$ & $\begin{pmatrix} 2&1\\ 1&-2\end{pmatrix}$ & \quad\quad $U\oplus E_8(-1)\oplus E_8(-1)\oplus \begin{pmatrix}2&1\\ 1&-2\end{pmatrix}$\\
\vspace{2mm}
$P_2^\circ$ & $\begin{pmatrix} 0&3\\ 3&-2\end{pmatrix}$ & \quad\quad $U\oplus E_8(-1)\oplus E_8(-1)\oplus \begin{pmatrix}0&3\\ 3&2\end{pmatrix}$\\
$P_3^\circ$ & $\begin{pmatrix} 0&3\\ 3&2\end{pmatrix}$ & \quad\quad $U\oplus E_8(-1)\oplus E_8(-1)\oplus \begin{pmatrix}0&3\\ 3&-2\end{pmatrix}$\\
\bottomrule
\end{tabular}
\end{center}
\begin{center}
{\bf Table 3.}
\end{center}
 {\rm    Table 3} and {\rm  Table 2} support   the mirror symmetry  conjecture for the reflexive polytopes $P_0$, $P_2$ and $P_3$.
\end{rem}

\begin{rem}
According to the above theorem, 
a generic member of  $\mathcal{F}_j$ $(j=0,1,2,3)$ has  
the Shioda-Inose structure.
{\rm (}see Morrison {\rm \cite{Morrison}, Theorem 6.3} {\rm )}.
\end{rem}

\begin{rem}
 The N\'{e}ron-Severi lattices  of $K3$ surfaces with non-symplectic involutions are studied    by  Nikulin {\rm\cite{Nikulin}}.
All of our cases are  not contained in his results.
The lattice structures of  $95$ weighted projective $K3$ surfaces given by M. Reid are studied by Belcastro  {\rm \cite{Belcastro}}. 
Our lattice of $\mathcal{F}_0$  coincides with No.30 and No.86 in her list.
Our lattices of $\mathcal{F}_1, \mathcal{F}_2$ and $\mathcal{F}_3$ are not contained in her results, neither. 
\end{rem}

\subsection{Proof for the case $P_0$}
We  prove Theorem \ref{latticeThm} for the case $P_0$ in a naive way.
Recall the lattice $L_0$ in (\ref{N-Sbasis}).
By Theorem \ref{thm:P4Picard}, for generic $(\lambda,\mu)\in\Lambda_0$, $${\rm dim}( {\rm NS}(S_0(\lambda,\mu))) = 18 ={\rm dim}(L_0).$$
According to Proposition \ref{prop:primitive}, we have $(L\otimes_\mathbb{Z} \mathbb{Q})\cap{\rm NS}(S_0(\lambda,\mu)) =L_0$.
Hence, we have 
$$
{\rm NS}(S_0(\lambda,\mu))=L_0
$$
for generic $(\lambda,\mu)\in\Lambda_0$.

\begin{lem}\label{thm:intersection}
The lattice $L_0$ is isomorphic to the lattice given by the intersection matrix
\begin{eqnarray*}\label{matrix:P4neronseveri}
M'_0=E_8(-1)\oplus E_8(-1)\oplus \begin{pmatrix} 2 &1\\1&-2\end{pmatrix},
\end{eqnarray*}
and its orthogonal complement  is given by
\begin{eqnarray*}\label{eq:P4transcendental}
A_0=U \oplus
\begin{pmatrix}2 &1 \\
 1&-2
 \end{pmatrix}.
\end{eqnarray*}
\end{lem}

\begin{proof}
Let $M_0$ be the matrix given in (\ref{matrix:P4preintersection}). 
Set
\begin{eqnarray}\label{unitv}
r_j={}^t (0,\cdots,0,\overbrace{1}^{{\rm j-th}},0,\cdots,0)\qquad(1\leq j\leq18)
\end{eqnarray}
and 
\begin{align*}
\begin{cases}
&v_{16}={}^t (-1,-2,-3,-4,-5,-2,-4,-3,-1,-2,-3,-4,-5,-2,-4,-2,1,1),\\
&v_{17}={}^t (5,10,15,20,25,13,17,9,1,2,3,4,5,3,3,1,1,-3),\\
&v_{18}={}^t (-2,-4,-6,-8,-10,-6,-6,-2,0,0,0,0,0,-1,1,2,-2,1).
\end{cases}
\end{align*}
Set
$$
U=(r_1,r_2,r_3,r_4,r_5,r_6,r_7,r_8,r_9,r_{10},r_{11},r_{12},r_{13},r_{14},r_{15},v_{16},v_{17},v_{18}).
$$
This is an unimodular matrix.
Then, we have ${}^t U M_0 U =M'_0$.
By observing $E_8(-1)\oplus E_8(-1) \oplus U \oplus U \oplus U$ and $M'_0$, we obtain the matrix $A_0$.
\end{proof}

Therefore, we obtain Theorem \ref{latticeThm} for $P_0$.

\subsection{The Mordell-Weil group of sections}

In the case  $\mathcal{F}_j$ $(j=1,2,3)$, we have $|{\det(L_j)}|=9$, which is a square number.
For this reason,  we cannot apply the method in the previous subsection to the case $\mathcal{F}_j$ $(j=1,2,3)$. 
To determine the lattices for $\mathcal{F}_j$ $(j=1,2,3)$, we use the theory of the Mordell-Weil lattices due to T. Shioda.
For detail, see \cite{Shioda} and \cite{ShiodaNote}.

Let $S$ be a compact complex surface and $C$ be  a  algebraic curve. Let  $\pi : S \rightarrow C$ be an elliptic fibration with sections.
For generic $v\in C$, the fibre $\pi^{-1}(v)$ is an elliptic curve.
In the following, we assume that the elliptic fibration $\pi:S\rightarrow C$ has singular fibres.
$\mathbb{C}(C)$ denotes the field of meromorphic functions on $C$. If $C=\mathbb{P}^1(\mathbb{C})$, the field $\mathbb{C}(C)$ is isomorphic to the field $\mathbb{C}(t)$ of rational functions.

Here, $E(\mathbb{C}(C))$ denotes the Mordell-Weil group of sections of $\pi:S\rightarrow C$.
For all $P\in E(\mathbb{C}(C))$ and $v\in C$, 
we have $(P\cdot \pi^{-1}(v))=1$.
Note that the section $P$ intersects  an irreducible component with multiplicity $1$ of every fibre $\pi^{-1}(v)$.
Let $O$ be the zero of the group $E(\mathbb{C}(C))$. The section $O$ is given by the set of the points at infinity on every generic fibre.

Set 
$$
R = \{ v\in C|  \pi^{-1}(C) {\rm \enspace is \enspace a \enspace singular \enspace fibre \enspace of \enspace} \pi \}.
$$
For all $v\in R$, we have
\begin{align}\label{singcomp}
\pi^{-1}(v)=\Theta_{v,0}+ \sum_{j=1}^{m_v-1} \mu _{v,j} \Theta_{v,j},
\end{align}
where $m_v$ is the number of irreducible components of $\pi^{-1}(v)$, $\Theta_{v,j} \hspace{1mm} (j=0,\cdots,m_v -1)$ are irreducible components with multiplicity $\mu_{v,j}$ of $\pi^{-1}(v)$,
and $\Theta_{v,0}$ is the component with $\Theta_{v,0} \cap O \not= \phi$.

Let $F$ be a generic fibre of $\pi$.
Set
$$
T=\langle F, O, \Theta_{v,j} | v\in R, 1\leq j \leq m_v-1   \rangle_{\mathbb{Z}} \subset {\rm NS}(S).
$$
We call $T$ the trivial lattice for $\pi$.
For $P\in E(\mathbb{C}(C))$, $(P) \in {\rm NS}(S)$ denotes the corresponding element.  

\begin{thm} \label{Shiodatheorem} {\rm (Shioda \cite{Shioda}, see also \cite{ShiodaNote} Theorem (3$\cdot$10))}\\
{\rm (1)} The Mordell-Weil group $E(\mathbb{C}(C))$ is a finitely generated Abelian group.\\
{\rm (2)} The N\'{e}ron-Severi group ${\rm NS} (S)$ is a finitely generated Abelian group and torsion free.  \\
{\rm (3)} We have the isomorphism of groups
$
E(\mathbb{C}(C)) \simeq {\rm NS}(S)/T $
given by
$$
 P\mapsto (P)  \hspace{2mm}{\rm mod} \hspace{1mm} T.
$$
\end{thm}

We set $\hat{T}=(T \otimes _{\mathbb{Z}}  \mathbb{Q}) \cap {\rm NS}(S)$ for the trivial lattice $T$.

\begin{cor}\label{torsionCor} {\rm(\cite{Shioda}, see also \cite{ShiodaNote} Proposition (3$\cdot $11))}\\
{\rm (1)} 
$$
{\rm rank} (E(\mathbb{C}(C))) = {\rm rank (NS}(S))-2- \sum_{v\in R} (m_v -1).
$$\\
{\rm (2)}  Let $E(\mathbb{C}(C))_{tor}$ be the torsion part of $E (\mathbb{C}(C))$. Then,
$$
E (\mathbb{C} (C))_{tor} \simeq \hat{T}/T.
$$
\end{cor}

\vspace{3mm}
Set 
$$
E(\mathbb{C}(C))^0 =\{ P\in E(\mathbb{C}(C))| P \cap \Theta _{v,0}\not = \phi \enspace {\rm for \enspace  all\enspace } v\in R\}.
$$
We have 
\begin{eqnarray}\label{narrow}
E(\mathbb{C}(C))^0 \subset E(\mathbb{C}(C))/E(\mathbb{C}(C))_{tor}
\end{eqnarray}
(see \cite{Shioda}, see also \cite{ShiodaNote} Section 5).

Let $v\in R$. Under the notation {\rm (\ref{singcomp})}, we set
$$
(\pi^{-1} (v))^\sharp = \bigcup_{0\leq j \leq m_v -1,\hspace{0.5mm} \mu_{v,j}=1} \Theta _{v,j}^\sharp,
$$
where $\Theta _{v,j}^\sharp = \Theta _{v,j}-\{ {\rm singular\enspace points\enspace of\enspace} \pi^{-1}(v)\}$.
Set $m_v^{(1)}=\sharp \{j|0\leq j\leq m_v-1 ,\hspace{0.7mm} \mu_{v,j}=1 \}$.

\begin{thm} {\rm (\cite{Neron}, \cite{Kodaira2}, see also \cite{ShiodaNote} Section 7)}
Let $v\in R$. The set $(\pi^{-1}(v))^\sharp$ has a canonical group structure.
\end{thm}

\begin{rem}
Especially, for the singular fibre $\pi^{-1}(v)$ of type $I_b\hspace{1mm} (b\geq1)$, we have
$$
(\pi^{-1}(v))^\sharp \simeq \mathbb{C}^\times \times (\mathbb{Z}/b\mathbb{Z}).
$$ 
For the singular fibre $\pi^{-1} (v)$ of type $I^*_b\hspace{1mm} (b\geq0)$, we have
\[
(\pi^{-1}(v))^\sharp \simeq
\begin{cases}
\mathbb{C} \times (\mathbb{Z}/4\mathbb{Z}) &(b \in 2\mathbb{Z}+1), \cr
\mathbb{C} \times (\mathbb{Z}/2\mathbb{Z})^2&(b \in 2\mathbb{Z}).\cr
\end{cases}
\]
\end{rem}

For each $v\in C$, we introduce the mapping
$$
sp_v : E(\mathbb{C}(C))\rightarrow (\pi^{-1}(v))^\sharp : P\mapsto P\cap \pi^{-1}(v).
$$
Note that 
\begin{align*}
P\cap \pi^{-1}(v)= (x,a) \in \begin{pmatrix}\mathbb{C}^\times \\ \mathbb{C}\end{pmatrix}\times {\rm \{finite\enspace group\}}
\end{align*}
(see \cite{ShiodaNote} Section 7).
We call $sp_v$ the specialization mapping.

\begin{thm} {\rm (\cite{ShiodaNote} Section 7)}
For all $v\in C$, the specialization mapping 
\begin{align*}
sp_v:P\mapsto (x,a)\in
 \begin{pmatrix}\mathbb{C}^\times\\
 \mathbb{C}
 \end{pmatrix}
 \times{\rm \{finite \enspace group\}}
\end{align*}
is a homomorphism of groups.
\end{thm}

\begin{rem}\label{sprem}
Especially for the singular fibre $\pi^{-1}(v)$ of type $I_b$ ($I^*_b$, resp.), the projection of $sp_v$
$$
E(\mathbb{C}(C))\rightarrow (\mathbb{Z}/b\mathbb{Z}) \quad((\mathbb{Z}/4\mathbb{Z}) {\rm \enspace or \enspace} (\mathbb{Z}/2\mathbb{Z})^2,  {\rm\enspace resp.} )
$$
is a homomorphism of groups.
\end{rem}

\subsection{Proof for the case $P_1$}

Recall  the elliptic fibration given by (\ref{P1preKodaira}) and   Figure 3.

The trivial lattice for this fibration is
$$
T_1=\langle a_1,a_2,a_3,a_4,a_4',a_3',a_2',a_1',c_1,b_0,b_1,b_2,b_3,c_2,c_3,O,F  \rangle_\mathbb{Z}.
$$
Let $Q$ be  the section in (\ref{P1section}).
From (\ref{L1definition}), we have
\begin{eqnarray*}
L_1=\langle Q, T_1\rangle_\mathbb{Z}.
\end{eqnarray*}
This is a subgroup of ${\rm NS}(S_1(\lambda,\mu))$.
According to Theorem \ref{18theorem}
and Theorem \ref{Shiodatheorem} (3), we obtain
$$
{\rm NS}(S_1(\lambda,\mu))\otimes_\mathbb{Z} \mathbb{Q}=L_1\otimes_\mathbb{Z}\mathbb{Q}.
$$
We obtain also
\begin{eqnarray}\label{NSspan}
{\rm NS}(S_1(\lambda,\mu)) = (\langle Q\rangle_\mathbb{Q} \cap {\rm NS}(S_1(\lambda,\mu))) +\hat{T_1}
\end{eqnarray}
for generic $(\lambda,\mu)\in \Lambda_1$.
Since ${\rm det}(L_1)=-9$, we deduce that
\begin{eqnarray}\label{P1 =3}
[{\rm NS}(S_1(\lambda,\mu)):L_1]=1 \quad{\rm or} \quad[{\rm NS}(S_1(\lambda,\mu)):L_1]=3.
\end{eqnarray}
In the following, we prove
$$
[{\rm NS}(S_1(\lambda,\mu)):L_1]=1. 
$$

\begin{lem}\label{T1lemma}
For generic $(\lambda,\mu)\in \Lambda_1$,
 $\hat{T_1} = T_1.$
\end{lem}

\begin{proof}
From (\ref{NSspan}) and (\ref{P1 =3}), we have $\hat{T_1}=T_1$ or $[\hat{T_1}:T_1]=3$.
We assume $[\hat{T_1}:T_1]=3$.
Then, according to Corollary \ref{torsionCor} (2),
\begin{eqnarray}\label{torsionisomorph}
E(\mathbb{C}(x_1))_{tor} \simeq \hat{T_1}/T_1 \simeq \mathbb{Z}/3\mathbb{Z}.
\end{eqnarray}
Therefore there exists $R_0\in E(\mathbb{C}(x_1))_{tor}$ such that $3R_0 = O$.
By Remark \ref{sprem} and (\ref{narrow}), 
we suppose  that $R_0\cap a_3\not = \phi$ at $x_1 =0$ 
and $R_0 \cap c_0\not = \phi$ at $x_1=\infty$. 
Put $(R_0\cdot O)=k\in \mathbb{Z}$.
Set $\bar{T_1}=\langle T_1, R_1  \rangle_\mathbb{Z}$.
 By calculating the intersection matrix,  we have
\begin{eqnarray}\label{rank18hatP1}
{\rm det}(\bar{T_1})=-72(1+k+k^2)\not=0.
\end{eqnarray}
On the other hand,   due to (\ref{torsionisomorph}),  we have ${\rm rank}(\bar{T_1})=17$ .
 So it follows
$
{\rm det}(\bar{T_1})=0.
$
This contradicts (\ref{rank18hatP1}).
\end{proof}

By the above lamma, we have
\begin{eqnarray}\label{trivialP1}
{\rm NS}(S_1(\lambda,\mu))=(\langle Q\rangle_\mathbb{Q} \cap {\rm NS}(S_1(\lambda,\mu)))+ T_1.
\end{eqnarray}

\begin{lem} \label{L1lemma}
For generic $(\lambda,\mu)\in \Lambda_1$,
 ${\rm NS}(S_1(\lambda,\mu))=L_1$.
\end{lem}

\begin{proof}
It is sufficient to prove $[{\rm NS}(S_1(\lambda,\mu)):L_1]=1$. We assume $[{\rm NS}(S_1(\lambda,\mu)):L_1]=3$.
By (\ref{trivialP1}), there exists  $R_1 \in E(\mathbb{C}(x_1))$ such that $3R_1 =Q$.
According to Remark \ref{sprem},
$$
(R_1 \cdot c_3)=1, \quad {\rm at \enspace } x_1=\infty
$$
and
\begin{align*}
\begin{cases}
(R_1 \cdot a_1)=1, \\
{\rm or}\\
(R_1 \cdot a_4)=1,\\
{\rm or}\\
(R_1 \cdot a_7)=1,
\end{cases}
\quad {\rm at\enspace} x_1 =0.
\end{align*}
We  assume $(R_1 \cdot O)=0$, for $Q$ in (\ref{P1section}) does not intersect $O$.  
By the addition theorem for elliptic curves, we have $2Q$ and we can check $2Q$ does not intersect $O$. 
If we have $p\in R_1\cap Q$, then it holds $R_1|_p=Q|_p$.
By the assumption, we have $(3R_1)|_p=Q|_p$.
It means that $2Q\cap O\not=\phi$.
But, it is not the case.
So,  we suppose $(R_1 \cdot Q)=0$ also.
Set $\tilde{L_1}=\langle L_1,R_1\rangle_\mathbb{Z}$.
By calculating the intersection matrix, we have
\begin{align}\label{rankP1tilde}
{\rm det}(\tilde{L_1})=
\begin{cases}
12 & ({\rm if \enspace}(R_1\cdot a_1)=1),\cr
-30 & ({\rm if \enspace}(R_1\cdot a_4)=1),\cr
6& ({\rm if \enspace }(R_1\cdot a_7)=1).
\end{cases}
\end{align}
On the other hand,  we have ${\rm rank}(\tilde{L_1})=18$ from Theorem \ref{18theorem}. 
Hence, we obtain 
$
{\rm det}(\tilde{L_1})=0.
$
This contradicts  (\ref{rankP1tilde}). 
Therefore, we have $[{\rm NS}(S_1(\lambda,\mu)):L_1 ]=1$.
\end{proof}

\begin{lem}\label{thm:intersectionP_1}
The   lattice $L_1$ is isomorphic to the lattice given by the intersection matrix  
\begin{eqnarray*}\label{matrix:P1neronseveri}
E_8(-1)\oplus E_8(-1) \oplus
\begin{pmatrix} 0&3 \\
  3&0 
 \end{pmatrix},
\end{eqnarray*}
and  its orthogonal complement is given by the intersection matrix
\begin{eqnarray*}\label{eq:P1transcendental}
A_1=U\oplus
\begin{pmatrix}
0 &3 \\
 3&0 
 \end{pmatrix} .
\end{eqnarray*}
\end{lem}

\begin{proof}
Let $M_1$ be the intersection matrix in (\ref{intersection L1}).
Set
\begin{align*}
\begin{cases}
&v_{15}^{(1)}={}^t (0,0,0,0,0,0,0,0,0,0,0,0,0,1,0,-1,0,-1),\\
&v_{16}^{(1)}={}^t (11,22,33,26,19,12,5,-2,2,4,6,8,10,7,5,1,18,-4),\\
&v_{17}^{(1)}={}^t (8,16,24,19,14,9,4,-1,2,4,6,8,10,7,5,-1,13,-5),\\
&v_{18}^{(1)}={}^t (91,182,273,214,155,96,37,-22,18,36,54,72,90,63,45,0,150,-36).
\end{cases}
\end{align*}
Recall the vectors in (\ref{unitv}).
Set
$$
U_1=(r_7,r_6,r_5,r_4,r_3,r_{17},r_2,r_1,r_9,r_{10},r_{11},r_{12},r_{13},r_{15},v_{15}^{(1)},v_{16}^{(1)},v_{17}^{(1)},v_{18}^{(1)}).
$$
This is an unimodular matrix.
We have 
$$
{}^t U_1 M_1 U_1 =E_8(-1)\oplus E_8(-1) \oplus \begin{pmatrix}0&3\\ 3&0 \end{pmatrix}.
$$
\end{proof}

Therefore, we obtain Theorem \ref{latticeThm} for $P_1$.

\subsection{Proof for the case $P_2$}

The elliptic fibration given by (\ref{P2preKodaira}) is illustrated in  Figure 4.

The trivial lattice for this fibration is
$$
T_2=\langle  a_1,a_2,a_3,a_4,a_5,a_5',a_4',a_3',a_2',a_1',c_1,b_0,b_1,c_2,c_3,O,F\rangle_\mathbb{Z}.
$$
Let $Q$ be the section in (\ref{P2section}).  
From (\ref{L2definition}), we have
\begin{eqnarray*}
L_2=\langle Q,T_2\rangle_\mathbb{Z}.
\end{eqnarray*}
This is a subgroup of ${\rm NS}(S_2(\lambda,\mu))$.
 As in the case $\mathcal{F}_1$, so we obtain
$$
{\rm NS}(S_2(\lambda,\mu))=(\langle Q\rangle_\mathbb{Q} \cap {\rm NS}(S_2(\lambda,\mu)))+\hat{T_2}
$$
for generic $(\lambda,\mu)\in \Lambda_2$. 
Since ${\rm det}(L_2)=-9$, we have
\begin{eqnarray}\label{P2 =3}
[{\rm NS}(S_2(\lambda,\mu)):L_2]=1 {\rm \enspace or \enspace} [{\rm NS}(S_2(\lambda,\mu)):L_2]=3.
\end{eqnarray}
In the following, we prove $[{\rm NS}(S_2(\lambda,\mu)):L_2]=1$.

\begin{lem}\label{T2lemma}
For generic $(\lambda,\mu)\in \Lambda_2$,   $\hat{T_2}=T_2$.
\end{lem}

\begin{proof}
Because  we have ${\rm det}(T_2)=-44$ and (\ref{P2 =3}), it follows $\hat{T_2}=T_2$.
\end{proof}

Therefore, we obtain
\begin{eqnarray} \label{trivialP2}
{\rm NS}(S_2(\lambda,\mu))=(\langle Q \rangle_\mathbb{Q}\cap{\rm NS}(S_2(\lambda,\mu)))+T_2.
\end{eqnarray}

\begin{lem}
For generic $(\lambda,\mu)\in \Lambda_2$,  ${\rm NS}(S_2(\lambda,\mu))=L_2$.
\end{lem}

\begin{proof}
We assume $[{\rm NS}(S_2(\lambda,\mu)):L_2]=3$. From (\ref{trivialP2}), there exists $R_1 \in E(\mathbb{C}(y))$ such that $3R_1 =Q$.
According to Remark \ref{sprem}, we obtain $(R_1 \cdot a_3')=1$ and $(R_1 \cdot c_3)=1$. Because the section $Q$ in (\ref{P2section})  and the section $2Q$ do not intersect $O$, we have  $(R_1\cdot O)=0$ and $(R_1 \cdot Q)=0$.
Set $\tilde{L_2}=\langle L_2,R_1\rangle_\mathbb{Z}$. Calculating its intersection matrix, we have ${\rm det}(\tilde{L_2})=-38$. As in the proof of  Lemma \ref{L1lemma},  this  contradicts to Theorem \ref{18theorem}.
\end{proof}

\begin{lem}\label{thm:intersectionP_2}
The lattice  $L_2$ is isomorphic to the lattice given by the following intersection matrix 
\begin{eqnarray*}\label{matrix:P2neronseveri}
 E_8(-1)\oplus E_8(-1) \oplus
\begin{pmatrix}0&3 \\
 3&2  
  \end{pmatrix},
\end{eqnarray*}
and its orthogonal complement  is given by the intersection matrix
\begin{eqnarray*}\label{eq:P2transcendental}
A_2=
U\oplus
\begin{pmatrix}
0 &3 \\
 3&-2 
 \end{pmatrix}.
\end{eqnarray*}
\end{lem}

\begin{proof}
Let $M_2$ be the intersection matrix in (\ref{intersection L2}).
Set
\begin{align*}
\begin{cases}
&v_{14}^{(2)}={}^t (5,4,15,26,13,10,8,6,4,2,12,24,36,30,18,-4,24,-8),\\
&v_{15}^{(2)}={}^t (1,-2,3,8,1,0,0,0,0,0,6,12,18,15,9,0,12,1),\\
&v_{17}^{(2)}={}^t (56,13,162,311,120,100,80,60,40,20,170,340,510,425,255,-28,340,-56),\\
&v_{18}^{(2)}={}^t (27,6,80,154,60,50,40,30,20,10,84,168,252,210,126,-14,168,-28).
\end{cases}
\end{align*}
Recall the vectors in (\ref{unitv}).
Set
$$
U_2=(r_3,r_4,r_{17},r_{14},r_{13},r_{15},r_{12},r_{11},r_{10},r_{9},r_{8},r_7,r_6,v_{14}^{(2)},v_{15}^{(2)},r_{16},v_{17}^{(2)},v_{18}^{(2)}).
$$
This is an unimodular matrix.
We have 
$$
{}^t U_2 M_2 U_2 =E_8(-1)\oplus E_8(-1) \oplus \begin{pmatrix}0&3\\ 3&2 \end{pmatrix}.
$$
\end{proof}

Therefore, we obtain Theorem \ref{latticeThm} for $P_2$.

\subsection{Proof for the case $P_3$}

The elliptic fibration given by (\ref{P3preKodaira}) is illustrated in  Figure 5.

The trivial lattice for this fibration is
$$
T_3=\langle a_1,a_2,a_3,a_4,a_0', a_4',a_3',a_2',a_1',c_1,b_0,b_1,b_2,c_2,c_3,O,F
 \rangle_\mathbb{Z}.
$$
Let $Q$ be the section in (\ref{P3section}). 
From (\ref{L3'def}), we see
$$
L_3'=\langle  Q,T_3\rangle_\mathbb{Z}.
$$
This is a subgroup of ${\rm NS}(S_3(\lambda,\mu))$ and
we have ${\rm det}(L_3')=-36$.
Moreover,  the section $O'$ in (\ref{P3section}) is a $2$-torsion section  for this elliptic fibretion.
Due to Corollary \ref{torsionCor},  $[\hat{T_3}:T_3]$ is divided by $2$.
Hence, we have
\begin{eqnarray}\label{P3 =6}
[{\rm NS}(S_3(\lambda,\mu)):L_3']=2 {\rm \enspace or \enspace} [{\rm NS}(S_3(\lambda,\mu)):L_3']=6.
\end{eqnarray}

\begin{lem} \label{trivialP3}
For generic $(\lambda,\mu)\in \Lambda_3$,  $[\hat{T_3}:T_3]=2$.
\end{lem}

\begin{proof}
We have ${\rm det}(T_3)=-40$. From (\ref{P3 =6}), we obtain $[\hat{T_3}:T_3]=2$.
\end{proof}

\begin{lem}
For generic $(\lambda,\mu)\in\Lambda_3$,  $[{\rm NS}(S_3(\lambda,\mu)):L_3']=2$.
\end{lem}

\begin{proof}
We shall show that $[{\rm NS}(S_3(\lambda,\mu)):L_3']=2$.
We assume $[{\rm NS}(S_3(\lambda,\mu)):L_3']=6$.
From Lemma \ref{trivialP3}, there exists $R_1 \in E(\mathbb{C}(x_1))$ such that $3R_1=Q$.
According to Remark \ref{sprem},  $(R_1 \cdot c_2)=1 $ and $(R_1 \cdot a_4)=1$. 
Also we have $(R_1 \cdot O)=0$, for $Q$ in (\ref{P3section}) does not intersect $O$. Moreover, we  assume that $(R_1 \cdot Q)=0$ or $1$, for the section $2P$   does not intersect $O$ at $x_1\not = \infty$.
Set $\tilde{L'_3}= \langle L'_3,R\rangle_\mathbb{Z}$. 
Calculating the intersection matrix, we have
\begin{eqnarray}\label{dettildeL3}
{\rm det}(\tilde{L_3'})=
\begin{cases}
-16& ({\rm if \enspace} (R_1\cdot Q)=0)\\
-112&({\rm if \enspace} (R_1\cdot Q)=1)
\end{cases}.
\end{eqnarray}
On the other hand, Theorem \ref{18theorem} implies that ${\rm rank}(\tilde{L'_3})=18$ and $
{\rm det }(\tilde{L'_3})=0.
$
This is a contradiction to (\ref{dettildeL3}).
\end{proof}

Due to the above lemma,
we have 
$$
|{\rm det}({\rm NS}(S_3(\lambda,\mu)))|=9
$$
for generic $(\lambda,\mu)\in \Lambda_3$.

To determine the explicit lattice structure for $\mathcal{F}_3$, we use another elliptic fibration  defined by (\ref{P3preKodairaanother}).
This fibration is illustrated in  Figure 6.

Let $Q_0$ and $R_0$  be the sections in (\ref{P3sectionanother}) for this elliptic fibration.
Recall
\begin{eqnarray*}
L_3 =\langle d_1,d_2,d_3,d_4,d_4',d_3',d_2',d_1',e_1,e_2,e_3,e_4,e_3',e_2',O,Q_0,R_0,F \rangle_\mathbb{Z}.
\end{eqnarray*}
 in (\ref{L3definition}).
 For generic $(\lambda,\mu)\in \Lambda_3$, since
 $$L_3\otimes_\mathbb{Z}\mathbb{Q} ={\rm NS }(S_3(\lambda,\mu))\otimes_\mathbb{Z} \mathbb{Q}$$ and ${\rm det}(L_3')=-9$,
we deduce that
$$
L_3 ={\rm NS}(S_3(\lambda,\mu)).
$$

\begin{lem}\label{thm:intersectionP_3}
The lattice $L_3$  is isomorphic to the lattice given by the intersection matrix
\begin{eqnarray*}\label{matrix:P3neronseveri}
 E_8(-1)\oplus E_8(-1) \oplus
 \begin{pmatrix}
 0&3 \\
 3&-2 
 \end{pmatrix} ,
\end{eqnarray*}
and its orthogonal complement is given by the intersection matrix
\begin{eqnarray*}\label{eq:P3transcendental}
A_3=
U\oplus
\begin{pmatrix}
0 &3 \\
3&2 \end{pmatrix} .
\end{eqnarray*}
\end{lem}

\begin{proof}
Let $M_3$ be the intersection matrix in (\ref{intersection L3}).
Set
\begin{align*}
\begin{cases}
&v_9^{(3)}={}^t (28,56,84,27,21,15,10,5,34,68,102,51,-1,-1,1,85,-1,-16),\\
&v_{17}^{(3)}={}^t (5,10,15,5,4,3,2,1,6,12,18,9,0,0,0,15,0,-3),\\
&v_{18}^{(3)}={}^t (468,936,1404,432,378,324,216,108,576,1152,1728,864,36,18,35,1440,54,-252).
\end{cases}
\end{align*}
Recall the vectors in (\ref{unitv}).
Set
$$
U_3=(r_1,r_2,r_3,r_{16},r_{11},r_{12},r_{10},r_9,v_9^{(3)},r_{14},r_{13},r_{17},r_6,r_5,r_7,r_8,v_{17}^{(3)},v_{18}^{(3)}).
$$
This is an unimodular matrix.
We have
$$
{}^t U_3 M_3 U_3 =E_8(-1)\oplus E_8(-1) \oplus \begin{pmatrix}0&3\\ 3&-2 \end{pmatrix}.
$$
\end{proof}

Therefore, we obtain Theorem \ref{latticeThm} for $P_3$.

\section{Period differential equations}

Recall $F_j$ $(j=0,1,2,3)$ in (\ref{equationsF_j}).
The unique holomorphic 2-form on $S_j(\lambda,\mu)$ $(j=0,1,2,3)$  is given by
\begin{eqnarray}\label{eq:2formP4}
\omega_0= \frac{z dz\wedge dx}{\partial F_0 /\partial y}, \quad\quad\omega_j=\frac{dz\wedge dx}{\partial F_j /\partial y}\quad(j=1,2,3),
\end{eqnarray}
up to a constant factor.

\begin{prop}\label{periodprop}
Let $j\in \{0,1,2,3\}$.
There is  a $2$-cycle $\Gamma_j$ on $S_j(\lambda,\mu)$ such that
the period integral
 $\displaystyle \iint_{\Gamma_j}\omega_j $ 
 has the following power series expansion, which is valid in a sufficiently small neighborhood of $(\lambda,\mu)=(0,0)$.

{\rm (0)}
{\rm (}Periods for $\mathcal{F}_0${\rm )}
\begin{eqnarray*}
\eta_0 (\lambda,\mu)=\iint_{\Gamma_0} \omega=(2\pi i)^2 \sum_{n,m=0}^\infty  (-1)^m \frac{(5m+2n)!}{n! (m!)^3 (2m+n)!}\lambda^n \mu^m.
\end{eqnarray*}

{\rm (1)} 
{\rm (}Periods for $\mathcal{F}_1${\rm )}
\begin{eqnarray*}\label{P1period} 
\eta _1(\lambda,\mu) =\iint_{\Gamma_1}\omega_1=(2\pi i)^2 \sum_{n,m=0}^\infty (-1)^{m+n} \frac{(3n+3m)!}{(n!)^2 (m!)^2 (n+m)!} \lambda^n \mu^m.
\end{eqnarray*}

{\rm (2)} 
{\rm (}Periods for $\mathcal{F}_2${\rm )}
\begin{eqnarray*} \label{P2period}
\eta_2 (\lambda ,\mu )=\iint_{\Gamma_2}\omega_2 =(2\pi i)^2 \sum_{n,m=0}^\infty (-1)^n \frac{(4m+3n)!}{(m!)^2 n! ((m+n)!)^2}\lambda^n \mu^m.
\end{eqnarray*}

{\rm (3)} 
{\rm (}Periods for $\mathcal{F}_3${\rm )}
\begin{eqnarray*}\label{P3period} 
\eta_3 (\lambda ,\mu) =\iint_{\Gamma_3}\omega_3=(2\pi i) ^2 \sum_{n,m=0}^\infty (-1)^{n} \frac{(3n+2m)!}{(n!)^3 (m!)^2 } \lambda^n \mu^m.
\end{eqnarray*}\end{prop}

\begin{proof}
Here, we state the detailed proof only for the case (0).

When $(\lambda,\mu)$ is sufficiently small, $S_0(\lambda,\mu)$ in (\ref{eq:P4}) is regarded  as a double cover by the projection
$$p: (x,y,z)\mapsto (x,z).$$
Let $\xi _1(x,z),\xi _2(x,z)$ be the two roots of $F_0(x,y,z)=0$ in $y$. Then, we have
$$
F_0(x,y,z)=xz^2(y-\xi _1(x,z))(y-\xi_2(x,z)).
$$
and
$$
\frac{\partial F_0}{\partial y}(x,y,z)=xz^2((y-\xi _1(x,z))+(y-\xi_2(x,z))).
$$
Therefore, at $(x,\xi_1(x,z),z)\in S_0(\lambda ,\mu)$, 
$$
\frac{\partial F_0}{\partial y}(x,\xi _1(x,z),z)=xz^2(\xi _1(x,z)-\xi _2(x,y)).
$$
We have a local inverse mapping of $p$
$$
q: (x,z)\mapsto (x,\xi_1(x,z),z).
$$
Let $\gamma _1$ ($\gamma _2$, $\gamma _3$, resp.) be a cycle in $x$-plane ($y$-plane, $z$-plane, resp.)  which goes around the origin once in the positive direction. We suppose that there exists $\delta> 0$ such that  it holds 
$$
|\xi_1(x,z)|-|\xi_2(x,z)|\geq \delta
$$
for any $(x,z)\in \gamma_1\times\gamma_3$.
We assume that  $x=-1$ stays outside of $\gamma_1$, $z=-1-x$ stays outside of $\gamma_3$ for any $x\in \gamma_1$, and that $y=\xi_1(x,z)$ stays inside of $\gamma_2$ and $y=\xi_2(x,z)$ and $-1-x-z$ stay outside of $\gamma_2$ for any $(x,z)\in \gamma_1\times\gamma_3$. Moreover, by taking a neighborhood $U$ of the origin sufficiently small, we  assume
$$
|\lambda xyz +\mu|\leq |xyz^2(x+y+z+1)|
$$
for any $(x,y,z)\in \gamma_1\times\gamma_2\times\gamma_3$ and $(\lambda,\mu)\in U$.
So, $q(\gamma_1 \times \gamma_3)$ is a $2$-cycle on $S_0(\lambda,\mu)$.

Let us calculate the period integral on the 2-cycle $q(\gamma_1\times\gamma_3)$ on $S_0(\lambda,\mu)$.
Let $\omega$ be the holomorphic 2-form given in (\ref{eq:2formP4}). By the residue theorem, 
\begin{align}\label{eq:integralP4}
&\iint_{q(\gamma_1\times\gamma_3)} \omega
=\iint_{\gamma_3 \times \gamma_1} \frac{z dz \wedge dx}{xz^2(\xi_1 (x,z) -\xi_2 (x,z))} \notag\\
&=\frac{1}{2\pi\sqrt{-1}}\iiint_{\gamma_3 \times \gamma_1\times\gamma_2}\frac{z dz \wedge dx \wedge dy}{xz^2(y-\xi_1 (x,z))(y-\xi_2(x,z))}\notag\\
&=\frac{1}{2\pi \sqrt{-1}}\iiint_{\gamma_3 \times \gamma_1\times\gamma_2}\frac{z dz \wedge dx \wedge dy}{xyz^2(x+y+z+1) + \lambda xyz +\mu}.
\end{align}
By the residue theorem and the binomial theorem, we have
\begin{align*}
&\frac{1}{2\pi \sqrt{-1}}\iiint_{\gamma_3 \times \gamma_1\times\gamma_2}\frac{z dz \wedge dx \wedge dy}{xyz^2(x+y+z+1) + \lambda xyz +\mu}\\
&=\frac{1}{2\pi \sqrt{-1}}\iiint_{\gamma_3 \times \gamma_1\times\gamma_2}\frac{1}{xyz^2(x+y+z+1)} \frac{z dz \wedge dx \wedge dy}{1+ \frac{\lambda xyz +\mu } {xyz^2(x+y+z+1)}}\\
&=\frac{1}{2\pi \sqrt{-1}}\sum_{l=0}^\infty \iiint _{\gamma_3\times\gamma_1\times\gamma_2} \frac{z (-\lambda x y z-\mu) ^l}{(xyz^2(x+y+z+1))^{l+1}} dz\wedge dx \wedge dy\\
&=\frac{1}{2\pi \sqrt{-1}}\sum_{m,n=0}^\infty \iiint _{\gamma_3\times\gamma_1\times\gamma_2} \begin{pmatrix}m+n\\m\end{pmatrix}\frac{x^n y^n z^{n+1} dz\wedge dx\wedge dy }{(xyz^2(x+y+z+1))^{m+n+1}}(-\lambda)^n (-\mu )^m\\
&=\frac{1}{2\pi \sqrt{-1}}\sum_{m,n=0}^\infty \iiint _{\gamma_3\times\gamma_1\times\gamma_2} \frac{(m+n)!}{n!m!}\frac{dz\wedge dx\wedge dy }{x^{m+1} y^{m+1} z^{2m+n+1}(x+y+z+1)^{m+n+1}}(-\lambda)^n (-\mu )^m\\
&=\sum_{n,m=0}^\infty \iint_{\gamma_3\times\gamma_1}\frac{(2m+n)!}{(m!)^2 n!}(-1)^m \frac{dz\wedge dx}{x^{m+1}z^{2m+n+1}(x+z+1)^{2m+n+1}}(-\lambda )^n(-\mu )^m\\
&=(2\pi \sqrt{-1})\sum_{n,m=0}^\infty \int_{\gamma_3} \frac{(3m+n)!}{(m!)^3 n!}\frac{dz}{z^{2m+n+1} (z+1)^{3m+n+1}} (-\lambda )^n (-\mu )^m\\
&=(2\pi \sqrt{-1})^2\sum_{n,m=0}^\infty (-1)^m \frac{(5m +2n)!}{(m!)^3 n! (2m+n)!} \lambda ^n \mu ^m.
\end{align*}
The above power series is holomorphic on $U$.
\end{proof}

\begin{rem}
In the case {\rm (1)}, our period reduces to the Appell $F_4$(see {\rm \cite{Koike}}
):
$$
\eta_1(\lambda ,\mu)=F_4\Big(\frac{1}{3},\frac{2}{3},1,1;27\lambda,27\mu\Big)=F\Big(\frac{1}{3},\frac{2}{3},1;x\Big)F\Big(\frac{1}{3},\frac{2}{3},1,;y\Big),
$$
where $F$ is the Gauss hypergeometric function and $x(1-y)=27\lambda, y(1-x)=27\mu$.
\end{rem}

In Section 1,  we obtain our families as a section of anti-canonical section of the toric varieties derived from the reflexive polytopes.
We  obtain the GKZ  system of equations  for the periods.

In the following, we use the notation 
$$
\theta _\lambda =\lambda\frac{\partial}{\partial \lambda}, \quad \theta _\mu =\mu \frac{\partial }{\partial \mu}.
$$

\begin{prop}
Let $\eta_j(\lambda,\mu)$ $(j=0,1,2,3)$ be the periods given in {\rm Proposition \ref{periodprop}}.
Then,
\begin{eqnarray*}
D^{(j)} _1 \eta_j (\lambda,\mu) = D^{(j)} _2 \eta_j (\lambda,\mu) =0   \quad\quad\quad(j=0,1,2,3),
\end{eqnarray*}
 where $D^{(j)}_1$ and $D^{(j)}_2$ are given as follows. 
 
{\rm (0)} 
{\rm (}The GKZ system of equations for $\mathcal{F}_0$ {\rm )}
\begin{eqnarray*}\label{system:GKZ}
&&
\begin{cases}
D^{(0)}_1&= \theta _\lambda (\theta  _\lambda +2\theta _\mu )-\lambda (2\theta _\lambda +5\theta _\mu +1)(2\theta _\lambda +5\theta _\mu +2),\cr
D^{(0)}_2&= \lambda ^2 \theta _{\mu }^3+\mu \theta _\lambda (\theta _\lambda -1)(2\theta _\lambda +5\theta _\mu +1).
\end{cases}
\end{eqnarray*}
 
{\rm (1)}
{\rm (}The GKZ  system of  equations for $\mathcal{F}_1${\rm )}
\begin{eqnarray*}
&&
\begin{cases}
D_1^{(1)}=\lambda \theta_\mu ^2-\mu \theta_\lambda^2,\\
D_2^{(1)}=\lambda(3\theta_\lambda+3\theta_\mu)(3\theta_\lambda+3\theta_\mu-1)(3\theta_\lambda+3\theta_\mu-2).
\end{cases}
\end{eqnarray*}

{\rm (2)} 
{\rm (}The GKZ  system of  equations for $\mathcal{F}_2${\rm )}
\begin{eqnarray*}
&&
\begin{cases}
D_1^{(2)}&=\lambda \theta _\mu ^2 +\mu \theta _\lambda  (3\theta _\lambda +4\theta _\mu +1),\cr
D_2^{(2)}&=\theta _\lambda (\theta _\lambda +\theta _\mu )^2+\lambda (3\theta _\lambda +4\theta _\mu +1)(3\theta _\lambda +4\theta _\mu +2)(3\theta _\lambda +4\theta _\mu +3).
\end{cases}
\end{eqnarray*}

{\rm (3)}
{\rm (}The GKZ  system of  equations for $\mathcal{F}_3${\rm )}
\begin{eqnarray*}
&&
\begin{cases}
D_1^{(3)}&=\theta _{\lambda}^2-\mu (3\theta _{\lambda }+2\theta _{\mu }+1)(3\theta _{\lambda }+2\theta _{\mu }+2),\cr
D_2^{(3)}&=\theta _{\lambda }^3+\lambda  (3\theta _{\lambda }+2\theta _{\mu }+1)(3\theta _{\lambda }+2\theta _{\mu }+2)(3\theta _{\lambda }+2\theta _{\mu }+3).
\end{cases}
\end{eqnarray*}
\end{prop}

\begin{proof}
Extending the matrix $P_j$ $(j=0,1,2,3)$, set
\begin{align*}
&\mathcal{A}_0=\begin{pmatrix}
1&1&1&1&1&1\\
0&1&0 &0& 0&-1\\
0&0&1  & 0& 0 & -1\\
0&0& 0&1 & -1& -2\end{pmatrix},
\quad\quad
\mathcal{A}_1=\begin{pmatrix}
1&1&1&1&1&1\\
0&1&0&0&-1&0\\
0&0&1&0&0&-1\\
0&0&0&1&-1&-1
\end{pmatrix},\\
&\mathcal{A}_2=\begin{pmatrix}
1&1&1&1&1&1\\
0&1&0&0&0&-1\\
0&0&1&0&-1&-1\\
0&0&0&1&-1&-1
\end{pmatrix},
\quad\quad
\mathcal{A}_3=\begin{pmatrix}
1&1&1&1&1&1\\
0&1&0&0&-1&0\\
0&0&1&0&-1&0\\
0&0&0&1&0&-1
\end{pmatrix},
\end{align*}
and
$\displaystyle 
\beta=\begin{pmatrix}
-1\\
0\\
0\\
0\\
\end{pmatrix}.
$
From the matrix $\mathcal{A}_j$ $(j=0,1,2,3)$ and the vector $\beta$, we have the GKZ system for $\eta_j(\lambda,\mu)$  $(j=0,1,2,3)$.
In the following, we state the detailed proof only for $\mathcal{F}_0$.

The GKZ system of equations  defined by  $\mathcal{A}_0$ and $\beta$  has a solution 
\begin{align}\label{eq:inttori}
&\iiint _{\Delta} R_0^{-1} t_1 ^{-1}t_2^{-1} t_3^{-1} dt_1 \wedge dt_2 \wedge dt_3 \notag\\
&=\iiint_{\Delta} \frac{ t_3  dt_1 \wedge dt_2 \wedge dt_3}{(t_1t_2t_3^2(a_1+a_2t_1 +a_3t_2+a_4 t_3 )+a_5 t_1 t_2 t_3+a_6 )},
\end{align}
where
$$
R_0=a_1 +a_2 t_1 +a_3t_2+a_4 t_3 +a_5 \frac{1}{t_3}+a_6 \frac{1}{t_1 t_2 t_3^2},
$$
and
$\Delta$ is a twisted cycle.
By the parameter transformation (\ref{eq:para}), (\ref{eq:inttori}) is transformed to
$$
\frac{1}{a_1} 
\iiint_{\Delta} 
\frac{z dx\wedge dy \wedge dz}
{xyz^2(x+y+z+1)+\lambda xyz +\mu}
 =\frac{1}{a_1} \eta(\lambda , \mu).
$$

Set $\displaystyle \theta_j = a_j \frac{\partial}{\partial a_j}$.
The above mentioned  GKZ system  is given by the following equations (\ref{system:GKZ1}),(\ref{eq:partial1}) and (\ref{eq:partial2}) :
\begin{eqnarray}\label{system:GKZ1}
\begin{cases}
(\theta_1 +\theta_2+\theta _3+\theta _4+\theta _5+\theta _6)\eta =-\eta,\cr
(\theta _2 -\theta _6)\eta =0,\cr
(\theta _3 -\theta _6)\eta =0,\cr
(\theta _4-\theta _5-2\theta _6) \eta = 0,
\end{cases}
\end{eqnarray}
\begin{align}
&\frac{\partial ^2}{\partial a_4 \partial a_5}\eta=\frac{\partial ^2 }{\partial a_1 ^2}\eta,  \label{eq:partial1}\\
&\frac{\partial ^3}{\partial a_2 \partial a_3 \partial a_6} \eta=\frac{\partial ^3}{\partial a_1 \partial a_5 ^2}\eta. \label{eq:partial2}
\end{align}

By (\ref{eq:para}), we have
\begin{align*}
\theta _\lambda =\theta_5, \quad
\theta_\mu =
\theta_6.
\end{align*}
So, from (\ref{system:GKZ1}) we have
\begin{align*}
\begin{cases}
\theta _2 \eta =\theta _\mu \eta,\\
\theta _3 \eta=\theta _\mu \eta,\\
\theta _4 \eta=(\theta _\lambda +2\theta _\mu )\eta,\\
\theta_ 1 \eta=(-2\theta _\lambda -5\theta _\mu -1)\eta. 
\end{cases}
\end{align*}
From (\ref{eq:partial1}), we have
\begin{align*}
\begin{cases}
&\vspace*{0.2cm}
\displaystyle\frac{\partial ^2}{\partial a_4 \partial a_5} \eta =\frac{1}{a_4 a_5}\theta_4 \theta_5 \eta =\frac{1}{a_4 a_5} (\theta _\lambda +2\theta \_\mu )\theta _\lambda \eta ,\\
&\displaystyle \frac{\partial ^2}{\partial a_1^2}\eta=\frac{1}{a_1^2} \theta _1(\theta _1 -1)\eta=\frac{1}{a_1^2}(2\theta _\lambda +5\theta _\mu +1)(2\theta _\lambda +5\theta _\mu +2)\eta.
\end{cases}
\end{align*}
Hence, we obtain
\begin{align*}
(\theta _\lambda +2\theta _\mu) \eta =\lambda (2\theta _\lambda +5\theta_\mu+1 )(2\theta _\lambda +5\theta _\mu +2)\eta.
\end{align*}
Similarly, from (\ref{eq:partial2}), we have
\begin{align*}
\begin{cases}
&\vspace*{0.2cm}
\displaystyle \frac{\partial ^3}{\partial a_2 \partial a_3 \partial a_6}\eta=\frac{1}{a_2 a_3 a_6} \theta _2 \theta _3 \theta _6 \eta=\frac{1}{a_2 a_3 a_6}\theta _\mu ^3 \eta,\\
&\displaystyle \frac{\partial ^3}{\partial a_1 \partial a_5 ^2}\eta=\frac{1}{a_5 ^2}\theta _1 \theta _5(\theta _5 -1)\eta=\frac{1}{a_1 a_5^2} (-2\theta _\lambda -5 \theta _\mu -1)\theta _\lambda (\theta _\lambda -1)\eta,
\end{cases}
\end{align*}
hence
\begin{align*}
\lambda ^2 \theta _\mu ^3 \eta = - \mu (2 \theta _\lambda +5\theta _\mu +1)\theta _\lambda (\theta _\lambda -1)\eta.
\end{align*}
\end{proof}

We  obtain  $6\times 6$ Pfaffian systems from the above GKZ systems  $D_1^{(j)}u=D_2^{(j)}u=0$ $(j=0,1,2,3)$. 
These systems are integrable. Therefore, each system  has a 6-dimensional  space of solutions. However, as we remarked in \rm Corollary \ref{cor:dim}, we expect the  systems of differential equations  with 4-dimensional space of solutions. It suggests that  the above systems  are reducible.
So,
using the above $D^{(j)}_1$  $(j=0,1,2,3)$, 
 we determine the period differential equation for $\mathcal{F}_j$ $(j=0,1,2,3)$  with 4-dimensional spaces of solutions.

\begin{thm}\label{thm:periodDE!}
Let $j\in\{0,1,2,3\}$.
Set the  system of differential equations  $D^{(j)}_1 u= D^{(j)}_3 u =0$ as follows.
Then, 
$$
D^{(j)}_1 \eta_j(\lambda,\mu) = D^{(j)}_3 \eta_j (\lambda,\mu) =0,
$$
where $\eta_j (\lambda,\mu)$ is given in {\rm Proposition \ref{periodprop}}.
The space of solutions of this system is $4$-dimensional.

{\rm (0)}
{\rm (}The period differential equation for $\mathcal{F}_0${\rm)}
 \begin{eqnarray} \label{eq:systemP4}
&&
\begin{cases}
D^{(0)}_1&= \theta _\lambda (\theta  _\lambda +2\theta _\mu )-\lambda (2\theta _\lambda +5\theta _\mu +1)(2\theta _\lambda +5\theta _\mu +2),\cr
D^{(0)}_3&=\lambda ^2(4\theta _\lambda ^2-2\theta _\lambda \theta _\mu +5\theta _\mu ^2)
-8\lambda ^3(1+3\theta _\lambda +5\theta _\mu +2\theta _\lambda ^2+5\theta _\lambda \theta _\mu )+25\mu \theta _\lambda (\theta _\lambda -1).
\end{cases}
\end{eqnarray}

{\rm (1)}
{\rm (}The period differential equation for $\mathcal{F}_1${\rm )}
\begin{eqnarray} \label{eq:systemP1}
&&
\begin{cases}
D_1^{(1)}&=\lambda \theta _\mu ^2 -\mu \theta _\lambda^2  ,\vspace{2mm}\\
D_3^{(1)}&=\displaystyle \frac{1}{27} \theta_{\lambda}^2 u + \lambda \Big(\frac{2}{9}  + \theta_\lambda + \theta_\mu + \theta_\lambda^2
+ 2 \theta_{\lambda\mu} + \theta_\mu^2\Big)u = 0.
\end{cases}
\end{eqnarray}

{\rm (2)}
{\rm (}The period differential equation for $\mathcal{F}_2${\rm )}
 \begin{eqnarray} \label{eq:systemP2}
&&
\begin{cases}
D_1^{(2)}&=\lambda \theta _\mu ^2 +\mu \theta _\lambda  (3\theta _\lambda +4\theta _\mu +1),\cr
D_3^{(2)}&=\lambda \theta _\lambda (3\theta _\lambda +2 \theta _\mu )
+\mu \theta _\lambda (1-\theta _\lambda )+9\lambda ^2(3\theta _\lambda +4\theta _\mu +1)(3\theta _\lambda +4\theta _\mu +2).
\end{cases}
\end{eqnarray}

{\rm (3)}
{\rm (}The period differential equation for $\mathcal{F}_3${\rm )}
 \begin{eqnarray} \label{eq:systemP3}
&&
\begin{cases}
D_1^{(3)}&=\theta _{\lambda}^2-\mu (3\theta _{\lambda }+2\theta _{\mu }+1)(3\theta _{\lambda }+2\theta _{\mu }+2),\cr
D_3^{(3)}&=\theta _{\lambda }(3\theta _\lambda -2\theta _\mu )+9\lambda (3\theta _{\lambda }+2\theta _{\mu }+1)(3\theta _{\lambda }+2\theta _{\mu }+2)+4\mu \theta _{\lambda }(3\theta _{\lambda }+2\theta _{\mu }+1).
\end{cases}
\end{eqnarray}
\end{thm}

\begin{proof}
 We determine  $D_3^{(j)}$ $(j=0,1,2,3)$ by the method of indeterminate coefficients. Set 
$
D=f_1+f_2\theta_\lambda+ f_3 \theta_\mu +f_4 \theta^2_\lambda +f_5 \theta_\lambda \theta_\mu +f_6 \theta^2_\mu,
$
where $f_1\cdots f_6\in\mathbb{C} [\lambda,\mu]$.
Let $j\in \{0,1,2,3\}$.
We can determine the polynomials $f_1,\cdots,f_6$ so that $D$ satisfies $D \eta_j =0$ ($\eta_j$ is given in {\rm Proposition \ref{periodprop}}) and  is  independent of $D^{(j)}_1$. Thus, we obtain the above $D^{(j)}_3$.

In the following, we prove that the spaces of solutions are 4-dimensional.
Let $j\in\{0,1,2,3\}$.
 By making up the Pfaffian system of $D^{(j)}_1 u=D^{(j)}_3 u=0$, we shall show the required statement. 
Set $\varphi = {}^t(1,\theta _\lambda ,\theta _\mu ,\theta _\lambda ^2)$. We obtain the Pfaffian system $\Omega_j=\alpha_j d\lambda +\beta_j d\mu$ with $d\varphi=\Omega_j \varphi$ as follows.
We can check that 
$$
d\Omega_j = \Omega_j \wedge \Omega_j.
$$
Therefore, each system $D_1^{(j)} u = D_3^{(j)} u =0 $ has   the $4$-dimensional space of solution.\\
(0) (The Pfaffian system for $\mathcal{F}_0$)

Setting
\begin{align*}
\begin{cases}
t=\lambda ^2(4\lambda -1)^3-2(2+25\lambda(20\lambda -1))\mu -3125\mu ^2,
\cr
s=1-15\lambda -100\lambda ^2,
\end{cases}
\end{align*}
we have
$$
\alpha_0=\left(
\begin{array}{cccc}
0&1 &0& 0\\
0&0 & 0&1 \\
{a_{11}}/{s}& {a_{12}}/{(2\lambda s)}&{a_{13}}/{(2s)} &{a_{14}}/{(2\lambda s)}\\
{a_{21}}/{(st)}& {a_{22}}/{(2st)}&{a_{23}}/{(2st)} &a_{24}/{(2st)}\\
\end{array} 
\right)
$$
with
\begin{align*}
\begin{cases}
a_{11}=\lambda (1+20\lambda ),\hspace{3cm}
a_{12}=6\lambda ^2+120\lambda ^3+125\mu , \\
a_{13}=5\lambda (3+40\lambda ) ,\hspace{2.8cm}
a_{14}=-(\lambda +16\lambda ^2-80\lambda ^3+125\mu ),\\
a_{21}=-\lambda ^3(2+2125\mu +\lambda (-17+616\lambda -2320\lambda ^2+2500(9+80\lambda )\mu )),\\
a_{22}=-(-2\lambda ^3(-1+4\lambda )(8+5\lambda (-13+4\lambda (83+40\lambda )))\\
\quad \quad\quad+ (-16+5\lambda (94+5\lambda (59+10\lambda (-73+20\lambda (37+160\lambda )))))\mu +3125(-4+5\lambda(21+200\lambda))\mu^2), \\
a_{23}=-\lambda ^3 (22+26875\mu +\lambda (-47+300000\mu +100\lambda (51+4\lambda (-49+20\lambda ) +20000\mu  ))),\\
a_{24}=12ts+3s(15\lambda -2)+2t(-3(1-4\lambda )^2\lambda ^2(-1+10\lambda )+75\lambda (-1+40\lambda )\mu),\\
\end{cases}
\end{align*}
and
$$
\beta_0=\left(
\begin{array}{cccc}
0&0 &1& 0\\
b_{11}/s& b_{12}/(2\lambda s)&b_{13}/(2s) &b_{14}/(2\lambda s)\\
b_{21}/(s)& b_{22}/(\lambda ^2 s)&b_{23}/(s) &b_{24}/(\lambda ^2 s)\\
b_{31}/(ts)& b_{32}/(2\lambda ts)&b_{33}/(2ts) &b_{34}/(2\lambda ts)\\
\end{array} 
\right)
$$
with
\begin{align*}
\begin{cases}
b_{11}=\lambda (1+20\lambda ),\hspace{3.5cm}
b_{12}=6\lambda ^2+120\lambda ^3+125\mu , \\
b_{13}=5\lambda (3+40\lambda ),\hspace{3.3cm}
b_{14}=-(\lambda +16\lambda ^2-80\lambda ^3+125\mu ),\\
b_{21}=-2\lambda (-1+4\lambda ),\hspace{2.9cm}
b_{22}=-(6\lambda ^3(-1+4\lambda )-5\mu +50\lambda \mu) ,\\
b_{23}=-\lambda (-11+20\lambda ),\hspace{2.8cm}
b_{24}=-((1-4\lambda )^2\lambda ^2-(5-50\lambda )\mu) ,\\
b_{31}=-(4(1-4\lambda )^2\lambda ^4(7+20\lambda )\\
\quad\quad\quad -\lambda (-4+25\lambda (-3+2\lambda (-7+20\lambda (1+80\lambda ))))\mu 
+3125\lambda (1+20\lambda )\mu ^2),\\
b_{32}=-(24(1-4\lambda )^2\lambda ^5(7+20\lambda )-2\lambda(-4+5\lambda (8+\lambda (-43+10\lambda (-57+20\lambda (7+160\lambda )))))\mu \\
\quad\quad\quad -125(-4+25\lambda (-3+32\lambda (1+10\lambda )))\mu ^2+390625\mu ^3)),\\
b_{33}=-(4\lambda ^3(-1+4\lambda )(-1+2\lambda (-32+25\lambda (1+12\lambda )))+15625\lambda (3+40\lambda )\mu ^2\\
\quad\quad\quad-5\lambda (-12+5\lambda (-1+10\lambda )(33+20\lambda (23+160\lambda )))\mu) 
,\\
b_{34}=-(4\lambda ^4(-1+4\lambda )^3(7+20\lambda )+3\lambda (-4+\lambda (31-490\lambda +76000\lambda ^3))\mu  \\
\quad\quad\quad+250(-2+25\lambda (-2+\lambda (11+260\lambda )))\mu ^2-390625\mu ^3).
\end{cases}
\end{align*}\\
(1) (The Pfaffian system for $\mathcal{F}_1$)

 Setting
$$
t_1=729\lambda^2 - 54\lambda(27\mu-1) +(1+27\mu)^2,
$$
we have
$$
\alpha_1=\left( 
\begin{array}{cccc}
0&1 &0& 0\\
0&0 & 0&1 \\
-1/9 & -1/2& -1/2 & -(1+27\lambda +27\mu )/(54\lambda)\\
a_{11}/t_1& a_{12}/(2t_1)&a_{23}/(2t_1) &a_{24}/(2t_1)\\
\end{array} 
\right)
$$
with
\begin{eqnarray*}
\begin{cases}
a_{11}=3\lambda (1-27\lambda+27\mu),\quad\quad
a_{12}=3\lambda(5-351\lambda +135\mu), \\
a_{13}= 27\lambda (1-3\lambda +27\mu),\quad\quad
a_{14}=3(-729\lambda ^2 +(1+27\mu)^2),
\end{cases}
\end{eqnarray*}
and
$$
\beta_1=\left( 
\begin{array}{cccc}
0&0 &1& 0\\
-1/9& -1/2 &-1/2 &-(1+27\lambda +27\mu)/(54\lambda)\\
0& 0&0 & \mu/\lambda \\
b_{11}/t_1& b_{12}/(2 t_1)&b_{13}/(2t_1)&b_{14}/(2t_1 )\\
\end{array} 
\right)
$$
with
\begin{eqnarray*}
\begin{cases}
b_{11}=3\lambda(1+27\lambda -27\mu),\quad\quad
b_{12}=27\lambda(1+ 27\lambda -3\mu) ,\\
b_{13}=3\lambda (5+135\lambda -351\mu),\quad\quad
b_{14}=(1+27\lambda)^2 +108(27\lambda -1)\mu -3645\mu ^2.
\end{cases}
\end{eqnarray*}\\
(2) (The Pfaffian system for $\mathcal{F}_2$)

 Setting
\begin{align*}
\begin{cases}
t_2=\lambda ^2(1+27\lambda )^2-2\lambda \mu(1+189\lambda )+(1+576\lambda )\mu ^2-256\mu ^3,\\
s_2=1+108\lambda -288\mu,
\end{cases}
\end{align*}
we have
$$
\alpha_2=\left( 
\begin{array}{cccc}
0&1 &0& 0\\
0&0 & 0&1 \\
a_{11}/s_2& a_{12}/(2\lambda s_2)& a_{13}/(s_2) & a_{14}/(2\lambda s_2)\\
a_{21}/(t_2s_2)& a_{22}/(t_2 s_2)&a_{23}/(t_2 s_2) &a_{24}/(t_2 s_2)\\
\end{array} 
\right)
$$
with
\begin{eqnarray*}
\begin{cases}
a_{11}=-{9\lambda },\hspace{3.1cm}
a_{12}=-(81\lambda ^2+\mu -144\lambda \mu) , \\
a_{13}=-54\lambda ,\hspace{3cm}
a_{14}=-3\lambda (1+27\lambda -144\mu ) +\mu ,\\
a_{21}=-6\lambda ^3(1+1458\lambda ^2-2592\lambda \mu +6\mu (-55+4608\mu )),\\
a_{22}=-3\lambda ^2(11+54\lambda (5+351\lambda ))+\lambda (1+4\lambda (61+810\lambda (5+72\lambda )))\mu  +64(17+2808\lambda )\mu ^3\\
\quad\quad\quad -147456\mu ^4-2(1+9\lambda (53+32\lambda (131+864\lambda )))\mu ^2,\\
a_{23}=-8\lambda ^3((2-27\lambda )^2+9(-133+2160\lambda )\mu +82944\mu ^2),\\
a_{24}=3 r_2 s_2+162\lambda r_2-3\lambda s_2(\lambda +81\lambda ^2+1458\lambda ^3-378\lambda \mu +\mu (-1+288\mu )),
\end{cases}
\end{eqnarray*}
and
$$
\beta_2=\left( 
\begin{array}{cccc}
0&0 &1& 0\\
b_{11}/{s_2}& b_{12}/(2\lambda  s_2)&b_{13}/s_2 &b_{14}/(2\lambda s_2)\\
b_{21}/(s_2)& b_{22}/(\lambda ^2 s_2)&b_{23}/s_2 &b_{24}/(\lambda ^2 s_2)\\
b_{31}/(t_2 s_2)& b_{32}/(2\lambda  t_2 s_2)&b_{33}/(t_2 s_2) &b_{34}/(2\lambda t_2 s_2)\\
\end{array} 
\right)
$$
with
\begin{eqnarray*}
\begin{cases}
b_{11}=-9\lambda ,\hspace{3cm}
b_{12}=-(81\lambda ^2+\mu -144\lambda \mu ),\\
b_{13}=-54\lambda,\hspace{2.8cm}
b_{14}=-3\lambda (1+27\lambda -144\mu )+\mu  ,\\
b_{21}=36\mu ,\hspace{3cm}
b_{22}=\mu (\lambda (-1+54\lambda )+2\mu ), \\
b_{23}=216\mu,\hspace{2.9cm}
b_{24}=(3(1-54\lambda )\lambda -2\mu  )\mu ,\\
b_{31}=3\lambda (81\lambda ^3(1+27\lambda )+\lambda (-1+36\lambda )(-5+108\lambda )\mu +3(-1+32\lambda )(1+432\lambda )\mu ^2+768\mu ^3,\\
b_{32}=2187\lambda ^5(1+27\lambda )-(1+192\lambda (11+1164\lambda ))\mu ^3+256(1+864\lambda )\mu ^4 \\
\quad\quad\quad -\lambda ^2(2+27\lambda (4+9\lambda (77+864\lambda )))\mu 
 +\lambda (5+\lambda (1279+864\lambda (85+864\lambda )))\mu ^2,\\
b_{33}=2\lambda (3\lambda ^2(1+27\lambda )(-1+135\lambda )+2\lambda (23+54\lambda (-11+972\lambda ))\mu \\
\quad\quad\quad +9(-3+64\lambda )(1+432\lambda )\mu ^2+6912\mu ^3,\\
b_{34}=-(-81\lambda ^4(1+27\lambda )^2+\lambda ^2(-7+9\lambda (-58+27\lambda (-125+3456\lambda )))\mu \\
\quad\quad\quad +\lambda (8+9\lambda (425+24192\lambda ))\mu ^2
-(1+3456\lambda (1+162\lambda ))\mu ^3+256(1+1440\lambda )\mu ^4.
\end{cases}
\end{eqnarray*}\\
(3) (The Pfaffian system for $\mathcal{F}_3$)

Setting
\begin{align*}
\begin{cases}
t_3=729\lambda ^2-(4\mu -1)^3+54\lambda (1+12\mu ),\\
s_3=-54\lambda +(1-4\mu )^2,
\end{cases}
\end{align*}
we have
$$
\alpha_3=\left( 
\begin{array}{cccc}
0&1 &0& 0\\
0&0 & 0&1 \\
a_{11}/s_3& a_{12}/(2s_3)&a_{13}/s_3 &a_{14}/(2s_3)\\
a_{21}/(t_3s_3)& a_{22}/(t_3s_3)&a_{23}/(t_3s_3) &a_{24}/(t_3s_3)\\
\end{array} 
\right)
$$
with
\begin{eqnarray*}
\begin{cases}
a_{11}=9\lambda, \hspace{3cm}
a_{12}=81\lambda +4(1-4\mu )\mu,\\ 
a_{13}=27\lambda ,\hspace{2.9cm}
a_{14}=3+81\lambda -48\mu ^2,\\
a_{21}=-2\lambda (-2187\lambda ^2+27\lambda (4\mu -9)(4\mu -1)-(-1+4\mu )^3(3+8\mu )),\\
a_{22}=3\lambda (9477\lambda ^2+(1-4\mu )^2(-11+4\mu (-9+16\mu ))-27\lambda (25+4\mu (-31+40\mu ))),\\
a_{23}=2\lambda (729\lambda ^2+(-1+4\mu )^3(11+16\mu )+27\lambda (-1+4\mu )(19+20\mu )),\\
a_{24}=81\lambda (-2+27\lambda +8\mu )(1+27\lambda -16\mu ^2),
\end{cases}
\end{eqnarray*}
and
$$
\beta_3=\left( 
\begin{array}{cccc}
0&0 &1& 0\\
b_{11}/s_3& b_{12}/(2s_3)&b_{13}/s_3 &b_{14}/(2s_3)\\
b_{21}/s_3& b_{22}/s_3&b_{23}/s_3 &b_{24}/s_3\\
b_{31}/(t_3 s_3)& b_{32}/(2 t_3 s_3)&b_{33}/(t_3 s_3)&b_{34}/(2t_3 s_3 )\\
\end{array} 
\right)
$$
with
\begin{eqnarray*}
\begin{cases}
b_{11}=9\lambda , \hspace{4.4cm}
b_{12}=81\lambda +4(1-4\mu )\mu ,\\
b_{13}=27\lambda ,\hspace{4.3cm}
b_{14}=3+81\lambda -48\mu ^2,\\
b_{21}=-2\mu (-1+4\mu ), \hspace{2.7cm}
b_{22}=-3\mu (-3+4\mu ),\\
b_{23}=-6\mu (-1+4\mu ), \hspace{2.7cm}
b_{24}=9\mu(3+4\mu ),\\
b_{31}=-3\lambda (2187\lambda ^2+32(1-4\mu )^2\mu (1+\mu )+27\lambda (3+16\mu (2+\mu )))\\
b_{32}=-9\lambda (6561\lambda ^2-81\lambda (-3+4\mu )(1+8\mu )+4\mu (-1+4\mu )(-33+4\mu (-3+16\mu ))) ,\\
b_{33}=-3\lambda (3645\lambda ^2+2(1-4\mu )^2(1+16\mu (3+2\mu ))+27\lambda (7+16\mu (5+9\mu ))) ,\\
b_{34}=-r_3 s_3+r_3(-8+351\lambda +32\mu )+s_3(9(729\lambda ^2+(1-4\mu )^2+54\lambda (1+8\mu )).
\end{cases}
\end{eqnarray*}
\end{proof}

\begin{rem}
By changing the system $\varphi = ^t(1,\theta _\lambda ,\theta _\mu ,\theta _\lambda ^2)$ to other ones, we see that $s=0$ is not a singularity. Together with the singularities of $\theta_\lambda$ and  $\theta_\mu$, we obtain the singular locus of the system {\rm (\ref{eq:systemP4})}:
\begin{align}\label{eq:P4singlar}
\lambda=0, \quad \mu=0, \quad \lambda ^2(4\lambda -1)^3-2(2+25\lambda(20\lambda -1))\mu -3125\mu ^2=0.
\end{align}
This is the locus mentioned in {\rm Remark \ref{Lambda remark}}.

By the same way, from the Puffian systems in the above proof, we obtain 
the singular locus of the system {\rm (\ref{eq:systemP1})}:
$$
\lambda=0,\quad \mu=0,\quad 729\lambda^2 - 54\lambda(27\mu-1) +(1+27\mu)^2=0,
$$
the singular locus of the system {\rm (\ref{eq:systemP2})}:
$$
\lambda =0,\quad \mu =0, \quad
\lambda ^2(1+27\lambda )^2-2\lambda \mu(1+189\lambda )+(1+576\lambda )\mu ^2-256\mu ^3=0,
$$
and the singular locus of the system {\rm (\ref{eq:systemP3})}:
$$
\lambda =0,\quad\mu =0,\quad
729\lambda ^2-(4\mu -1)^3+54\lambda (1+12\mu )=0.
$$
Omitting these locus from $\mathbb{C}^2$ we have the domain $\Lambda_j\hspace{1mm}(j=1,2,3) $ in  {\rm (\ref{Lambda_1}), (\ref{Lambda_2})} and {\rm (\ref{Lambda_3})}. 
\end{rem}

\begin{rem}
 Takayama and  Nakayama {\rm  \cite{TakayamaNakayama}} determined the systems of differential equations for
 the  Fano polytopes with  $6$ vertices
 by their new approximation method, that is a  special use  of  $D$-module algorithm.
\end{rem}

\section{Monodromy groups}

We defined the projective monodromy groups of our period mappings in Section 2.
Those are nothing but the projective monodromy groups of the period differential equations determined by  the previous section.
We determine them in this section.
We make a precise argument  only for  the period mapping $\Phi:\Lambda_0 \rightarrow \mathcal{D}_0$ for $\mathcal{F}_0$. 
In this section, we set $\mathcal{F}:=\mathcal{F}_0, S(\lambda,\mu):=S_0(\lambda,\mu),\Lambda:=\Lambda_0, L:=L_0,A:=A_0$  and $\mathcal{D}:=\mathcal{D}_0$.

First, take a generic point $(\lambda_0,\mu_0)\in \Lambda$. Let $\check{S}=S(\lambda_0,\mu_0)$ be a reference surface. Set  $\check{L}={\rm NS}(\check{S})$ which is generated by the system (\ref{N-Sbasis}). 
Recalling the argument of Section 2 and 3, we have a $\mathbb{Z}$-basis 
$\{\gamma_1,\cdots,\gamma_{22}\}$ of $H_2(\check{S},\mathbb{Z})$ with $\langle\gamma_5,\cdots,\gamma_{22}\rangle_\mathbb{Z}=\check{L}$.
 
  $A(=A_0)$ is the intersection matrix  of  the transcendental lattice given in Theorem \ref{latticeThm}.
Set
\begin{eqnarray}\label{eq:O(L'')}
PO(A,\mathbb{Z})=\{g\in GL(4,\mathbb{Z}) | {}^tg Ag=A\}.
\end{eqnarray}
It acts on $\mathcal{D}$ by
$$
{}^t \xi \mapsto g{}^t\xi \quad (\xi\in \mathcal{D}, g\in PO(A,\mathbb{Z})).
$$

Recall that $\mathcal{D}$ is composed of two connected components:
$$
\mathcal{D} = \mathcal{D}_+ \cup \mathcal{D}_-.
$$

\begin{df}
Let $PO^+(A,\mathbb{Z})$ denote the subgroup of $PO(A,\mathbb{Z})$ given by
$$\{g\in PO(A,\mathbb{Z})| g(\mathcal{D}_{\pm })=\mathcal{D}_{\pm } \}.$$
\end{df}

\begin{rem}
$PO(A,\mathbb{Z})$ is generated by the system: 
\begin{align*}
\begin{cases}
&G_1=
\begin{pmatrix}
1&1&-1&2\\
0&1&0&0\\
0&0&1&0\\
0&1&0&1
\end{pmatrix},
G_2=
\begin{pmatrix}
1&-1&-2&-1\\
0&1&0&0\\
0&1&1&0\\
0&0&0&1
\end{pmatrix},
G_3=
\begin{pmatrix}
0&1&0&0\\
1&0&0&0\\
0&0&1&0\\
0&0&1&-1
\end{pmatrix},\\
&H_1=
\begin{pmatrix}
1&0&0&0\\
0&1&0&0\\
0&0&-1&0\\
0&0&-1&1
\end{pmatrix},
H_2=
\begin{pmatrix}
0&1&0&0\\
1&0&0&0\\
0&0&1&0\\
0&0&0&1
\end{pmatrix}.
\end{cases}
\end{align*}
$G_1, G_2, G_3, H_2$ generate $PO^+(A,\mathbb{Z})$
(see {\rm \cite{IshigeChiba}} or {\rm \cite{Matsumoto}}).
\end{rem}

In the following, we show that the projective monodromy group of our period mapping is isomorphic to the  group $PO^+(A,\mathbb{Z})$. To prove this, we apply the Torelli type theorem for polarized $K3$ surfaces.

\subsection{The Torelli theorem for P-marked $K3$ surfaces}
First, we state necessary properties of polarized $K3$ surfaces.

\begin{df}
Let $S$ be an algebraic $K3$ surface. An isomorphism $\psi:H_2(S,\mathbb{Z}) \rightarrow H_2(\check{S},\mathbb{Z})$ is said to be a P-marking if we have 
\par
{\rm (i)} $\psi ^{-1} (\check{L})\subset {\rm NS} (S)$,
\par
{\rm (ii)} $\psi^{-1} (F),\psi^{-1} (O), \psi^{-1} (Q),\psi^{-1} (R),\psi^{-1}(b_j)$ and $\psi^{-1}(b_j')$  $(1\leq j\leq 7)$ 
are all effective divisors,
\par
{\rm (iii)} $\psi^{-1} (F)$ is nef.  Namely, $(\psi^{-1} (F) \cdot C)\geq 0$ for any effective class $C$.
\end{df}

A pair $(S,\psi)$ of a $K3$ surface and a P-marking is called a P-marked $K3$ surface.
A S-marked $K3$ surface $(S(\lambda,\mu),\psi)$ is a P-marked $K3$ surface.

\begin{df}\label{df:P-marked}
Two P-marked $K3$ surfaces
$(S_1,\psi_1)$ and $(S_2,\psi_2)$ are said to be isomorphic if there is a biholomorphic mapping
$f:S_1\rightarrow S_2$ with 
$$
\psi_2 \circ f_* \circ \psi_1^{-1} = id _{H_2(\check{S},\mathbb{Z})}.
$$
Two P-marked $K3$ surfaces 
$(S_1,\psi_1)$ and $(S_2,\psi_2)$ are said to be equivalent if there is a biholomorphic mapping
$f:S_1\rightarrow S_2$ with 
$$
\psi_2 \circ f_* \circ \psi_1^{-1} |_{\check{L}}= id _{\check{L}}.
$$
\end{df}

The period of a P-marked $K3$ surface $(S,\psi)$ is defined by
\begin{eqnarray}\label{P-marked period map}
\Phi(S,\psi)= \Big(\int _{\psi^{-1}(\gamma_1)} \omega :\cdots :\int_{\psi^{-1}(\gamma_4)} \omega \Big).
\end{eqnarray}

We use some general facts. These are exposed in \cite{Koike-Shiga}.

\begin{prop}{\rm (Pjatecki\u{i}-\v{S}apiro and \v{S}afarevi\v{c} \cite{Pj-S})
} \label{prop:Pj-S}
Let $S$ be a $K3$ surface.
\par
{\rm (1)} Suppose $C\in{\rm NS}(S)$ satisfies $(C\cdot C)=0$ and $C\not=0$. Then there exists an isometry $\gamma $ of ${\rm NS}(S)$ such that $\gamma (C)$ becomes to be effective and nef.
\par
{\rm (2)} Suppose $C\in {\rm NS}(S)$ is effective, nef and $(C\cdot C)=0$. Then, for certain $m\in \mathbb{N}$ and an elliptic curve $E\in S$,  we have $C=m[E]$.
\par
{\rm (3)} A linear system of an elliptic curve $E$ on $S$ determines an elliptic fibration $S\rightarrow \mathbb{P}^1(\mathbb{C})$.
\end{prop}

\begin{prop}\label{ prop:generalfibre}
A P-marked $K3$ surface $(S,\psi)$ is realized as an elliptic $K3$ surface which has $\psi^{-1}(F)$ as a general fibre. Especially, if $S$ is realized as a $K3$ surface $S(\lambda,\mu)$ by the Kodaira normal form  for some $(\lambda,\mu)\in\Lambda$, it is a S-marked $K3$ surface.
\end{prop}

\begin{proof}
Set $C=\psi^{-1}(F)\in \text{Div}(S)$. By Definition \ref{df:P-marked}, $C$ is effective, nef and $(C\cdot C)=0$. 
According to Proposition \ref{prop:Pj-S} (2), there exists a positive integer $m$ and an elliptic curve $E$ such that $C=m[E]$.
Since
$$
m(E\cdot \psi^{-1}(O))=(C\cdot \psi^{-1}(O))=(F\cdot O)=1,
$$
 we deduce that  $m=1$. Proposition \ref{prop:Pj-S} (3) says that there is an elliptic fibration $\pi: S \rightarrow \mathbb{P}^1(\mathbb{C})$ which has $C=\psi^{-1}(F)$ as a general fibre.
\end{proof}

Let $X$ be the isomorphic classes of P-marked $K3$ surfaces and set
$$
[X]= X/\text{ P-marked equivalence}.
$$
By (\ref{P-marked period map}), we obtain our period mapping $\Phi:X \rightarrow \mathbb{P}^3(\mathbb{C}).$

\begin{thm}\label{thm:p-marked Torelli}
{\rm (The Torelli theorem for polarized $K3$ surfaces)}
\par
{\rm (1)} $\Phi (X)\subset \mathcal{D}$.
\par
{\rm (2)} $\Phi : X\rightarrow \mathcal{D}$ is a bijective correspondence.
\par
{\rm (3)} Let $S_1$ and $S_2$ be algebraic $K3$ surfaces. Suppose an isometry $\varphi : H_2(S_1,\mathbb{Z})\rightarrow H_2(S_2,\mathbb{Z})$ preserves ample classes. Then there exists a biholomorphic map  $f:S_1\rightarrow S_2$ such that $\varphi=f_*$.
\end{thm}

Here, we prove the following two key lemmas. 

\begin{lem}\label{lem:P-marked}
A P-marked $K3$ surface $(S,\psi)$ is equivalent to the P-marked reference surface $(\check{S},\check{\psi})$ if and only if  $\Phi(S,\psi)=g\circ \Phi(\check{S},\check{\psi})$
for some $g \in PO(A,\mathbb{Z})$.
\end{lem}

\begin{proof}
The necessity is clear.
We prove the sufficiency. Suppose $\Phi(\check{S},\check{\psi})=p\in \mathcal{D}$. Take $g\in  PO(A,\mathbb{Z})$. 
According to Theorem \ref{thm:p-marked Torelli} (2), we  take a P-marked $K3$ surface $(S_{g},\psi_g)$ such that  $\Phi(S_{g},\psi_g)=g\circ \Phi(\check{S},\check{\psi})$. Let $L_t$ be the transcendental lattice given by (\ref{df:L''}).
Note $g\in \text{Aut}(L_t)=PO(A,\mathbb{Z}).$
Due to   Nikulin \cite{Nikulin}, 
$g :L_t\rightarrow L_t$ is extended to an isomorphism $\hat{g}:H_2(\check{S},\mathbb{Z})\rightarrow H_2({S}_g,\mathbb{Z})$ which preserves the N\'{e}ron-Severi lattice $L$. Then, by Theorem \ref{thm:p-marked Torelli} (3), there is a biholomorphic mapping $f:\check{S}\rightarrow S_{g}$ such that $f_*=\hat{g}$. Therefore, two P-marked $K3$ surfaces $(\check{S},\check{\psi})$ and $(S_{g},\psi_g)$ are equivalent.
\end{proof}

\begin{rem}
$PO(A,\mathbb{Z})$ is a reflection group (see {\rm \cite{Matsumoto}}).
\end{rem}

According to the Torelli theorem and Lemma \ref{lem:P-marked}, we  identify $[X]$ with $\mathcal{D}/PO(A,\mathbb{Z})$.

\begin{lem}\label{lem:key lemma}
Let $(S,\psi)$ be a P-marked $K3$ surface which is equivalent to $(\check{S},\check{\psi})$. Then $(S,\psi)$ has a unique canonical elliptic fibration $(S,\pi,\mathbb{P}^1(\mathbb{C}))$ that is given by the Kodaira normal form of $\check{S}=S(\lambda_0,\mu_0)$ not coming from any other $(\lambda,\mu ) \in \Lambda$.
\end{lem}

\begin{proof}
From Proposition \ref{ prop:generalfibre},  $(S,\psi)$ ($(\check{S},\check{\psi})$, resp.) has an elliptic fibration $(S,\pi,\mathbb{P}^1(\mathbb{C}))$ ($(\check{S},\check{\pi},\mathbb{P}^1(\mathbb{C}))$, resp.) with a general fibre $\psi^{-1}(F)$ ($\check{\psi}^{-1}(F)$, resp.). 
Because $(S,\psi)$ and $(\check{S},\check{\psi})$ are equivalent as P-marked $K3$ surfaces, we have a biholomorphic mapping $f:S\rightarrow \check{S}$ such that 
$$
\check{\psi} \circ f_*=\psi  \quad (f_*:H_2(S,\mathbb{Z})\simeq  H_2(\check{S},\mathbb{Z})).
$$
So, we have
$$
f_* = \psi.
$$
It means that $f$ preserves general fibres of $S$ and $\check{S}$.
According to the uniqueness of the fibration (Lemma \ref{lem:K3fibre}), $(S,\pi,\mathbb{P}^1(\mathbb{C}))$ and $(\check{S},\check{\pi},\mathbb{P}^1(\mathbb{C}))$ are isomorphic as elliptic surfaces.
Therefore, there exists  $\varphi\in\text{Aut}(\mathbb{P}^1(\mathbb{C}))$
 such that $\varphi\circ \pi=\pi_0 \circ f$.
 
Let 
$
y^2 = 4 x^3 -g_2(z) x -g_3(z)
$
($
y^2 = 4 x^3 -\check{g}_{2}(z) x -\check{g}_{3}(z),
$ resp.)
be the Kodaira normal form of $(S,\pi,\mathbb{P}^1(\mathbb{C}))$ ($(\check{S},\check{\pi},\mathbb{P}^1(\mathbb{C}))$, resp.).
According to Proposition \ref{prop:j-inv}, we  assume $\pi^{-1}(0)=I_3$ and $\pi^{-1}(\infty)=I_{15}$. So as in the proof of Lemma \ref{lem:elliptic1-1}, $\varphi$ is given by $z\mapsto a z$ $(a\in \mathbb{C}-0)$. 
Let $j$ ($\check{j}$, resp.) be the $j$-invariant  and  $D$ ($\check{D}$, resp.) be 
the discriminant of $S$ ($\check{S}$, resp.).
By Proposition \ref{prop:j-inv}, we have $D=\check{D}\circ \varphi $ and $j=\check{j} \circ \varphi$.
Observing the expressions (\ref{eq:P4kodaira2}), (\ref{eq:P4discriminant}) around $z=\infty$ and the definition of $j$-function (\ref{eq:j-function}), we have $a^3=1$. By the transformation $z\mapsto \omega z$ or $z \mapsto \bar{\omega} z$ (where $\omega$ is a cubic root of unity), we  assume $a=1$.
Comparing $j$ with $\check{j}$ and $D$ with $\check{D}$, we have $g_2^3=\check{g}_{2}^3$ and $g_3^2=\check{g}_{3}^2$. By the transformations in the form $x\mapsto \omega x$ or $x\mapsto \bar{\omega} x$ or $y\mapsto - y$, we obtain $g_2=\check{g}_{2}$ and $g_3=\check{g}_{3}$. 
 Hence, as in the proof of Lemma \ref{lem:elliptic1-1}, we have the required statement.
\end{proof}

\begin{rem}
According to the above two lemmas, $\Lambda =\Lambda_0$ is embedded in $[X]$. 
\end{rem}

\subsection{Projective monodromy groups}

\begin{thm}\label{thm:monodromy}
The projective monodromy group of the period mapping $\Phi:\Lambda\rightarrow \mathcal{D}$ is isomorphic to $PO^+(A,\mathbb{Z})$.
\end{thm}

\begin{proof}
Let $*=(\lambda_0,\mu_0)$ be a generic point of $\Lambda$.  
Set $\check{S}=S(\lambda_0,\mu_0)$.
Note that $\text{NS}(\check{S})\simeq {L}$.
Let $G$ be the projective monodromy group induced from the fundamental group $\pi_1(\Lambda, \ast)$ (see Definition \ref{Def:monodromy}).
We have clearly the inclusion $G \subset PO^+(A,\mathbb{Z})$.

Therefore, we prove the converse inclusion $PO^+(A,\mathbb{Z})\subset G$.
Take an element $g\in PO^+(A,\mathbb{Z})$, and let $p=\Phi(\check{S},\check{\psi}) \in \mathcal{D}$ and let $q= g(p)\in\mathcal{D}$. $p,q$ are in the same connected component of $\mathcal{D}$.
So we suppose that $p,q \in \mathcal{D}^+$. Let $\alpha$ be an arc connecting $p$ and $q$ in $\mathcal{D}^+$. 
By the Torelli theorem, we obtain $[\Phi^{-1}(\alpha)] \subset [X]$. 
By Lemma \ref{lem:P-marked} and Lemma  \ref{lem:key lemma}, we have $q=\Phi (\check{S},\psi)$ so that $(\check{S},\psi)$ is equivalent to $(\check{S},\check{\psi})$.
Hence, the end point of $[\Phi^{-1}(\alpha)]$ is $(\lambda_0,\mu_0)$.

Next, we show that there is  $\alpha$ such that $[\Phi^{-1}(\alpha)]\subset \Lambda$. 
For this purpose, it is enough to show that $\Lambda$ is a Zariski open set in some compactification $K$ of $[X]$.
Here, we note that the compact $(\lambda,\mu)$ space $\mathbb{P}^2(\mathbb{C})$ and $K$ are birationally equivalent and they contain $\Lambda$ 
as a common open set. $\Lambda$ is a Zariski open set in $\mathbb{P}^2(\mathbb{C})$. Hence, $\Lambda$  is Zariski open in $K$ also.
Therefore, we obtain the required inclusion.
\end{proof}

\vspace{5mm}

We have the elliptic fibration
(\ref{P1preKodaira}) ((\ref{P2preKodaira}), (\ref{P3preKodairaanother}), resp.) for $\mathcal{F}_1$ ($\mathcal{F}_2$,  $\mathcal{F}_3$, resp.).
Using these fibrations, we can define the P-markings for $\mathcal{F}_j$ $(j=1,2,3)$.
Moreover, 
as we prove  Lemma \ref{lem:key lemma},
so  we can prove the corresponding lemmas through  observations of the coefficients of the Kodaira normal forms of elliptic fibrations for $\mathcal{F}_j$ $(j=1,2,3)$.
Therefore, we have

\begin{thm}\label{monodromytheorem}
Let $j\in \{1,2,3\}.$
The projective monodromy group of the period mapping for the family $\mathcal{F}_j$ is equal to $PO^+(A_j,\mathbb{Z})$.
\end{thm}

\begin{rem}
This   is essentially noticed  in the research    of  Ishige {\rm \cite{Ishige}} on the family of $K3$ surfaces coming from  the polytope $P_4$.
He found this result by a  computer-aided approximation of a generator system of the monodromy group.
However, it is not given an exact error estimation there.
So, for our cases $P_0,P_1,P_2$ and $P_3$, we give here a proof  based on the Torelli  theorem for  polarized  $K3$ surfaces. 
\end{rem}

\section{Period differential equation and the Hilbert modular  orbifold for the field $\mathbb{Q}(\sqrt{5})$}

Let $\mathcal{O}$ be the ring of integers in the real quadratic field $\mathbb{Q}(\sqrt{5})$.
Set $\mathbb{H}_{\pm }=\{z\in \mathbb{C}|\pm {\rm Im}( z) >0\}.$ The Hilbert modular group $PSL(2,\mathcal{O})$ acts on 
$(\mathbb{H}_+\times\mathbb{H}_+)\cup (\mathbb{H}_-\times\mathbb{H}_-)$ by 
\begin{eqnarray*}
\begin{pmatrix}
\alpha &\beta \\
\gamma &\delta 
\end{pmatrix}
:
(z _1, z_2)\mapsto 
\Big(\frac{\alpha z_1 +\beta}{\gamma z_1 +\delta} ,\frac{\alpha ' z_2+\beta'}{\gamma' z_2+\delta' }\Big),
\end{eqnarray*}
for 
$
g = \begin{pmatrix}\alpha&\beta\\ \gamma&\delta\end{pmatrix} \in PSL(2,\mathcal{O}),
$
where $'$ means the conjugate in $\mathbb{Q}(\sqrt{5})$.

Set 
$$
W = \begin{pmatrix}
1&1\\
(1-\sqrt{5})/2 & (1+\sqrt{5})/2
\end{pmatrix}.
$$ 
It holds 
$$
A =A_0= U \oplus \begin{pmatrix} 2&1\\ 1&-2 \end{pmatrix} = U \oplus WU{}^tW .
$$ 
The correspondence 
$$
(z_1,z_2) \rightarrow (I_2\oplus {}^tW^{-1})\begin{pmatrix}z_1 z_2\\ -1\\ z_1 \\ z_2 \end{pmatrix}
$$
defines a biholomorphic isomorphism
$$
\iota :(\mathbb{H}_+ \times \mathbb{H}_+) \cup (\mathbb{H}_- \times \mathbb{H}_-) \rightarrow \mathcal{D}.
$$
Setting
\begin{align*}
\begin{cases}\vspace*{0.1cm}&\tau: (z_1,z_2)\rightarrow (z_2,z_1),\\
&\tau':(z_1,z_2)\rightarrow \Big(\displaystyle  \frac{1}{z_1},\frac{1}{z_2}\Big),
\end{cases}
\end{align*}
we set
$$
\rho(g)=\iota \circ g \circ \iota^{-1}
$$
for $g\in \langle PSL(2,\mathcal{O}), \tau, \tau' \rangle.$
We have $\rho(\langle PSL(2,\mathcal{O}), \tau, \tau' \rangle)=PO(A,\mathbb{Z})$. 
Put $\mathbb{H}=\mathbb{H}_+$.
The pair $(\iota, \rho )$ gives a modular isomorphism 
\begin{eqnarray}\label{modulariso}
(\mathbb{H} \times \mathbb{H}, \langle PSL(2,\mathcal{O}),
 \tau \rangle)
\simeq
(\mathcal{D}_+, PO^+(A,\mathbb{Z})).
\end{eqnarray}

There are several researches on the  Hilbert modular   orbifolds  for the field $\mathbb{Q}(\sqrt{5})$. 
Hirzebruch \cite{Hirzebruch} studied the orbifold $(\mathbb{H}\times\mathbb{H})/ \langle  \Gamma ,\tau \rangle$  (the group $\Gamma$ is given in (\ref{eq:dfgamma})).
 There, he used Klein's icosahedral polynomials.  
  Kobayashi,  Kushibiki and  Naruki \cite{KobaNaru}  studied the orbifold  $(\mathbb{H} \times
 \mathbb{H}) / \langle PSL(2, \mathcal{O}),\tau \rangle $ and determined its branch divisor in terms of the icosahedral invariants.
   Sato  \cite{Sato} gave the uniformizing differential equation (see Definition \ref{df:UDE}) of the orbifold   $(\mathbb{H} \times
 \mathbb{H}) / \langle PSL(2, \mathcal{O}),\tau \rangle $.

Because of  the modular isomorphism (\ref{modulariso}) and Theorem \ref{thm:monodromy},  our period differential equation (\ref{eq:systemP4}) for the family 
$\mathcal{F}_0=\{S_0(\lambda,\mu)\}$
 should be connected  to the uniformizing differential equation of the orbifold $(\mathbb{H} \times \mathbb{H}) / \langle PSL(2, \mathcal{O}),\tau \rangle $.

 In this section, we realize the explicit relation between our period differential equation and the  uniformizing differential equation of the orbifold  $(\mathbb{H} \times \mathbb{H})/  \langle PSL(2,\mathcal{O}),
 \tau \rangle$. 
We give the exact  birational transformation (\ref{trans}) from our $(\lambda,\mu)$-space to $(x,y)$-space, where $(x,y)$ are  affine coordinates expressed by  Klein's icosahedral polynomials in (\ref{(x,y)}). Moreover, we show that  the uniformizing differential equation with the normalization factor (\ref{normal}) coincides with our period differential equation (\ref{eq:systemP4}).

\subsection{Linear differential equations in 2 variables of rank 4}

First, we survey the study of  Sasaki and  Yoshida \cite{SasakiYoshida}.  It supplies a fundamental tool for the research on  uniformizing differential equations of the  Hilbelt modular orbifolds.  

We consider a system of  linear differential equations
\begin{eqnarray}\label{linearDE}
\begin{cases}
& Z_{XX}= l Z_{XY} + a Z_X + b Z_Y +p Z,\\
& Z_{YY}=m Z_{XY} + c Z_X +d Z_Y+q Z,
\end{cases}
\end{eqnarray}
where $(X,Y)$ are independent variables and $Z$ is the unknown. We assume  its  space of solutions is  $4$-dimensional.

\begin{df}\label{df:conformal}
We call the  symmetric $2$-tensor 
\begin{eqnarray*}
 l (dX)^2 +2 (dX)(dY) +m (dY)^2
\end{eqnarray*}
the holomorphic conformal structure  of {\rm (\ref{linearDE})} .
\end{df}

\begin{rem}
The above symmetric $2$-tensor is equal to the holomorphic conformal structure of the complex surface patch 
 embedded in $\mathbb{P}^3(\mathbb{C})$ defined by the projective solution of  {\rm (\ref{linearDE})}.
\end{rem}

\begin{df}
Let $Z_0, Z_1,Z_2$ and $Z_3$ be linearly independent solutions of {\rm (\ref{linearDE})}. Put $Z={}^t(Z_0,Z_1,Z_2,Z_3)$. The  function 
\begin{eqnarray*}
e^{2\theta} ={\rm det} (Z,Z_X,Z_Y,Z_{XY})
\end{eqnarray*}
 is called the normalization factor of {\rm (\ref{linearDE})}.
 \end{df}

\begin{prop}{\rm (\cite{SasakiYoshida} Proposition 4.1,  see also \cite{Sato} p.181)}\label{prop:ABCD}
The surface patch by the projective solution of  {\rm (\ref{linearDE})} is a part of non degenerate quadratic surface in $\mathbb{P}^3(\mathbb{C})$ if and only if
\begin{align*}
\begin{cases}
\vspace*{0.2cm}
a=\displaystyle \frac{\partial}{\partial X}\Big(\frac{1}{4} \xi +\theta\Big) -\frac{l}{2}\frac{\partial}{\partial Y}\Big({\rm log} (l)-\frac{1}{4} \xi +\theta\Big),\\
\vspace*{0.2cm}
b=\displaystyle \frac{l}{2}\frac{\partial}{\partial X}\Big({\rm log} (l)-\frac{3}{4} \xi -\theta \Big),\\
\vspace*{0.2cm}
c=\displaystyle \frac{m}{2}\frac{\partial }{\partial Y}\Big({\rm log}(m)-\frac{3}{4} \xi -\theta \Big),\\
d=\displaystyle \frac{\partial}{\partial Y}\Big(\frac{1}{4} \xi +\theta  \Big)- \frac{m}{2}\frac{\partial}{\partial X}\Big({\rm log} (m) -\frac{1}{4} \xi +\theta \Big),
\end{cases}
\end{align*}
where $\xi ={\rm log}(1-l m)$. 
\end{prop}

\begin{prop}\label{LM} {\rm (\cite{SasakiYoshida}  Section 3)}
Perform a coordinate change of the equation  {\rm (\ref{linearDE})} from $(X,Y)$ to $(U,V)$ and denote the coefficients of the transformed equation by the same letter with bars. Then 
\begin{align*}
\begin{cases}
\vspace*{0.2cm}
\overline{l}=-\lambda/\nu ,\quad
\overline{m}=-\mu/\nu,\\
\vspace*{0.2cm}
\overline{a}=(R(U)\beta -S(U)\alpha)/\nu,\quad
\overline{b}= (R(V)\beta -S(V)\alpha)/\nu,\\
\vspace*{0.2cm}
\overline{c}= (S(U)\gamma -R(U)\delta))/\nu,\quad
\overline{d}= (S(V)\gamma -R(V)\delta)/\nu,\\
\overline{p}= (\alpha q -\beta p)/\nu,\quad
\overline{q}= (\delta p -\gamma q)/\nu ,
\end{cases}
\end{align*}
where
\begin{align*}
\begin{cases}
\Delta = U_X V_Y - U_Y V_X,\\
\lambda=l V_Y^2 - 2 V_X V_Y +m V_X^2,\\
\mu=l U_Y^2 -2 U_X U_Y + m U_X^2,\\
\nu=l U_Y V_Y - U_X V_Y -U_Y V_X +m U_X V_X,\\
\end{cases}
\end{align*}
and
\begin{align*}
\begin{cases}
\vspace*{0.1cm}
\alpha =(V_X^2 -l V_X  V_Y)/\Delta ,\quad
\beta=(V_Y^2 -m V_X  V_Y)/\Delta,\\
\vspace*{0.1cm}
\gamma=(U_X^2 - l U_X U_Y)/\Delta,\quad
\delta=(U_Y^2-m U_X U_Y)/\Delta,\\
\vspace*{0.1cm}
R(U)=U_{XX} -(l U_{X Y} +a U_X +b U_Y),\\
\vspace*{0.1cm}
S(U)=U_{YY} -(m U_{X Y} +c U_X +d U_Y),\\
\vspace*{0.1cm}
R(V)=V_{XX} - (l V_{X Y} + a V_X +b V_Y),\\
S(V) =V_{YY} -(m V_{X Y}+c V_X +d V_Y)  .
\end{cases}
\end{align*}
\end{prop}

\subsection{Uniformizing differential equation of the Hilbert modular orbifold  $(\mathbb{H}\times \mathbb{H}) / \langle  PSL(2,\mathcal{O}),\tau \rangle$}

The quotient space $(\mathbb{H}\times\mathbb{H})/\langle PSL(2,\mathcal{O}), \tau\rangle$ carries the structure of an orbifold.Let us sum up the facts about the orbifold  $(\mathbb{H}\times \mathbb{H}) / \langle  PSL(2,\mathcal{O}),\tau\rangle$ and the result of  Sato \cite{Sato}  on the uniformizing differential equation.

Set
\begin{eqnarray}\label{eq:dfgamma}
\Gamma=\Big\{\begin{pmatrix}\alpha & \beta 
\\ \gamma & \delta \end{pmatrix}\in PSL(2,\mathcal{O})
\Big|  \hspace{0.1cm}\alpha \equiv \delta \equiv 1, \hspace{0.1cm}\beta \equiv \gamma \equiv 0  \quad(\text{mod}\sqrt{5})\Big\}.
\end{eqnarray}
$\Gamma$ is a normal subgroup of $PSL(2, \mathcal{O})$.
The quotient group $PSL(2,\mathcal{O})/\Gamma$ is isomorphic to the alternating group $\mathcal{A}_5$ of degree $5$. $\mathcal{A}_5$ is isomorphic to the icosahedral group $I$.
Let $\overline{M}$ be a compactification of an orbifold $M$. 
 Hirzebruch \cite{Hirzebruch} showed that  $\overline {\mathbb{H}\times\mathbb{H}/\langle \Gamma,\tau\rangle}$ is isomorphic to $\mathbb{P}^2(\mathbb{C})$. Therefore, $\mathbb{P}^2(\mathbb{C})$ admits an action of the alternating  group $\mathcal{A}_5$.   This action is equal to the action of the  icosahedral group $I$  on $\mathbb{P}^2(\mathbb{C})$  introduced by F. Klein.
We list  Klein's $I$-invariant polynomials on $\mathbb{P}^2(\mathbb{C})=\{(\zeta_0:\zeta_1:\zeta_2)\}$:
\begin{align*}
\begin{cases}
\vspace*{0.2cm}
\mathfrak{A}(\zeta_0:\zeta_1:\zeta_2)=\zeta_0^2 +\zeta_1 \zeta_2,\\
\vspace*{0.2cm}
\mathfrak{B}(\zeta_0:\zeta_1:\zeta_2)=8\zeta_0^4 \zeta_1 \zeta_2 -2 \zeta_0^2 \zeta_1^2 \zeta_2^2 +\zeta_1^3 \zeta_2^3 -\zeta_0 (\zeta_1^5 + \zeta_2 ^5),\\
\mathfrak{C}(\zeta_0:\zeta_1:\zeta_2)=320\zeta_0^6 \zeta_1^2 \zeta_2^2 -160 \zeta_0^4 \zeta_1^3 \zeta_2^3 +20 \zeta_0^2 \zeta_1^4 \zeta_2^4 +6 \zeta_1^5 \zeta_2^5\\
\vspace*{0.2cm}
\quad\quad\quad\quad \quad\quad\quad\quad -4\zeta_0(\zeta_1^5+\zeta_2^5) (32 \zeta_0^4-20 \zeta_0^2 \zeta_1 \zeta_5 +5\zeta_1^2 \zeta_2^2)+\zeta_1^{10}+\zeta_2^{10},\\
12\mathfrak{D}(\zeta_0:\zeta_1:\zeta_2)=(\zeta_1^5-\zeta_2^5)(-1024 \zeta_0^{10}+3840 \zeta_0^8 \zeta_1 \zeta_2 -3840 \zeta_0^6 \zeta_1^2 \zeta_2^2\\
\quad\quad\quad\quad\quad\quad\quad \quad\quad\quad\quad\quad\quad\quad\quad\quad\quad\quad+1200 \zeta_0^4 \zeta_1^3 \zeta_2^3 -100\zeta_0^2 \zeta_1^4 \zeta_2^4 +\zeta_1^5 \zeta_2^5)\\
\quad\quad\quad\quad\quad\quad\quad \quad\quad\quad\quad\quad+\zeta_0(\zeta_1^{10}-\zeta_2^{10})(352\zeta_0^4 -160 \zeta_0^2 \zeta_1 \zeta_2 +10 \zeta_1^2 \zeta_2^2)+(\zeta_1^{15}-\zeta_2^{15}).
\end{cases} 
\end{align*}
We have the following relation:
$$
144\mathfrak{D}^2=-1728\mathfrak{B}^5 +720\mathfrak{A}\mathfrak{C}\mathfrak{B}^3 -80 \mathfrak{A}^2 \mathfrak{C}^2 \mathfrak{B} +64\mathfrak{A}^3(5\mathfrak{B}^2-\mathfrak{A}\mathfrak{C})^2+\mathfrak{C}^3.
$$

 Kobayashi,  Kushibiki and  Naruki \cite{KobaNaru} showed that  
a compactification $\overline{(\mathbb{H}\times\mathbb{H})/\langle PSL(2,\mathcal{O}), \tau\rangle}$ is isomorphic to $\mathbb{P}^2(\mathbb{C})$.
Let  
$$
\varphi : \mathbb{P}^2(\mathbb{C})=\overline{(\mathbb{H}\times\mathbb{H})/\langle\Gamma,\tau \rangle}\rightarrow
\overline{(\mathbb{H}\times\mathbb{H})/\langle PSL(2,\mathcal{O}),\tau \rangle}=\mathbb{P}^2(\mathbb{C})
$$
be a rational mapping defined by
$$
(\zeta_0:\zeta_1:\zeta_2)\mapsto(\mathfrak{A}^5:\mathfrak{A}^2 \mathfrak{B}: \mathfrak{C}).
$$
$\varphi $ is  a holomorphic mapping of $\mathbb{P}^2(\mathbb{C})-\{A=0\}$ to $\mathbb{P}^2(\mathbb{C})-(\text{a line at infinity }\hspace{0.1cm} L_\infty)
\subset  \overline{(\mathbb{H}\times\mathbb{H})/ \langle PSL(2,\mathcal{O}),\tau \rangle}.$
Set
\begin{eqnarray}\label{(x,y)}
x=\frac{\mathfrak{B}}{\mathfrak{A}^3},   \quad y=\frac{\mathfrak{C}}{\mathfrak{A}^5}.
\end{eqnarray}
$x$ and $y$ are the affine coordinates identifying $(1:x:y)\in \mathbb{P}^2(\mathbb{C})-L_\infty$ with $(x,y) \in \mathbb{C}^2$.

\begin{prop} {\rm (\cite{KobaNaru})}\label{orbi}
The branch locus of the orbifold $(\mathbb{H}\times\mathbb{H})/\langle PSL(2,\mathcal{O}),\tau \rangle$ in $\mathbb{P}^2(\mathbb{C}) - L_{\infty}=\mathbb{C}^2$ is, using the affine coordinates {\rm (\ref{(x,y)})},
$$
D=y(1728 x^5 -720 x^3 y + 80 x y^2 - 64 (5x^2 -y)^2 -y^3)=0
$$
of index $2$.
The orbifold structure on $\overline{(\mathbb{H}\times\mathbb{H})/\langle PSL(2,\mathcal{O}),\tau \rangle}$ is given by $(\mathbb{P}^2(\mathbb{C}), 2 D+\infty L_\infty)$.
\end{prop}

We note that $\mathbb{H}\times\mathbb{H}$ is embedded in $\mathbb{P}^1(\mathbb{C})\times\mathbb{P}^1(\mathbb{C})$ which is isomorphic to a non-degenerate quadric surface in $ \mathbb{P}^3(\mathbb{C})$.
Let $\pi : \mathbb{H}\times \mathbb{H}\rightarrow (\mathbb{H}\times\mathbb{H})/\langle PSL(2,\mathcal{O}),\tau \rangle$ be the canonical projection. The multivalued inverse mapping $\pi^{-1}$ is called  the developing map of the orbifold $(\mathbb{H}\times\mathbb{H})/\langle PSL(2,\mathcal{O}),\tau \rangle$.

\begin{df}\label{df:UDE}
Let us  consider a system of linear differential equations on the orbifold $(\mathbb{H}\times\mathbb{H})/\langle PSL(2,\mathcal{O}),\tau \rangle$ with $4$-dimensional  space of solutions. Let   $z_0,z_1,z_2,z_3$ be linearly independent solutions of the system. 
If 
$$
(\mathbb{H}\times\mathbb{H})/\langle PSL(2,\mathcal{O}),\tau \rangle\rightarrow \mathbb{P}^3 (\mathbb{C}):p\mapsto (z_0(p):z_1(p):z_2(p):z_3(p))
$$
gives the developing map of the orbifold $(\mathbb{H}\times\mathbb{H})/\langle PSL(2,\mathcal{O}),\tau \rangle$, we call this system the uniformizing differential equation of the orbifold.
\end{df}

From Proposition \ref{orbi}, T. Sato obtained the following result.

\begin{thm}\label{prop:Sato}{\rm (\cite{Sato} Example. 4)}
The holomorphic  conformal structure of the  uniformizing differential equation of the orbifold $(\mathbb{H}\times\mathbb{H})/\langle PSL(2,\mathcal{O}), \tau\rangle$ 
is
\begin{align}\label{Sato's}
\displaystyle \frac{-20 (4 x^2+3 x y -4y)}{36 x^2 -32 x -y}(dx)^2
+2(dx)(dy)
+\displaystyle \frac{-2(54 x^3 -50 x^2 -3 x y +2y )}{5 y (36 x^2 -32 x -y)}(dy)^2,
\end{align}
where $(x,y)$ is the affine coordinates in {\rm (\ref{(x,y)}) }.
\end{thm}

Let 
\begin{eqnarray}
\begin{cases}\label{HUDE}
& z_{xx}= l z_{xy} + a z_x + b z_y +p z,\\
& z_{yy}=m z_{xy} + c z_x +d z_y +q z
\end{cases}
\end{eqnarray}
be the uniformizing differential equation of $(\mathbb{H}\times\mathbb{H})/\langle PSL(2,\mathcal{O}), \tau\rangle$,  where $(x,y)$ is the affine coordinates in {\rm (\ref{(x,y)})}.
We already  obtained the coefficients $l$ and $m$ (see Definition \ref{df:conformal} and Theorem \ref{prop:Sato}).
If the normalization factor of (\ref{HUDE}) is given, the coefficients $a, b, c$ and $d$ are determined by Proposition \ref{prop:ABCD}.
The other coefficients $p$ and $q$ are determined by the integrability  condition of (\ref{HUDE}).

\begin{rem}\label{rem:Sato}
Sato {\rm \cite{Sato}} determined  the uniformizing differential equation of $(\mathbb{H}\times\mathbb{H})/\langle PSL(2,\mathcal{O}), \tau\rangle$ 
\begin{eqnarray*}
\begin{cases}
& z_{xx}= l z_{xy} + a_s z_x + b_s z_y +p_s z,\\
& z_{yy}=m z_{xy} + c _s z_x +d_s z_y +q_s z
\end{cases}
\end{eqnarray*}
with
\begin{align*}
\begin{cases}
\vspace*{0.2cm}
a_s(x,y)=\displaystyle \frac{-20(3x-2)}{36x^2 -32x -y},\quad
b_s(x,y)=\displaystyle \frac{-10(8x+3y)}{36x^2 -32x -y},\\
\vspace*{0.2cm}
c_s(x,y)=\displaystyle \frac{3x-2}{5y(36x^2 -32x -y)},\quad
d_s(x,y)=\displaystyle \frac{-198x^2 +180x+7y}{5y(36x^2 -32x -y)},\\
\vspace*{0.2cm}
p_s(x,y)=\displaystyle \frac{-3}{(36x^2 -32x -y)},\quad
q_s(x,y)=\displaystyle \frac{3}{100y(36x^2 -32 x -y)}.
\end{cases}
\end{align*}
Here, the normalization factor 
\begin{align}\label{Satonf}
e^{2\theta} =\frac{-36 x^2 +32 x +y}{y^{1/2} (1728 x^5 -720 x^3 y +80 x y^2 -64 (5x^2 -y)^2 -y^3))^{3/2}}.
\end{align}
 exactly corresponds to the above data $a_s,b_s,c_s,d_s,p_s$ and $q_s$.
It should coincides with  the original normalization factor
 in {\rm \cite{Sato}} p.185,
because Sato used the above data.
However, it is not the case.
We suppose there would be  contained some typos in the original one.
\end{rem}

\subsection{Exact relation between period differential equation and  unifomizing differential equation}

The modular isomorphism (\ref{modulariso}) implies that  our period differential equation (\ref{eq:systemP4}) should be  related to  the uniformizing differential equation of the orbifold $(\mathbb{H}\times\mathbb{H})/\langle PSL(2,\mathcal{O}),\tau \rangle$.
In this subsection, we show that the holomorphic conformal structure of (\ref{eq:systemP4}) is transformed to (\ref{Sato's}) in Theorem \ref{prop:Sato} by an explicit birational transformation. 
Moreover, we determine a normalization factor which is different from that of Sato's (\ref{Satonf}).
The uniformizing differential equation of the orbifold $(\mathbb{H}\times\mathbb{H})/\langle PSL(2,\mathcal{O}),\tau \rangle$ with our normalizing factor corresponds  to the period differential equation  (\ref{eq:systemP4}).

\begin{prop}
The period differential equation {\rm (\ref{eq:systemP4})}  is represented in the form
\begin{eqnarray}
\begin{cases}\label{ours}
& z_{\lambda \lambda}= l _0 z_{\lambda \mu} + a_0 z_\lambda + b_0  z_\mu +p_0 z,\\
& z_{\mu \mu}=m _0 z_{\lambda \mu} + c _0 z_\lambda +d_0  z_\mu +q_0 z
\end{cases}
\end{eqnarray}
with
\begin{align*}
\begin{cases}
\vspace*{0.2cm}
l_0=\displaystyle \frac{2\mu (-1 + 15\lambda + 100\lambda ^2)}{\lambda + 16 \lambda^2 -80 \lambda ^3 +125 \mu},\quad
m_0=\displaystyle \frac{2( \lambda ^2 - 8\lambda ^3 + 16 \lambda ^4 + 5 \mu - 50 \lambda \mu)}{\mu(\lambda + 16 \lambda^2 -80 \lambda ^3 +125 \mu )},\\
\vspace*{0.2cm}
a_0=\displaystyle \frac{(-1+10\lambda)(1+20\lambda)}{\lambda + 16 \lambda^2 -80 \lambda ^3 +125 \mu},\quad
b_0=\displaystyle \frac{5\mu (3+40 \lambda)}{\lambda + 16 \lambda^2 -80 \lambda ^3 +125 \mu},\\
\vspace*{0.2cm}
c_0=-\displaystyle\frac{5(-1+10\lambda)}{\mu(\lambda + 16 \lambda^2 -80 \lambda ^3 +125 \mu)},\quad
d_0=\displaystyle\frac{-\lambda -20 \lambda ^2 +96 \lambda ^3 -200\mu}{\mu(\lambda + 16 \lambda^2 -80 \lambda ^3 +125 \mu)},\\
p_0=\displaystyle\frac{2(1+20\lambda)}{\lambda + 16 \lambda^2 -80 \lambda ^3 +125 \mu},\quad
q_0=-\displaystyle\frac{10}{\mu(\lambda + 16 \lambda^2 -80 \lambda ^3 +125 \mu)}.
\end{cases}
\end{align*}
\end{prop}

\begin{proof}
Straightforward calculation.
\end{proof}

Especially, the holomorphic conformal structure of the period differential equation (\ref{eq:systemP4}) is 
\begin{align} \label{PerHCS}
\displaystyle \frac{2\mu (-1 + 15\lambda + 100\lambda ^2)}{\lambda + 16 \lambda^2 -80 \lambda ^3 +125 \mu} (d\lambda)^2 + 2 (d\lambda)(d\mu) + \displaystyle \frac{2( \lambda ^2 - 8\lambda ^3 + 16 \lambda ^4 + 5 \mu - 50 \lambda \mu)}{\mu(\lambda + 16 \lambda^2 -80 \lambda ^3 +125 \mu )}(d\mu)^2.\end{align}

\begin{thm}
Set  a birational transformation
\begin{eqnarray}\label{trans}
f:(\lambda,\mu)\mapsto (x,y)=\Big(\frac{25 \mu}{2 (\lambda -1/4)^3}, -\frac{3125 \mu ^2}{(\lambda -1/4)^5}\Big)
\end{eqnarray}
 from $(\lambda,\mu)$-space to $(x,y)$-space.
The holomorphic conformal structure  {\rm (\ref{PerHCS})} is transformed to the holomorphic conformal structure {\rm (\ref{Sato's})} by $f$. 
\end{thm}

\begin{proof}
The inverse $f^{-1}$ is given by
\begin{eqnarray}\label{invtrans}
\lambda ( x , y )= \frac{1}{4}-\frac{y}{20 x^2},\enspace\quad \mu (x,y )= -\frac{y^3}{10^5 x^5}.
\end{eqnarray}
We have
\begin{align}
\begin{cases}\label{L0M0cmpl}
\vspace*{0.2cm}
l_0(\lambda(x,y) ,\mu(x,y) )=\displaystyle \frac{-y^2(4 x^2 -y)(9 x^2 -y )}{250x^3 (240 x^4 -88x^2 y + 8 y^2 -x y^2)},\\
\vspace*{0.2cm}
m_0(\lambda(x,y) ,\mu(x,y) )=\displaystyle \frac{-4000 x^3 (100 x^4 -40 x^2 y + 3 x^3 y + 4 y^2 -x y^2)}{y^2 (240 x^4 -88x^2 y + 8 y^2 -x y^2)}.
\end{cases}
\end{align}
By (\ref{trans}) and (\ref{invtrans}), we have
\begin{align}
\label{biratiofirstdiff}
\vspace*{0.2cm}
x_\lambda =\displaystyle \frac{60 x^3}{y}, \quad  y_\lambda = 100 x^2,\quad
x_\mu =-\displaystyle\frac{10^5 x^6}{y^3},\quad y_\mu =-\displaystyle\frac{2\cdot 10^5 x^5}{y^2}.
\end{align}
From (\ref{L0M0cmpl}) and (\ref{biratiofirstdiff}) and Proposition \ref{LM}, by the birational transformation $f:(\lambda,\mu)\mapsto (x,y)$,  the coefficients $l_0$ and $m_0$ are transformed to
\begin{align*}
\overline{l_0}=\displaystyle \frac{-20 (4 x^2+3 x y -4y)}{36 x^2 -32 x -y},\quad
\overline{m_0}=\displaystyle \frac{-2(54 x^3 -50 x^2 -3 x y +2y )}{5 y (36 x^2 -32 x -y)}.
\end{align*}
These are equal to the coefficients of the holomorphic conformal structure (\ref{Sato's}).
Therefore, the holomorphic conformal structure (\ref{PerHCS})  is transformed to (\ref{Sato's}).
\end{proof}

\begin{rem}
The birational transformation {\rm (\ref{trans})} is obtained as the composition of certain birational transformations.
First, blow up at $(\lambda, \mu )=(1/4,0)\in \text{($(\lambda,\mu)$-space)}$
 three times:  $(\lambda,\mu)\mapsto(\lambda,u_1)=\Big(\lambda,\displaystyle \frac{\mu}{\lambda -1/4}\Big), \hspace{0.3cm} (\lambda,u_1)\mapsto(\lambda,u_2)=\Big(\lambda,\displaystyle \frac{u_1}{\lambda-1/4}\Big), \hspace{0.3cm} (\lambda,u_2)\mapsto(\lambda,u_3)=\Big(\lambda,\displaystyle\frac{u_2}{\lambda-1/4}\Big)$. 
 Cancel $\lambda$ by $\lambda=\displaystyle\frac{u_2}{u_3}+\displaystyle \frac{1}{4}$.
Then, we have the following birational transformation:
$$
\psi_0 :(\lambda,\mu)\mapsto(u_2,u_3)=\Big(\frac{\mu}{(\lambda-1/4)^2},\frac{\mu}{(\lambda-1/4)^3}\Big).
$$
 (Its inverse is given by
 $$
 \psi_0^{-1}:(u_2,u_3)\mapsto(\lambda,\mu)=\Big(\frac{u_2}{u_3}+\frac{1}{4},\frac{u_2^3}{u_3^2}\Big).)
 $$
On the other hand, blow up at  $(x,y)=(0,0)\in (\text{$(x,y)$-space})$:
$$
\psi_1:(x,y)\mapsto(x,s)=\Big(x,\frac{y}{x}\Big).
$$
(Its inverse is  given by
$$
\psi_1^{-1}: (x,s)\mapsto(x,y)=(x,xs).)
$$
Moreover, we define  the holomorphic mapping
$$
\chi :(u_2,u_3)\mapsto(x,s)=\Big(\frac{25}{2}u_3,-250u_2\Big).
$$
We have  $f=\psi_1^{-1}\circ \chi \circ \psi_0$.
\end{rem}

Instead the normalization factor (\ref{Satonf}) used by Sato, that is referred in Remark \ref{rem:Sato},
we need a new normalization factor (\ref{normal}).
Together with the conformal structure coming from $(l_1,m_1)=(l,m)$,
we obtain the new uniformizing differential equation which we are looking for.

\begin{prop}
The uniformizing differential equation of the orbifold $(\mathbb{H}\times\mathbb{H})/\langle PSL(2,\mathcal{O}), \tau\rangle$ 
with the normalization factor
\begin{eqnarray}\label{normal}
e^{2\theta}=\frac{x^4 (-36 x^2 +32 x +y)}{y^{5/2} (1728 x^5 -720 x^3 y +80 x y^2 -64 (5x^2 -y)^2 -y^3)^{3/2}}
\end{eqnarray}
is
\begin{eqnarray}
\begin{cases}\label{periodUDE}
& z_{xx}= l _1z_{xy} + a_1 z_x + b_1 z_y +p_1 z,\\
& z_{yy}=m _1z_{xy} + c _1z_x +d_1 z_y +q_1 z
\end{cases}
\end{eqnarray}
with
\begin{align*}
\begin{cases}
\vspace*{0.2cm}
l_1=\displaystyle \frac{-20 (4 x^2+3 x y -4y)}{36 x^2 -32 x -y},\quad
m_1=\displaystyle \frac{-2(54 x^3 -50 x^2 -3 x y +2y )}{5 y (36 x^2 -32 x -y)},\\
\vspace*{0.2cm}
a_1=\displaystyle \frac{-2(20 x^3 -8 x y   + 9x^2 y + y^2)}{x y (36x^2 -32 x - y)},\quad
b_1=\displaystyle \frac{10y(-8 +3x)}{x(36x^2 -32 x -y)},\\
\vspace*{0.2cm}
c_1=\displaystyle \frac{-2(-25x^2 + 27 x^3 + 2y -3 x y )}{5y^2(36x^2 -32 x -y)},\quad
d_1=\displaystyle \frac{-2(-120 x^2 +135 x^3 -2y - 3x y)}{5xy( 36x^2 -32x -y)}\\
\vspace*{0.2cm}
p_1=\displaystyle \frac{-2(8x-y)}{x^2 (36x^2 -32 x -y)},\quad
q_1=\displaystyle \frac{-2(-10+9x)}{25x y (36x^2 - 32 x -y)}.
\end{cases}
\end{align*}
\end{prop}

\begin{proof}
$l_1$ and $m_1$ are given in Theorem \ref{prop:Sato}. 
According to Proposition \ref{prop:ABCD}, the other coefficients are determined by $l_1, m_1$ and $\theta$ in (\ref{normal}).
\end{proof}

\begin{thm} \label{periodUDEThm}
By the birational transformation $f$ in {\rm (\ref{trans})},
our period differential equation {\rm (\ref{ours})} is transformed to the uniformizing differential  equation {\rm (\ref{periodUDE})} of the orbifold $(\mathbb{H}\times\mathbb{H})/\langle PSL(2,\mathcal{O}),\tau \rangle$.  
\end{thm}

\begin{proof}
We have 
\begin{align}
\begin{cases}\label{theotherscmpl}
\vspace*{0.2cm}
a_0(\lambda(x,y) ,\mu(x,y) )=\displaystyle \frac{400 x^2 (3 x^2 -y) (6x^2 -y)}{y(240 x^4 -88 x^2 y +8 y^2 -x y^2)},\\
\vspace*{0.2cm}
b_0(\lambda(x,y) ,\mu(x,y) )=\displaystyle \frac{-y^2 (13 x^2 - 2 y)}{25 x (240 x^4 -88 x^2 y+8 y^2 - x y^2)},\\
\vspace*{0.2cm}
c_0(\lambda(x,y) ,\mu(x,y) )=\displaystyle \frac{2 \cdot10^8 x^9 (3 x^2 -y)}{y^4(240 x^4 -88 x^2 y+8 y^2 - x y^2),}\\
\vspace*{0.2cm}
d_0(\lambda(x,y) ,\mu(x,y) )=\displaystyle \frac{160000 x^5 (175 x^4 - 65 x^2 y + 6 y^2 - x y^2)}{y^3 (240 x^4 -88 x^2 y+8 y^2 - x y^2)},\\
\vspace*{0.2cm}
p_0(\lambda(x,y) ,\mu(x,y) )=\displaystyle \frac{1600 x^4 (6 x^2 -y)}{y (240 x^4 -88 x^2 y+8 y^2 - x y^2)},\\
\vspace*{0.2cm}
q_0(\lambda(x,y) ,\mu(x,y) )=\displaystyle \frac{8\cdot 10^{8} x^{11}}{y^4 (240 x^4 -88 x^2 y+8 y^2 - x y^2)}.
\end{cases}
\end{align}
By (\ref{trans}) and (\ref{invtrans}),
 we have
\begin{align}
\label{biratioseconddiff}
\begin{cases}
\vspace*{0.2cm}
x_{\lambda\lambda}=\displaystyle \frac{4800 x^5}{y^2},\quad y_{\lambda\lambda}=\frac{12000 x^4}{y},\quad
x_{\mu\mu}=0,\\
 y_{\mu\mu}=\displaystyle \frac{2\cdot 10^{10} x^{10}}{y^5},\quad
x_{\lambda\mu}=\displaystyle\frac{- 6\cdot 10^6 x^8}{y^4},\quad y_{\lambda \mu}=\frac{-2\cdot 10^7 x^7}{y^3}.
\end{cases}
\end{align}
From (\ref{L0M0cmpl}), (\ref{biratiofirstdiff}), (\ref{theotherscmpl}) and (\ref{biratioseconddiff}) and Proposition \ref{LM}, by the birational transformation $f:(\lambda,\mu)\mapsto (x,y)$,
 the coefficients $a_0,b_0,c_0,d_0,p_0$ and $q_0$ are transformed to
\begin{align*}
\begin{cases}
\vspace*{0.2cm}
\overline{a_0}=\displaystyle \frac{-2(20 x^3 -8 x y   + 9x^2 y + y^2)}{x y (36x^2 -32 x - y)},\quad
\overline{b_0}=\displaystyle \frac{10y(-8 +3x)}{x(36x^2 -32 x -y)},\\
\vspace*{0.2cm}
\overline{c_0}=\displaystyle \frac{-2(-25x^2 + 27 x^3 + 2y -3 x y )}{5y^2(36x^2 -32 x -y)},\quad
\overline{d_0}=\displaystyle \frac{-2(-120 x^2 +135 x^3 -2y - 3x y)}{5xy( 36x^2 -32x -y)},\\
\vspace*{0.2cm}
\overline{p_0}=\displaystyle \frac{-2(8x-y)}{x^2 (36x^2 -32 x -y)},\quad
\overline{q_0}=\displaystyle \frac{-2(-10+9x)}{25x y (36x^2 - 32 x -y)}.
\end{cases}
\end{align*}
These are equal to the coefficients of (\ref{periodUDE}).
\end{proof}

Therefore, the uniformizing differential equation of the orbifold $(\mathbb{H}\times\mathbb{H})/\langle PSL(2,\mathcal{O}),\tau \rangle$ with the normalization factor {\rm (\ref{normal})} 
is
connected to our family $\mathcal{F}_0=\{S_0(\lambda,\mu)\}$ of $K3$ surfaces.

\section*{Acknowledgment}
The author would like to express his gratitude to Professor  Hironori Shiga for various advices and suggestions. He is  grateful to Professor Kimio Ueno for his comments and encouragement.  
He is thankful to Professor Takeshi Sasaki for the suggestion about  uniformizing differential equations.
He is also indebted   to Toshimasa Ishige  for his useful comments  based on  his pioneering and eager research. 
The author also would like to express his gratitude to Professor Kazushi Ueda for information about these errors.

\begin{center}
\hspace{7.7cm}\textit{Atsuhira  Nagano}\\
\hspace{7.7cm}\textit{ c.o. Prof. Kimio Ueno, Department of Mathematics}\\
\hspace{7.7cm}\textit{ Waseda University}\\
\hspace{7.7cm}\textit{Okubo 3-4-1, Shinjuku-ku, Tokyo, 169-8555}\\
\hspace{7.7cm}\textit{Japan}\\
 \hspace{7.7cm}\textit{(E-mail: atsuhira.nagano@gmail.com)}
  \end{center}

\end{document}